\RequirePackage{fix-cm}
\documentclass[smallextended]{svjour3}
\usepackage{caption}
\usepackage{graphicx}
\usepackage{booktabs}
\usepackage{multirow}
\usepackage{amsmath,amssymb,amsfonts}
\usepackage{mathrsfs}
\usepackage[title]{appendix}
\usepackage{xcolor}
\usepackage{manyfoot}
\usepackage{booktabs}
\usepackage{algorithm}
\usepackage{algorithmicx}
\usepackage{algpseudocode}
\usepackage{listings}
\usepackage[justification=centering]{caption}
\usepackage{hyperref}
\hypersetup{hidelinks,hypertexnames=false}
\usepackage{geometry}
\usepackage{nicefrac}
\usepackage[final]{showlabels}
\usepackage{float}
\usepackage{enumitem}
\usepackage{pifont}

\newcounter{appthm}[section]
\newcounter{applemma}[section]
\newcounter{appprop}[section]
\newcounter{appdef}[section]
\renewcommand{\theappthm}{\thesection.\arabic{appthm}}
\renewcommand{\theapplemma}{\thesection.\arabic{applemma}}
\renewcommand{\theappprop}{\thesection.\arabic{appprop}}
\renewcommand{\theappdef}{\thesection.\arabic{appdef}}
\newenvironment{appthm}[1][]{
  \refstepcounter{appthm}
  \par\noindent\textbf{Theorem~\theappthm\ifx&#1&\else\ (#1)\fi.}
  \itshape
}{\par}
\newenvironment{applemma}[1][]{
  \refstepcounter{applemma}
  \par\noindent\textbf{Lemma~\theapplemma\ifx&#1&\else\ (#1)\fi.}
  \itshape
}{\par}
\newenvironment{appprop}[1][]{
  \refstepcounter{appprop}
  \par\noindent\textbf{Proposition~\theappprop\ifx&#1&\else\ (#1)\fi.}
  \itshape
}{\par}
\newenvironment{appdef}[1][]{
  \refstepcounter{appdef}
  \par\noindent\textbf{Definition~\theappdef\ifx&#1&\else\ (#1)\fi.}
  \upshape
}{\par}
\usepackage{hyperref}
\makeatletter
\newcommand{\restatedthmhead}[3]{
  \par\addvspace{\topsep}\noindent
  \textbf{#1}
  \ifx\@empty#2\@empty\else\ #2\fi
  \ifx\@empty#3\@empty\else\ (#3)\fi
  .\
}
\makeatletter
\newcommand{\restatedenv}[3]{
  \par\addvspace{\topsep}
  \noindent\textbf{#1~#2 (restated).}
  #3
}
\newenvironment{theoremrep}[1]
  {\restatedenv{Theorem}{\ref{#1}}{\itshape}}
  {\par\addvspace{\topsep}}
\newenvironment{lemmarep}[1]
  {\restatedenv{Lemma}{\ref{#1}}{\itshape}}
  {\par\addvspace{\topsep}}
\newenvironment{propositionrep}[1]
  {\restatedenv{Proposition}{\ref{#1}}{\itshape}}
  {\par\addvspace{\topsep}}

\makeatother
\usepackage{float}
\usepackage{soul}
\newcommand{\la}{\langle}
\newcommand{\ra}{\rangle}

\newcommand{\Rbb}{\mathbb{R}}
\newcommand{\Sbb}{\mathbb{S}}

\newcommand{\Xbb}{\mathbb{X}}
\newcommand{\Ybb}{\mathbb{Y}}

\newcommand{\Ncal}{{\cal N}}
\newcommand{\Mcal}{{\cal M}}
\newcommand{\Scal}{{\cal S}}
\newcommand{\Mpq}{{\cal M}_{p,q}}
\newcommand{\tMcal}{\widetilde{\Mcal}}

\newcommand{\tMpq}{\widetilde{\Mcal}_{p,q}}
\newcommand{\Mpo}{{\cal M}_{p,0}}
\newcommand{\tOmega}{\widetilde{\Omega}}

\newcommand{\Bcal}{{\cal B}}

\newcommand{\Ucal}{{\cal U}}

\newcommand{\Tcal}{{\cal T}}
\newcommand{\Lcal}{{\cal L}}
\newcommand{\Dcal}{{\cal D}}
\newcommand{\Ecal}{{\cal E}}

\newcommand{\Wcal}{{\cal W}}
\newcommand{\Xcal}{{\cal X}}

\newcommand{\mubar}{\overline{\mu}}
\newcommand{\Pbar}{\overline{P}}
\newcommand{\zbar}{\overline{z}}
\newcommand{\xbar}{\overline{x}}
\newcommand{\ybar}{\overline{y}}

\newcommand{\zhatk}{{\widehat{z}^k}}
\newcommand{\vhatk}{{\widehat{v}^k}}

\newcommand{\vLM}{v^{\text{LM}}}

\newcommand{\aff}{{\rm aff}}
\newcommand{\lin}{{\rm lin}}
\newcommand{\rank}{\rm{rank}}
\newcommand{\dd}{d}
\newcommand{\app}{{\rm app}}
\newcommand{\appl}{{\rm appl}}
\newcommand{\im}{{\rm Im}}

\newcommand{\ep}{ $\hfill \Box$}

\newcommand{\mN}{\mathcal{N}}

\newcommand{\mQ}{\mathcal{Q}}
\newcommand{\mT}{\mathcal{T}}
\newcommand{\mU}{\mathcal{U}}
\newcommand{\mV}{\mathcal{V}}

\newcommand{\dist}{{\rm dist}}

\newcommand{\nn}{{\nonumber}}

\newcommand{\argmin}{{\rm argmin}}

\newcommand{\SLMN}{\texttt{SLMN}}

\usepackage{tikz}
\usetikzlibrary{positioning, shapes, arrows.meta}
\begin{document}
	\title{Stratification for Nonlinear Semidefinite Programming}
	\author{
	Chenglong Bao \thanks{Authors are listed alphabetically by surname.}\and
    Chao Ding \and
	Fuxiaoyue Feng \and
    Jingyu Li
	}
	\institute{
        C.L. Bao \at
        Yau Mathematical Sciences Center, Tsinghua University, Beijing 100084, P.R. China; Yanqi Lake Beijing Institute of Mathematical Sciences and Applications, Beijing 101408, P.R. China. The research of this author is supported in part by National Natural Science Foundation of China (12271291) and New Cornerstone Investigator Program (100001127). \\
        \email{clbao@tsinghua.edu.cn}
        \and
		C. Ding \at
		State Key Laboratory of Mathematical Sciences, Academy of Mathematics and Systems Science, Chinese Academy of Sciences, Beijing 100190, P.R. China; School of Mathematical Sciences, University of Chinese Academy of Sciences, Beijing 100049, P.R. China; Institute of Applied Mathematics, Academy of Mathematics and Systems Science, Chinese Academy of Sciences, Beijing 100190, P.R. China. The research of this author is supported in part by National Key R\&D Program of China (No. 2021YFA1000300, No. 2021YFA1000301), National Natural Science Foundation of China (No. 12531014), and CAS Project for Young Scientists in Basic Research (No. YSBR-034).   \\
		\email{dingchao@amss.ac.cn}
		\and
        F.X.Y. Feng\at
		Institute of Applied Mathematics, Academy of Mathematics and Systems Science, Chinese Academy of Sciences, Beijing,  P.R. China;  School of Mathematical Sciences, University of Chinese Academy of Science, Beijing 100049,  P.R. China. \\
		\email{fengfuxiaoyue@amss.ac.cn}
        \and
        J.Y. Li\at
        Qiuzhen College, Tsinghua University, Beijing 100084, P.R. China. \\
        \email{lijingyu23@mails.tsinghua.edu.cn}
	}
	\date{This version: September 22, 2026}
	\maketitle
	\begin{abstract}

        This paper introduces a stratification framework for nonlinear semidefinite programming (NLSDP) that reveals and utilizes the geometry behind the nonsmooth KKT system. Based on the index stratification of $\mathbb{S}^n$ and its lift to the primal-dual space, a stratified variational analysis is developed. Specifically, we define the stratum-restricted regularity property, characterize it by the verifiable weak second order condition (W-SOC) and weak strict Robinson constraint qualification (W-SRCQ), and interpret the W-SRCQ geometrically via transversality, with stability along strata. The interactions of these properties across neighboring strata are further examined, leading to the conclusion that classical strong-form regularity conditions correspond to the local uniform validity of stratum-restricted counterparts. On the algorithmic side, a stratified Gauss--Newton method with normal steps and a correction mechanism is proposed for globally solving the KKT equation through a least-squares merit function. We demonstrate that the algorithm converges globally to directional stationary points. Moreover, under the second order sufficient condition (SOSC) and the strict Robinson constraint qualification (SRCQ) at an accumulation point, with a suitable correction threshold, the whole sequence converges superlinearly to this point, which is a KKT pair, and eventually identifies the active stratum. The rate is quadratic if the problem data are additionally of class $LC^2$ near the solution.
		\vskip 10 true pt
		\noindent {\bf Key words:} nonlinear semidefinite programming; stratification; variational analysis; perturbation analysis; strong metric regularity;  Gauss--Newton
		\vskip 10 true pt
		\noindent {\bf }
	\end{abstract}

\section{Introduction}\label{sec:intro-2}
Consider the nonlinear semidefinite programming (NLSDP) problem
\begin{equation}\label{prog:SDP}
\begin{aligned}
\min_{x\in\Xbb} \quad & f(x) \\
\text{s.t.} \quad & g(x) \in \Sbb_{+}^n ,
\end{aligned}
\end{equation}
where $\mathbb{X}$ is a Euclidean space, the objective function $f:\mathbb{X}\to\mathbb{R}$ and the constraint mapping $g:\mathbb{X}\to\mathbb{S}^n$ are twice continuously differentiable, and $\Sbb_{+}^n$ denotes the closed convex cone of symmetric positive semidefinite matrices.
The NLSDP~\eqref{prog:SDP} naturally includes both classical nonlinear programming and linear semidefinite programming as special cases.
Representative applications include certified trajectory optimization in robotics~\cite{kang2025global}, network design and control in transportation systems~\cite{coogan2017offset,ouyang2018conic}, and structural optimization and material design~\cite[Chapter~10]{stinglsolution}; further examples arise in control and systems theory~\cite{lee2016sequential} and robust optimization~\cite{duarte2016model}.
From a theoretical standpoint, NLSDP serves as a central model for modern variational analysis, perturbation theory, and numerical optimization; see, e.g.,~\cite{bonnans2013perturbation} and the references therein.

A primal-dual pair $z=(x,y)\in\mathbb{X}\times\mathbb{S}^n$ is a Karush-Kuhn-Tucker (KKT) pair for~\eqref{prog:SDP} if it solves the nonsmooth equation
\begin{equation}\label{eq:kkt-mapping}
F(z):=
\begin{bmatrix}
\nabla f(x)+\nabla g(x)y\\[2pt]
-g(x)+\Pi_{\mathbb{S}^n_{+}}\big(G(z)\big)
\end{bmatrix}
=0,
\end{equation}
where $\Pi_{\mathbb{S}^n_{+}}:\mathbb{S}^n\to\mathbb{S}^n$ denotes the metric projection onto the positive semidefinite cone $\mathbb{S}^n_{+}$, the symbol $\nabla$ denotes the gradient (i.e., the adjoint of the differential), and $G(z):=g(x)+y$. The inherent nonsmoothness of the KKT system~\eqref{eq:kkt-mapping} stems exclusively from the projection operator $\Pi_{\mathbb{S}^n_{+}}$. Although globally Lipschitz continuous and directionally differentiable, this operator fails to be Fréchet differentiable at any matrix possessing zero eigenvalues. Consequently, even if the underlying problem functions $f$ and $g$ are entirely smooth, the composite KKT mapping $F$ remains nonsmooth, thereby precluding the direct application of classical smooth Newton methods.

Semismooth Newton-type methods~\cite{qi1993nonsmooth,qi1993convergence} constitute a standard approach for solving the nonsmooth KKT system~\eqref{eq:kkt-mapping}.
Their local superlinear or quadratic convergence is typically established under \emph{strong regularity}~\cite[Section~2]{robinson1980strongly}.
For the NLSDP~\eqref{prog:SDP}, Sun~\cite{sun2006strong} proved that, under the Robinson constraint qualification (RCQ, Definition~\ref{def:RCQ}), strong regularity at a local minimizer is equivalent to constraint nondegeneracy (Definition~\ref{def:CN-nonKKT}) together with the strong second order sufficient condition (S-SOSC, Definition~\ref{def:S-SOSC-nonKKT}).
These strong-form assumptions, however, are frequently violated in degenerate regimes.
More recently, Feng et al.~\cite{feng2025quadratically} relaxed the classical requirements for superlinear or quadratic local convergence by introducing the weak strict Robinson constraint qualification (W-SRCQ, Definition~\ref{def:W-SRCQ-nonKKT}) and the weak second order condition (W-SOC, Definition~\ref{def:W-SOC-nonKKT}), accompanied by a correction mechanism. Despite these advances, two issues remain. First, the relationship between the weak conditions and the classical perturbation/regularity theory of NLSDP has not been fully clarified, and geometric interpretations and stability properties are largely missing. Second, the correction mechanism in~\cite{feng2025quadratically} is local, and a globalization strategy supported by a complete global convergence theory is not yet available. We address both issues in the present work.

In recent years, solving~\eqref{prog:SDP} in regimes where low-rank solutions are expected has attracted considerable attention.
A representative geometric approach is manifold optimization on the fixed-rank manifold
\[
\mathcal{M}_r := \{X \in \mathbb{S}^n_+ \mid \rank(X) = r\},
\]
which is a smooth Riemannian submanifold of $\mathbb{S}^n$. This setting permits intrinsic algorithms such as the Riemannian Newton method, which attains local quadratic convergence under standard assumptions on the retraction and the nonsingularity of the Riemannian Hessian~\cite{absil2008optimization}.
A limitation is that the rank $r$ must be specified in advance; otherwise, one encounters the manifold identification issue~\cite{zhou2016riemannian,gao2022riemannian}. Alternatively, the Burer--Monteiro factorization $X=RR^\top$ with $R\in\mathbb{R}^{n\times r}$ converts the low-rank SDP into a nonconvex problem in Euclidean space~\cite{burer2003nonlinear,burer2005local}. Under rank conditions such as $r(r+1)/2>m$ (with $m$ being the number of affine constraints), the BM formulation enjoys favorable landscape properties for generic data~\cite{boumal2020deterministic}.
On the other hand, choosing $r$ overly large may lead to ill-conditioning and increased computational costs.

In this paper, rather than optimizing a smooth function over a prescribed manifold, we analyze the nonsmooth KKT system by decomposing the ambient space into smooth manifolds on which the relevant mappings are smooth. This viewpoint is related to the theory of \emph{partial smoothness}~\cite{lewis2002active}, where Clarke regularity and normal sharpness yield a tractable local geometry along an active manifold. In the NLSDP setting, normal sharpness is typically absent due to the interaction between nonconvexity and nondifferentiability. To this end, a finer geometric description is required, and it is provided by a \emph{stratification} framework for~\eqref{prog:SDP}.

Stratification is a mathematical technique for decomposing a singular space (e.g., a nonsmooth set or an algebraic variety) into a disjoint union of smooth manifolds, termed strata. Originally developed to study singularities of mappings and spaces, it has become a fundamental tool in algebraic geometry~\cite{whitney1965local,whitney1965tangents,goresky1988stratified}, singularity theory~\cite{arnold2012singularities,mather1970notes}, and nonsmooth optimization~\cite{van1998tame,ioffe2009invitation}. The formal concept emerged in the mid-20th century through the seminal works of Hassler Whitney, Ren\'e Thom and John Mather. Whitney~\cite{whitney1965local,whitney1965tangents} introduced geometric regularity conditions guaranteeing that a stratified space admits a locally trivial structure along each stratum. Thom~\cite{thom1969ensembles} subsequently showed that these conditions provide the precise geometric prerequisites for topological rigidity. Mather~\cite{mather1970notes} later formalized this theory by introducing control data and establishing the stability of stratified mappings. This framework has since been generalized to broader classes of sets, such as definable sets in o-minimal structures, enabling its application to a broad class of functions and constraints in applied mathematics. In particular, stratification plays a central role in tame optimization. While a comprehensive review of this extensive literature is beyond the scope of this paper, we refer interested readers to~\cite{ioffe2009invitation} for a detailed exposition.

In this paper, we do not use stratification as a modeling assumption on the feasible set. Instead, we employ it as an analytical tool for the KKT system. Specifically, when restricted to the set of symmetric matrices with fixed inertia, the metric projection $\Pi_{\mathbb{S}^n_{+}}$ becomes smooth. This property naturally induces the \emph{inertia stratification} of $\mathbb{S}^n$ into smooth manifolds of matrices with prescribed inertia. Through the mapping $G(z)=g(x)+y$, this structure is lifted to a stratification of the primal-dual space $\mathbb{X}\times\mathbb{S}^n$. When restricted to an individual stratum, the KKT mapping becomes smooth and admits an explicit differential. This structural insight is the basis of both
our variational analysis and our algorithmic design.  It enables the design of a globally convergent Newton method that {achieves local superlinear convergence under SOSC and SRCQ, with quadratic convergence under additional $LC^2$ regularity},  and it also provides a new geometric viewpoint
for revisiting the perturbation theory of NLSDP.

A fundamental departure of the present framework from the classical literature lies in the domain where regularity and second-order conditions are evaluated. Traditionally, conditions such as the SOSC, S-SOSC, SRCQ, and constraint nondegeneracy are imposed at a KKT pair and then used to derive local stability and local convergence rates. In contrast, the present framework extends these conditions, together with their weak-form counterparts, to arbitrary primal--dual pairs. This makes it possible to characterize the behavior of manifold differential $dF_z$ even at non-KKT pairs. Moreover, these extended conditions describe the local geometric landscape of the lifted stratum.

This perspective leads to a refined perturbation theory for the NLSDP \eqref{prog:SDP}. Perturbation analysis describes the sensitivity of stationary points and multipliers with respect to data perturbations and provides the basis for error bounds and local convergence rates; see, e.g.,~\cite{dontchev2009implicit,bonnans2013perturbation}.
A central object is the KKT solution mapping together with its regularity properties, such as strong regularity, which yield quantitative Lipschitz stability estimates.
In nonlinear programming, strong regularity admits characterizations in terms of the S-SOSC and the linear independence constraint qualification (LICQ)~\cite{robinson1980strongly,jongen1987inertia,bonnans_pseudopower_1995,dontchev_characterizations_1996}.
For the NLSDP~\eqref{prog:SDP}, Sun~\cite{sun2006strong} established that, under the RCQ, strong regularity at a local minimizer is equivalent to the S-SOSC and constraint nondegeneracy.
These results, while fundamental, treat the KKT system as a single nonsmooth equation and therefore lead naturally to neighborhood conditions on generalized Jacobians or to error-bound-type assumptions that are often difficult to verify.
This issue is reflected in the analysis of augmented Lagrangian and Newton-type methods, where convergence rates typically hinge on error bounds or calmness properties of the KKT solution mapping.

Our analysis employs the stratification framework at two levels.
Within a fixed stratum, the smoothness of $\Pi_{\mathbb{S}^n_{+}}$ allows the KKT mapping to be viewed as a smooth equation on a manifold, and the corresponding manifold differential admits a closed form.
Building on this observation, we introduce \emph{stratum-restricted} strong metric regularity, tailored to the geometry induced by the inertia stratification.
 {At a KKT pair, this regularity is equivalent to injectivity of the manifold differential. Under the SRCQ and the second order necessary condition (SONC, Definition~\ref{def:SONC-nonKKT}), it is further equivalent to the W-SOC and the W-SRCQ of~\cite{feng2025quadratically}.}
Moreover, stratum-restricted strong metric regularity is proved to be necessary for the superlinear convergence of semismooth Newton methods: if it fails, then every element of the Clarke generalized Jacobian $\partial_C F(\overline z)$ is singular.
A transversality-based interpretation further yields geometric insight into stability within a given stratum. Beyond individual strata, we investigate how these properties propagate across adjacent strata.
This, together with the stratum-wise analysis, yields a refined perturbation theory: weak-form conditions remain stable under partial-inertia-preserving perturbations, whereas classical strong-form conditions correspond to the uniform validity of the weak conditions over a neighborhood.
Finally, since the stability of the KKT solution mapping also controls the local behavior of augmented Lagrangian and other primal--dual methods, the proposed stratum-restricted framework provides a route to relaxing the assumptions used in the analysis of algorithms for~\eqref{prog:SDP} and related nonpolyhedral problems.

On the algorithmic side, we leverage the stratification framework to develop the \emph{Stratified Gauss--Newton} (SGN) method.
The method addresses the KKT system~\eqref{eq:kkt-mapping} by minimizing the nonsmooth least-squares merit function $\varphi(z):=\tfrac12\|F(z)\|^2$.
It combines three components: (i) a Levenberg--Marquardt step tangent to the current stratum; (ii) normal steps intended to escape stratum-confined stationarity; and (iii) an eigenvalue-triggered correction mechanism that updates the inertia.
A merit-function acceptance rule yields a monotone decrease of $\varphi$ and guarantees global convergence to a directionally stationary (D-stationary) point (Definition~\ref{def:D-stationary}).
 Furthermore, by combining the stratified perturbation theory with the local analysis, it is shown that, under the {SOSC} and the SRCQ {at an accumulation point and a suitable correction threshold}, the iterates eventually identify the active stratum and attain {superlinear convergence, with quadratic convergence if $f$ and $g$ are additionally of class $LC^2$ near the primal solution}.

The remainder of the paper is organized as follows.
Section~\ref{section:pre} reviews the metric projection onto $\mathbb{S}^n_{+}$, standard regularity conditions for NLSDP, and basic material from differential geometry and stratification theory.
Section~\ref{sec:geometric} introduces the inertia stratification of $\mathbb{S}^n$, proves that $\Pi_{\mathbb{S}^n_{+}}$ is $C^\infty$-smooth on each stratum, and derives its explicit differential.
Moreover, a stratum-wise variational theory is developed to analyze regularity both within and across strata, including the characterization of stratum-restricted strong regularity via the W-SOC and the W-SRCQ and its geometric interpretation via stability and transversality.
 Section~\ref{section:global} presents the stratified Gauss--Newton method and establishes global convergence together with local {superlinear convergence, with quadratic convergence under additional $LC^2$ regularity}.
Section~\ref{sec:numerical} reports numerical results, and Section~\ref{sect:conclusion} concludes the paper.
Technical proofs are deferred to the appendices.

\section{Preliminaries}
\label{section:pre}
We first establish the notation used throughout the paper. Let $\mathbb{X}$ be a Euclidean space. We denote the space of $n \times n$ symmetric matrices by $\mathbb{S}^n$, and the cone of symmetric positive semidefinite matrices by $\mathbb{S}^n_+$. The metric projection onto $\mathbb{S}^n_+$ is denoted by $\Pi_{\mathbb{S}^n_+}$. For a smooth mapping $F$ between Euclidean spaces, $F'(x)$ denotes its Fréchet derivative (Jacobian), and $\nabla F(x)$ denotes its adjoint (or gradient, if $F$ is real-valued). For a directionally differentiable function $f$, $f'(x;d)$ denotes its directional derivative at $x$ along the direction $d$. Additional notation is summarized as follows:
\begin{itemize}[left=0pt, nosep]
    \item $A_{ij}$: the $(i,j)$-th entry of a matrix $A\in\mathbb{R}^{m\times n}$.
    \item $A_{i}$: the $i$-th column of $A\in\mathbb{R}^{m\times n}$.
    \item $A_{\alpha\beta}$: the submatrix of $A$ consisting of the rows and columns indexed by the sets $\alpha$ and $\beta$, respectively.
    \item $\circ$: the Hadamard (elementwise) product of matrices; for $A, B \in \mathbb{R}^{m \times n}$, $(A \circ B)_{ij} = A_{ij}B_{ij}$.
    \item $\operatorname{Diag}(\cdot)$: the operator mapping a vector to a diagonal matrix whose main diagonal is given by the vector.
    \item $\operatorname{diag}(\cdot)$: the operator extracting the main diagonal of a matrix into a vector.
    \item $\lin(C)$: the linearity space of a closed convex set $C\subseteq \mathbb{X}$ (i.e., the largest linear subspace contained in $C$).
    \item $\aff(C)$: the affine hull of $C$ (i.e., the smallest affine set containing $C$).
    \item $\Mpq$ (or $\Mcal$): the manifold of symmetric matrices in $\mathbb{S}^n$ with exactly $p$ positive and $q$ negative eigenvalues.
    \item $\tMcal_{p,q}$ (or $\tMcal$): the preimage $G^{-1}(\Mpq)$ (see~\eqref{eq:tilde-Mpq}).
    \item $\im ( \cdot)$: the image of some map
    \item $\oplus$: the direct sum of vector spaces.
    \item $\succeq$ ($\preceq$) and $\succ$ ($\prec$): the Loewner partial order. $A \succeq B$ ($A \preceq B$) means $A-B$ is positive (negative) semidefinite; $A \succ B$ ($A \prec B$) means $A-B$ is positive (negative) definite.

\end{itemize}

\subsection{Variational properties of $\Pi_{\mathbb{S}^n_+}$}

This subsection reviews fundamental properties of the metric projection $\Pi_{\mathbb{S}^n_+}$. For a given matrix $A \in \mathbb{S}^n$, let its eigenvalues (counted with multiplicities) be ordered as
\[
\lambda_{1}(A) \ge \lambda_{2}(A) \ge \cdots \ge \lambda_{n}(A).
\]
Let $\lambda(A)$ be the vector of these ordered eigenvalues, and define $\Lambda(A) := \operatorname{Diag}(\lambda(A))$. We partition the index set $\{1, \ldots, n\}$ into $\alpha(A)$, $\beta(A)$, and $\gamma(A)$ according to the signs of the eigenvalues:
\begin{equation}\label{eq:index}
\alpha(A) := \{ i \mid \lambda_i(A) > 0 \}, \quad
\beta(A) := \{ i \mid \lambda_i(A) = 0 \} \quad \text{and} \quad
\gamma(A) := \{ i \mid \lambda_i(A) < 0 \}.
\end{equation}

When the context is unambiguous, we suppress the explicit dependence on $A$ and simply write $\alpha$, $\beta$, and $\gamma$. The positive and negative inertia indices of $A$ are denoted by $p:= |\alpha|$ and $q:= |\gamma|$, respectively. The eigenvalue decomposition of $A$ takes the block form
\begin{equation}\label{eq:eig-de}
A
= \begin{bmatrix}
P_{\alpha} & P_{\beta} & P_{\gamma}
\end{bmatrix}
\begin{bmatrix}
\Lambda(A)_{\alpha\alpha} & 0 & 0 \\[3pt]
0 & 0 & 0 \\[3pt]
0 & 0 & \Lambda(A)_{\gamma\gamma}
\end{bmatrix}
\begin{bmatrix}
P_{\alpha}^{\top} \\[3pt]
P_{\beta}^{\top} \\[3pt]
P_{\gamma}^{\top}
\end{bmatrix},
\end{equation}
where the orthogonal matrix $P \in \mathcal{O}^{n}$ is partitioned accordingly. Let $\mathcal{O}^{n}(A)$ denote the set of all such orthogonal matrices $P$ satisfying~\eqref{eq:eig-de}. We refer to the tuple $(\alpha,\beta,\gamma,p,q,P,\lambda)$ as an \emph{indexed eigenvalue decomposition (IED)} of $A$.

\begin{remark}
    We refer to \emph{an} IED rather than \emph{the} IED because the orthogonal matrix $P\in\mathcal{O}^n(A)$ is generally not unique. Nevertheless, our subsequent constructions and analytical results are invariant to the specific choice of $P$.
\end{remark}

The metric projection of $A$ onto $\mathbb{S}^n_+$ is explicitly given by $\Pi_{\mathbb{S}^n_+}(A)= P_{\alpha}\Lambda(A)_{\alpha\alpha}P_{\alpha}^{\top}$. Furthermore, $\Pi_{\mathbb{S}^n_+}$ is globally Lipschitz continuous and directionally differentiable (i.e., B-differentiable~\cite{bonnans1998sensitivity,bonnans1999second}) and strongly semismooth~\cite{sun2002semismooth}, although it is generally not Fr\'echet differentiable. Specifically, the directional derivative of $\Pi_{\mathbb{S}^n_+}$ at $A$ along the direction $H \in \mathbb{S}^n$ is formulated as
\begin{equation}\label{eq:Pi-dir}
\Pi_{\mathbb{S}^n_+}^{\prime}(A;H)
= P
\begin{bmatrix}
\widetilde{H}_{\alpha\alpha} & \widetilde{H}_{\alpha\beta} & \Xi_{\alpha\gamma} \circ \widetilde{H}_{\alpha\gamma} \\
\widetilde{H}_{\alpha\beta}^{\top} & \Pi_{\mathbb{S}^{|\beta|}_+}(\widetilde{H}_{\beta\beta}) & 0 \\
\Xi_{\alpha\gamma}^{\top} \circ \widetilde{H}_{\alpha\gamma}^{\top} & 0 & 0
\end{bmatrix}
P^{\top},
\end{equation}
where $P \in \mathcal{O}^n(A)$, $\widetilde{H} := P^{\top} H P$, and the symmetric weight matrix $\Xi \in \mathbb{S}^n$ has entries defined by
\begin{equation}\label{eq:Xi}
\Xi_{ij}
:= \frac{\max\{\lambda_{i}(A),0\} - \max\{\lambda_{j}(A),0\}}
{\lambda_{i}(A) - \lambda_{j}(A)},
\quad i,j = 1,\ldots,n,
\end{equation}
adopting the convention that $0/0 := 1$.

\subsection{Regularity conditions}
In this subsection, we collect several variational notions and regularity conditions that will be used throughout the paper. A comprehensive account of the background material can be found, e.g., in~\cite{rockafellar_variational_1998,bonnans2013perturbation}.

We begin with the tangent cone. For a closed convex set $\mathcal{K} \subseteq \mathbb{S}^n$, the tangent cone to $\mathcal{K}$ at $y \in \mathcal{K}$ is defined by
\begin{equation}\label{eq:def-Tangent-g}
    \mathcal{T}_{\mathcal{K}}(y) := \bigl\{ d \in \mathbb{S}^n \mid \exists\, t_k \downarrow 0 \text{ such that } \operatorname{dist}(y + t_k d, \mathcal{K}) = o(t_k) \bigr\}.
\end{equation}

When $\mathcal{K}$ is a smooth embedded submanifold $\mathcal{M}$ of a Euclidean space $\mathbb{X}$, the variational tangent cone $\mathcal{T}_{\mathcal{M}}(x)$ coincides with the classical tangent space $\mathcal{T}_x\mathcal{M}$. To keep track of whether we are working in the variational or differential-geometric convention, we retain both notations: we write $\mathcal{T}_{\mathcal{K}}(y)$ for the tangent cone to a set $\mathcal{K}$ at $y$, and $\mathcal{T}_A\mathcal{M}$ for the tangent space to a manifold $\mathcal{M}$ at $A$. In particular, switching the positions of the set and the base point will always indicate the intended geometric context.

We next recall the Robinson constraint qualification (RCQ) for~\eqref{prog:SDP}. Let $x$ be a feasible solution of~\eqref{prog:SDP}. As shown in~\cite[Theorem~3.9 and Proposition~3.17]{bonnans2013perturbation}, the RCQ stated below is equivalent to the nonemptiness and boundedness of the associated Lagrange multiplier set $M(x)$.

\begin{definition}\label{def:RCQ}
	The \emph{Robinson constraint qualification (RCQ)} holds at a feasible point $x$ if
    \begin{equation*}
        g'(x)\mathbb{X} + \mathcal{T}_{\mathbb{S}^n_+}(g(x)) = \mathbb{S}^n.
    \end{equation*}
\end{definition}

We next introduce the strict Robinson constraint qualification (SRCQ) at a primal--dual point. This condition can be viewed as an extension of the classical SRCQ formulated at KKT pairs; see, e.g.,~\cite[Proposition~4.50]{bonnans2013perturbation} for the standard form and its consequences.
\begin{definition}\label{def:SRCQ-nonKKT}
    For $z = (x,y) \in \mathbb{X} \times \mathbb{S}^n$ with an IED $(\alpha,\beta,\gamma,p,q, P, \lambda)$ of $G(z)$, the  \emph{strict Robinson constraint qualification (SRCQ)} holds at $z$ if
    \begin{equation*}
        g'(x)\mathbb{X} + \Bigl\{ P B P^\top\in\mathbb{S}^n \,\Big|\,
B_{\beta\beta}\succeq 0,\ B_{\beta\gamma}=0,\ B_{\gamma\gamma}=0\Bigr\}= \mathbb{S}^n.
    \end{equation*}
\end{definition}

By exploiting IED of $G(z)$, one can directly verify that, when $z=(x,y)$ is a KKT pair, Definition~\ref{def:SRCQ-nonKKT} reduces to the classical SRCQ. In particular, under this classical SRCQ, the Lagrange multiplier set $M(\overline{x})$ is a singleton; see~\cite[Proposition~4.50]{bonnans2013perturbation}.

Following the characterization in~\cite[(36)]{sun2006strong}, we also introduce an extension of constraint nondegeneracy that is formulated at a primal--dual pair and hence explicitly involves the multiplier variable. This explicit dependence is immaterial in the classical setting, since classical constraint nondegeneracy implies multiplier uniqueness and the dual variable can be suppressed. By contrast, once one allows situations where multiplier uniqueness may fail, keeping the dual variable becomes natural and, for our purposes, essential. Moreover, at a feasible point $x$, the classical constraint nondegeneracy is equivalent to its extended version in Definition~\ref{def:CN-nonKKT} evaluated at any $(x,y)$ satisfying the complementarity condition $\Sbb^n_+ \ni g(x) \perp y \in \Sbb^n_-$.

\begin{definition}\label{def:CN-nonKKT}
    For $z = (x,y) \in \mathbb{X} \times \mathbb{S}^n$ with an IED $(\alpha,\beta,\gamma,p,q, P, \lambda)$ of $G(z)$, the \emph{constraint nondegeneracy} holds at $z$ if
    \begin{equation}\label{eq:CN-nonKKT}
        g'(x)\mathbb{X} + \bigl\{ P B P^\top \mid B \in \mathbb{S}^n,\ B_{\beta\beta} = 0,\ B_{\beta\gamma} = 0,\ B_{\gamma\gamma} = 0 \bigr\} = \mathbb{S}^n.
    \end{equation}
\end{definition}

We will also make use of the weak strict Robinson constraint qualification (W-SRCQ) and the weak second order condition (W-SOC) proposed recently by Feng et al.~\cite{feng2025quadratically}. These notions are originally formulated at KKT pairs of~\eqref{prog:SDP} and serve as relaxed regularity/second order requirements in that setting. Since our subsequent sensitivity analysis involves general primal--dual points (not necessarily satisfying the KKT conditions), we extend both concepts to arbitrary primal--dual pairs.

We begin with the extension of the W-SRCQ. As shown in~\cite[Lemma~1]{feng2025quadratically}, the definition below reduces to the original W-SRCQ at any KKT pair; cf.~\cite[Definition~7]{feng2025quadratically}.

\begin{definition}\label{def:W-SRCQ-nonKKT}
    For $z = (x,y) \in \mathbb{X} \times \mathbb{S}^n$ with an IED $(\alpha,\beta,\gamma,p,q, P, \lambda)$ of $G(z)$, the \emph{weak strict Robinson constraint qualification (W-SRCQ)} holds at $z$ if
    \begin{equation}\label{eq:W-SRCQ-nonKKT}
        g'(x)\mathbb{X} + \bigl\{ P B P^\top \mid B \in \mathbb{S}^n,\ B_{\beta\gamma} = 0,\ B_{\gamma\gamma} = 0 \bigr\} = \mathbb{S}^n.
    \end{equation}
\end{definition}

To formulate second order conditions at a general primal--dual point, we associate with $z=(x,y)$ the quadratic form below (defined via an IED of $G(z)$):
\begin{equation}\label{eq:def-Qz}
    \mQ_{z}(v_x) := \langle v_x, \nabla^2_{xx} L(x, y) v_x \rangle + 2 \sum_{i \in \alpha} \sum_{j \in \gamma} \frac{-\lambda_j}{\lambda_i} \big[ P^\top (\nabla g(x)^* v_x) P \big]_{ij}^2.
\end{equation}
When $(x, y)$ is a KKT pair of~\eqref{prog:SDP}, it is well known~\cite[(28)]{sun2006strong} that
\[
\mQ_{z}(v_x) = \langle v_x, \nabla^2_{xx} L(\overline{x}, \overline{y}) v_x \rangle - \sigma\left( \overline{y}, \mathcal{T}^2_{\mathbb{S}^n_+}(g(\overline{x}), g'(\overline{x}) v_x) \right),
\]
where $\sigma(\cdot, \cdot)$ denotes the support function and $\mathcal{T}^2_{\mathbb{S}^n_+}$ is the second order tangent sets~\cite[(3.50)]{bonnans2013perturbation} to $\mathbb{S}^n_+$.

Using $\mQ_z$, we now state second order sufficient conditions at an arbitrary primal--dual pair. The following SOSC is adapted from~\cite[(3.276)]{bonnans2013perturbation}; one readily checks that it reduces to the classical SOSC when $(x,y)$ is a KKT pair.

\begin{definition}\label{def:SOSC-nonKKT}
	For $z = (x,y) \in \mathbb{X} \times \mathbb{S}^n$ with an IED $(\alpha,\beta,\gamma,p,q, P, \lambda)$ of $G(z)$. The \emph{second order sufficient condition (SOSC)} holds at $z$ if
    \begin{equation}\label{eq:SOSC-P}
        \mQ_{z}(v_x) > 0, \quad \forall\, v_x \in \mathcal{C}(z) \setminus \{0\},
    \end{equation}
    where $\mathcal{C}(z)$ is given by
    \begin{equation}\label{eq:def-C-nonKKT}
    	\mathcal{C}(z) := \left\{ v_x \in \mathbb{X} \mid [P^\top (\nabla g(x)^*v_x) P]_{\beta\beta} \succeq 0,\ [P^\top (\nabla g(x)^*v_x) P]_{\beta\gamma} = 0,\ [P^\top (\nabla g(x)^*v_x) P]_{\gamma\gamma} = 0 \right\}.
    \end{equation}
\end{definition}

\begin{definition}\label{def:SONC-nonKKT}
    At a primal--dual pair $z$, the \emph{second order necessary condition (SONC)} holds if $\mQ_z(v_x)\ge0$ for every $v_x\in\mathcal C(z)$, where $\mathcal C(z)$ is defined in~\eqref{eq:def-C-nonKKT}.
\end{definition}

We also require a stronger variant based on the outer-approximation set used in~\cite{sun2006strong}. Specifically, relying on the explicit expressions for $\operatorname{app}(\overline{x}, \overline{y})$ in~\cite[(38)]{sun2006strong}, we extend the strong SOSC (S-SOSC) from~\cite[Definition~3.2]{sun2006strong} to arbitrary primal--dual pairs. Again, the definition below reduces to the original one whenever $(\overline{x}, \overline{y})$ is a KKT pair.

\begin{definition}\label{def:S-SOSC-nonKKT}
	For $z = (x,y) \in \mathbb{X} \times \mathbb{S}^n$ with an IED $(\alpha,\beta,\gamma,p,q, P, \lambda)$ of $G(z)$, the \emph{strong second order sufficient condition (S-SOSC)} holds at $z$ if
    \begin{equation}\label{eq:S-SOSC-nonKKT}
        \mQ_{z}(v_x) > 0 \quad \forall\, v_x \in \operatorname{app}(z) \setminus \{0\},
    \end{equation}
    where $\operatorname{app}(x, y)$ is given by
    \begin{equation}\label{eq:def-app-nonKKT}
    	\operatorname{app}(z) := \left\{ v_x \in \mathbb{X} \mid [P^\top (\nabla g(x)^*v_x) P]_{\beta\gamma} = 0,\ [P^\top (\nabla g(x)^*v_x) P]_{\gamma\gamma} = 0 \right\}.
    \end{equation}
\end{definition}

We now turn to weak second order conditions in the sense of~\cite{feng2025quadratically}. Motivated by the characterizations in~\cite[(31) and (34)]{feng2025quadratically}, we define the extended W-SOC as follows.

\begin{definition}\label{def:W-SOC-nonKKT}
	For $z = (x,y) \in \mathbb{X} \times \mathbb{S}^n$ with an IED $(\alpha,\beta,\gamma,p,q, P, \lambda)$ of $G(z)$, the \emph{weak second order condition (W-SOC)} holds at $z$ if
    \begin{equation}\label{eq:W-SOC-nonKKT}
        \mQ_{z}(v_x) \neq 0 \quad \forall\, v_x \in \operatorname{appl}(z) \setminus \{0\},
    \end{equation}
    where $\operatorname{appl}(z)$ is given by
    \begin{equation}\label{eq:def-appl-nonKKT}
    	\operatorname{appl}(z) := \left\{ v_x \in \mathbb{X} \mid [P^\top (\nabla g(x)^*v_x) P]_{\beta\beta} = 0,\ [P^\top (\nabla g(x)^*v_x) P]_{\beta\gamma} = 0,\ [P^\top (\nabla g(x)^*v_x) P]_{\gamma\gamma} = 0 \right\}.
    \end{equation}
\end{definition}

\begin{remark}
    In Definition~\ref{def:W-SOC-nonKKT}, we modify the W-SOC from a strict inequality (``$>0$'') as originally proposed in~\cite{feng2025quadratically} to a non-vanishing condition (``$\neq 0$''). This adjustment streamlines the exact equivalence established later in Theorem~\ref{thm:WR-dF-nons}. By the intermediate value property of continuous functions, this condition ensures that the quadratic form maintains a constant sign. Furthermore, under standard assumptions at a KKT pair (e.g., if~\eqref{prog:SDP} is convex or satisfies the weak second order necessary condition), our definition naturally reduces to ``$>0$'', recovering the formulation in~\cite{feng2025quadratically}.
\end{remark}

It can be verified directly from the definitions that
\[
\mbox{constraint nondegeneracy} \quad \Longrightarrow \quad \mbox{SRCQ} \quad \Longrightarrow \quad \mbox{W-SRCQ}
\]
and
\[
\mbox{S-SOSC} \quad \Longrightarrow \quad \mbox{SOSC} \quad \Longrightarrow \quad \mbox{W-SOC}.
\]
Finally, we introduce a weak second order necessary condition (W-SONC), which extends the notion in~\cite[Definition~1]{fukuda2023weak} from KKT pairs to general primal--dual points and will be used in our subsequent analysis.

\begin{definition}\label{def:W-SONC}
    Let $z = (x,y) \in \mathbb{X} \times \mathbb{S}^n$ with an IED $(\alpha,\beta,\gamma,p,q,P,\lambda)$ of $G(z)$. The \emph{weak second order necessary condition (W-SONC)} holds at $z$ if
    \begin{equation*}
        \mQ_{z}(v_x) \ge 0 \quad \forall\, v_x \in \operatorname{appl}(z),
    \end{equation*}
    where $\operatorname{appl}(z)$ is defined in~\eqref{eq:def-appl-nonKKT}.
\end{definition}

\begin{remark}
    In~\cite{fukuda2023weak}, the ``W-SOC'' refers to the \emph{weak second order necessary} condition. In the present paper, however, we reserve ``W-SOC'' for the condition in Definition~\ref{def:W-SOC-nonKKT}, which is introduced for our regularity analysis and should not be confused with a second order sufficient condition. To avoid ambiguity, we therefore refer to the weak second order necessary condition in Definition~\ref{def:W-SONC} as \emph{W-SONC}. There is also a minor difference between Definition~\ref{def:W-SONC} and~\cite[Definition~1]{fukuda2023weak}. The latter is stated in terms of $\lin(\mathcal{C}(x))$, whereas we use the set $\appl(z)$. From a practical standpoint, explicitly computing $\lin(\mathcal{C}(x))$ can be inconvenient, while $\appl(z)$ provides a tractable surrogate that preserves the variational information needed in our analysis. This choice is consistent with the use of the outer-approximation set $\app(z)$ in~\cite{sun2006strong}. We refer to~\cite{fukuda2023weak} for a detailed discussion of the W-SONC at KKT pairs, which is not central to our development. We only emphasize one key point: even in nonlinear programming, barrier methods~\cite{gould1999note} and augmented Lagrangian methods~\cite{andreani2018note} may produce limit points that violate the classical second order necessary condition, whereas such limit points are still guaranteed to satisfy the W-SONC  (see~\cite{andreani2017second}).
\end{remark}

\begin{remark}\label{remark:convex-W-SONC}
    Suppose that problem~\eqref{prog:SDP} is convex, i.e., $f$ is convex and $g$ is matrix-concave, meaning that
    \[
        g(t x_1 + (1-t)x_2) \succeq t g(x_1) + (1-t)g(x_2)
        \quad \forall\, x_1,x_2\in\Xbb,\ t\in[0,1].
    \]
    Then, for any $x\in\Xbb$ and any $y\in\Sbb^n_-$, we have
    \[
        \nabla_{xx}^2 L(x,y) \succeq 0.
    \]
    Consequently, the W-SONC holds automatically at every $z=(x,y)\in \Xbb\times \Sbb^n_-$.
\end{remark}

\begin{remark}\label{remark:extended-not-KKT}
It is worth emphasizing that the extended regularity conditions in Definitions~\ref{def:SRCQ-nonKKT}--\ref{def:W-SONC} are defined at an arbitrary primal--dual pair $z$ via an IED of $G(z)$. At non-KKT pairs, they should not be viewed as standalone optimality conditions. In particular, by themselves they do not imply primal feasibility, complementarity, stationarity, or the identity $F(z)=0$. This level of generality is essential for the analysis in Section~\ref{sec:geometric}. Indeed, the stratified variational analysis is carried out on the lifted strata, and therefore necessarily involves arbitrary primal--dual pairs, not only KKT pairs. The following elementary example shows that the S-SOSC from Definition~\ref{def:S-SOSC-nonKKT} and the constraint nondegeneracy from Definition~\ref{def:CN-nonKKT} may hold simultaneously at a non-KKT pair. Consequently, the weaker conditions the SOSC, W-SOC, SRCQ, and W-SRCQ may also hold at a non-KKT pair. Consider the scalar problem
\[
\min_{x\in\mathbb{R}} \ \frac{1}{2}(x-2)^2
\quad\text{s.t.}\quad
g(x):=x\in \mathbb{S}^1_+.
\]
For any primal--dual pair $z=(x,y)\in \Rbb\times \Sbb^1$, since $g'(x)\equiv 1$ is surjective, constraint nondegeneracy holds automatically at $z$. Moreover,
\[
Q_z(v_x)=\langle v_x,\nabla^2_{xx}L(x,y)v_x\rangle=v_x^2>0
\qquad\forall\, v_x\neq 0,
\]
so the S-SOSC also holds at $z$. Hence, both the S-SOSC and constraint nondegeneracy are satisfied at every primal--dual pair $z$, including non-KKT pairs. In conclusion, these extended regularity conditions alone are not intended to characterize optimality or stationarity. Rather, they provide the appropriate geometric and variational framework for the stratum-wise analysis developed later and, when combined with additional first-order information, can be used to lead to stronger conclusions.
\end{remark}

From this point onward, unless stated otherwise, the terms SRCQ, constraint nondegeneracy, W-SRCQ, SOSC, S-SOSC, and W-SOC are understood in their extended senses.

The constraint qualifications and second order conditions introduced in Definitions~\ref{def:RCQ}--\ref{def:W-SONC} will be referred to as \emph{problem-level} regularity conditions, in the sense that they are formulated directly in terms of the problem data (derivatives and spectral information) and can, at least in principle, be checked without solving auxiliary variational systems. In contrast, we call the conditions introduced next \emph{solution-level} regularity conditions, since they are typically expressed through properties of set-valued mappings and often involve solving equations and/or evaluating distances, which can be substantially more demanding computationally. These solution-level notions coincide with the conventional meaning of regularity in variational analysis. Nonetheless, we also view the above constraint qualifications and second order conditions as regularity conditions, because they capture the geometric structure underlying a well-behaved local solution mapping. Among solution-level conditions, the most prominent one is the \emph{strong metric regularity}, which we recall next.

\begin{definition}\label{def:strongly metric regular}
    We say a set-valued mapping $\Phi:\Xbb\rightrightarrows\Ybb$ is \emph{strongly metrically regular} at $(\overline  x,\overline  y)$, if $\overline  y\in \Phi(\overline  x)$, and there exist constants $ \kappa > 0 $ and neighborhoods $ \mU$ of $ \overline {x} $, $ \mV $ of $ \overline {y} $, such that for all $ x \in \mU $ and $ y \in \mV$, the following inequality holds
    \begin{equation*}
        \dist(x, \Phi^{-1}(y)) \leq \kappa \cdot \dist(y, \Phi(x))
    \end{equation*}
    and $\Phi^{-1}(y) \cap \mathcal{U}$ is a singleton for every $y \in \mathcal{V}$.
\end{definition}

It is well-known that, at a KKT pair of \eqref{prog:SDP}, the strong metric regularity of the KKT natural mapping $F$ we considered is equivalent to the strong regularity \cite{robinson1980strongly} of the KKT system formulated as a generalized equation \cite{dontchev2009implicit}. Moreover, under the RCQ at a local optimum, this strong regularity is in turn equivalent to the simultaneous satisfaction of the S-SOSC (Definition~\ref{def:S-SOSC-nonKKT}) and constraint nondegeneracy (Definition~\ref{def:CN-nonKKT})~\cite{sun2006strong}, whereas strong metric subregularity corresponds to the conjunction of the SRCQ and the SOSC~\cite{ding2017characterization}.

Metric regularity is a weaker property, as it does not require the inverse mapping to be single-valued.
\begin{definition}\label{def:metric regular}
    We say a set-valued mapping $\Phi:\Xbb\rightrightarrows\Ybb$ is \emph{metrically regular} at $(\overline{x},\overline{y})$, if $\overline{y}\in \Phi(\overline{x})$, and there exist a constant $\kappa > 0$ and neighborhoods $\mathcal{U}$ of $\overline{x}$, $\mathcal{V}$ of $\overline{y}$, such that for all $x \in \mathcal{U}$ and $y \in \mathcal{V}$, the following inequality holds:
    \begin{equation*}
        \dist(x, \Phi^{-1}(y)) \leq \kappa \cdot \dist(y, \Phi(x)).
    \end{equation*}
\end{definition}

Although metric regularity is, in general, weaker than strong metric regularity for an arbitrary set-valued mapping,
the two properties coincide for the \emph{KKT (natural) mapping} associated with several important nonpolyhedral conic programs.
In particular, such an equivalence has been established for nonlinear second order cone programming~\cite{chen_Aubin_25},
for NLSDP~\cite{chen_characterizations_2026}, and, more broadly, for $C^2$-cone reducible conic programs~\cite{ma2025aubin}.

Another closely related notion is the \emph{local error bound}. Various error bound formulations appear in the literature, and a systematic discussion is beyond the scope of this paper. We record one standard version below, noting that it follows directly from metric regularity.
\begin{definition}\label{def:local error bound}
    Let \( F: \mathbb{X} \to \mathbb{Y} \) be a single-valued mapping, and let \( \overline{x} \in \mathbb{X} \) satisfy \( F(\overline{x}) = 0 \).
    We say that the equation \( F(x) = 0 \) admits a \emph{local error bound} at \( \overline{x} \) if there exist a constant \( \kappa > 0 \) and a neighborhood \( \mathcal{U} \) of \( \overline{x} \) such that for all \( x \in \mathcal{U} \),
    \begin{equation*}
        \dist\bigl(x, F^{-1}(0)\bigr) \leq \kappa \cdot \|F(x)\|.
    \end{equation*}
\end{definition}

\subsection{Manifolds and stratification}

This subsection briefly reviews essential concepts from differential geometry and stratification theory used throughout the paper. For standard definitions regarding smooth manifolds (e.g., local charts, tangent spaces, tangent maps, tangent bundles, Riemannian metrics, and exponential maps), we refer the reader to comprehensive texts such as~\cite{lee2012introduction,absil2008optimization}.

We first clarify a potential notational ambiguity. Let $\mathcal{M}$ be a manifold embedded in a Euclidean space $\mathbb{Y}$, and let $p \in \mathcal{M}$. Under the canonical identification $\mathcal{T}_p \mathbb{Y} \cong \mathbb{Y}$, we identify any tangent vector $v \in \mathcal{T}_p \mathbb{Y}$ with an element of $\mathbb{Y}$. This natural identification permits the vector sum $p + v$ within the ambient space $\mathbb{Y}$. To distinguish the underlying geometries, we use $dh_x$ to denote the differential of a mapping between manifolds, and reserve $h'(x)$ for the Fr\'echet derivative of mappings defined on Euclidean spaces. In accordance with standard optimization conventions, the adjoint of the linear operator $h'(x)$ is denoted by $\nabla h(x) := (h'(x))^*$.

Next, we recall the concept of transversality from differential topology (see, e.g.,~\cite[Chapter~3.2]{hirsch2012differential}). This topological concept, which slightly generalizes the definition typically employed in variational analysis~\cite{bonnans2013perturbation}, formalizes the geometric intuition of a smooth map intersecting a submanifold in a ``regular'' manner.

\begin{definition}[{\cite[Chapter 3.2]{hirsch2012differential}}]\label{def:transversal}
    Suppose $h:\Mcal \rightarrow \Ncal$ is a smooth mapping between smooth manifolds $\mathcal{M}$ and $\mathcal{N}$, where $\Ncal' \subseteq \Ncal$ is a submanifold of $\Ncal$.  We say $h$ \emph{intersects with} $\Ncal'$ \emph{transversally} at $x$ in $\Ncal$, or simply  $h$ \emph{is transverse to} $\Ncal'$ at $x$ in $\Ncal$, denoted by $h \pitchfork_{x} \Ncal'$ in $\Ncal$, whenever $h(x) \in \Ncal'$,
    \begin{align}\label{eq:transversality}
        dh_x(\Tcal_x \Mcal) + \Tcal_{h(x)} \Ncal' = \Tcal_{h(x)} \Ncal,
    \end{align}
    {where $dh_x$ is the differential of $h$ at $x$ and $\Tcal$ denotes the tangent space.}
    Moreover, for a given closed subset $\Lcal$  in $\Mcal$ if \eqref{eq:transversality} holds whenever $x\in \Lcal$ and $h(x) \in \Ncal'$, then we say $h$ \emph{intersects with} $\Ncal'$ \emph{transversally along} $\Lcal$ in $\Ncal$, denoted by $h \pitchfork_{\Lcal} \Ncal'$ in $\Ncal$. We also simply denote $h \pitchfork \Ncal'$ in $\Ncal$ if $\Lcal \equiv \Mcal$.
\end{definition}

Stratification is a fundamental notion in topology and singularity theory. It is well-known that smooth manifolds provide a flexible framework that covers many nonlinear spaces, such as orthogonal groups and Grassmannians. However, sets with singularities lie outside the class of manifolds. The concept of a stratified space is designed as a natural extension of the manifold framework to accommodate singular sets, such as real algebraic varieties~\cite{whitney1965local,whitney1965tangents}, which arise frequently in polynomial optimization and related areas \cite{ioffe2009invitation}. For background on (Whitney) stratifications, we refer the reader to~\cite{whitney1965local,whitney1965tangents,mather1970notes,molina2020handbook}. In the present work, the following basic definition will suffice for our subsequent developments.

\begin{definition}\label{def:stratification}
    Let \( \Xbb \) be a Euclidean space and \( \Xcal \subseteq \Xbb \) a closed subset. A \emph{stratification} of \( \Xcal \) is a locally finite partition \( S = \{\Scal_i\}_{i \in I} \) of \( \Xcal \) into connected embedded \( C^k \) submanifolds (typically $k \geq 1$) \( \Scal_i \) (called \emph{strata}) such that
    \begin{enumerate}
        \item \emph{Frontier condition}: For any two strata \( \Scal_i, \Scal_j \in S \), if \( \Scal_i \cap \overline{\Scal_j} \neq \emptyset \), then \( \Scal_i \subseteq \overline{\Scal_j} \), where \( \overline{\Scal_j}\) means the closure of $\Scal_j$. This implies the boundary \( \overline{\Scal_j} \setminus \Scal_j \) is a union of strata.
        \item \emph{Local finiteness}: Every point \( x \in \Xcal \) has a neighborhood in $\Xbb$ intersecting finitely many strata.
    \end{enumerate}
    A topological space with a stratification structure is called a \emph{stratified space}.
\end{definition}

Roughly speaking, a stratification may be viewed as a decomposition of a given set into smooth pieces together with their well-organized boundaries. In Section~\ref{sec:geometric}, we shall introduce a stratification of the Euclidean space $\mathbb{S}^n$ tailored to our purposes, even though $\mathbb{S}^n$ is itself a linear space and hence a smooth manifold without singularities.

\section{The inertia stratification of $\Sbb^n$}\label{sec:geometric}

In this section, we study a stratification of the symmetric matrix space $\mathbb{S}^n$ based on eigenvalue inertia. Specifically, $\mathbb{S}^n$ is decomposed into finitely many smooth strata comprising matrices with prescribed inertia. We refer to this partition as the \emph{inertia stratification} (see~\eqref{eq:index-stra}). This structure is naturally suited for semidefinite optimization: the cone $\mathbb{S}^n_+$ is the union of strata with zero negative inertia, and the metric projection $\Pi_{\mathbb{S}^n_+}$ is smooth when restricted to any individual stratum.

While $\mathbb{S}^n$ admits infinitely many stratifications in the sense of Definition~\ref{def:stratification}, we focus on the inertia-based partition because it admits a natural geometric description and often appears implicitly in related literature. For completeness, we rigorously define this structure and summarize its essential geometric properties.

For given integers $p, q \in \mathbb{Z}_+$ with $p+q \le n$, we define the set
\begin{equation}\label{eq:M_pq}
    \mathcal{M}_{p,q} := \bigl\{A \in \mathbb{S}^n \mid |\alpha(A)| = p,\ |\gamma(A)| = q \bigr\}.
\end{equation}
When the context is clear, we simply write $\mathcal{M}$ in place of $\mathcal{M}_{p,q}$. As shown in~\cite{HELMKE19951}, each $\mathcal{M}_{p,q}$ is a smooth embedded submanifold of $\mathbb{S}^n$ with dimension $\dim(\mathcal{M}_{p,q}) = n(p+q) - \frac{1}{2}(p+q)(p+q-1)$. The family $\{\mathcal{M}_{p,q}\}_{0\le p+q\le n}$ yields a finite disjoint decomposition of $\mathbb{S}^n$:
\begin{equation}\label{eq:index-stra}
    \mathbb{S}^n = \bigcup_{0\le p+q\le n} \mathcal{M}_{p,q}.
\end{equation}
This decomposition satisfies the frontier condition and local finiteness required by Definition~\ref{def:stratification}. Hence, it forms a stratification of $\mathbb{S}^n$---referred to as the \emph{inertia stratification}\footnote{The family $\{\mathcal{M}_{p,q}\}_{0\le p+q\le n}$ formally constitutes a Whitney stratification of $\mathbb{S}^n$ (see, e.g.,~\cite{whitney1965local,whitney1965tangents,mather1970notes} for definitions). This follows from the canonical Whitney stratification of the real algebraic variety of symmetric matrices with prescribed inertia, defined by the vanishing of suitable minors. We refer to~\cite{whitney1965local,whitney1965tangents} for a rigorous treatment. Moreover, Olikier~\cite{Olikier2023FirstorderOO} shows that the inertia stratification satisfies certain regularity properties naturally suited for first order optimality analysis.}---and each $\mathcal{M}_{p,q}$ is termed a \emph{stratum} of $\mathbb{S}^n$.

Figure~\ref{fig:S2-stra} illustrates the stratification of $\mathbb{S}^2 \simeq \mathbb{R}^3$. This space decomposes into three $3$-dimensional strata, two $2$-dimensional strata, and the trivial zero-dimensional stratum $\{0\}$. Moreover, the adjacency among these strata ensures that the closure of any stratum is precisely the union of itself and all strata of strictly lower dimensions, consistent with the frontier condition in Definition~\ref{def:stratification}.

\begin{figure}[htbp]
    \centering
    \includegraphics[width=0.8\textwidth]{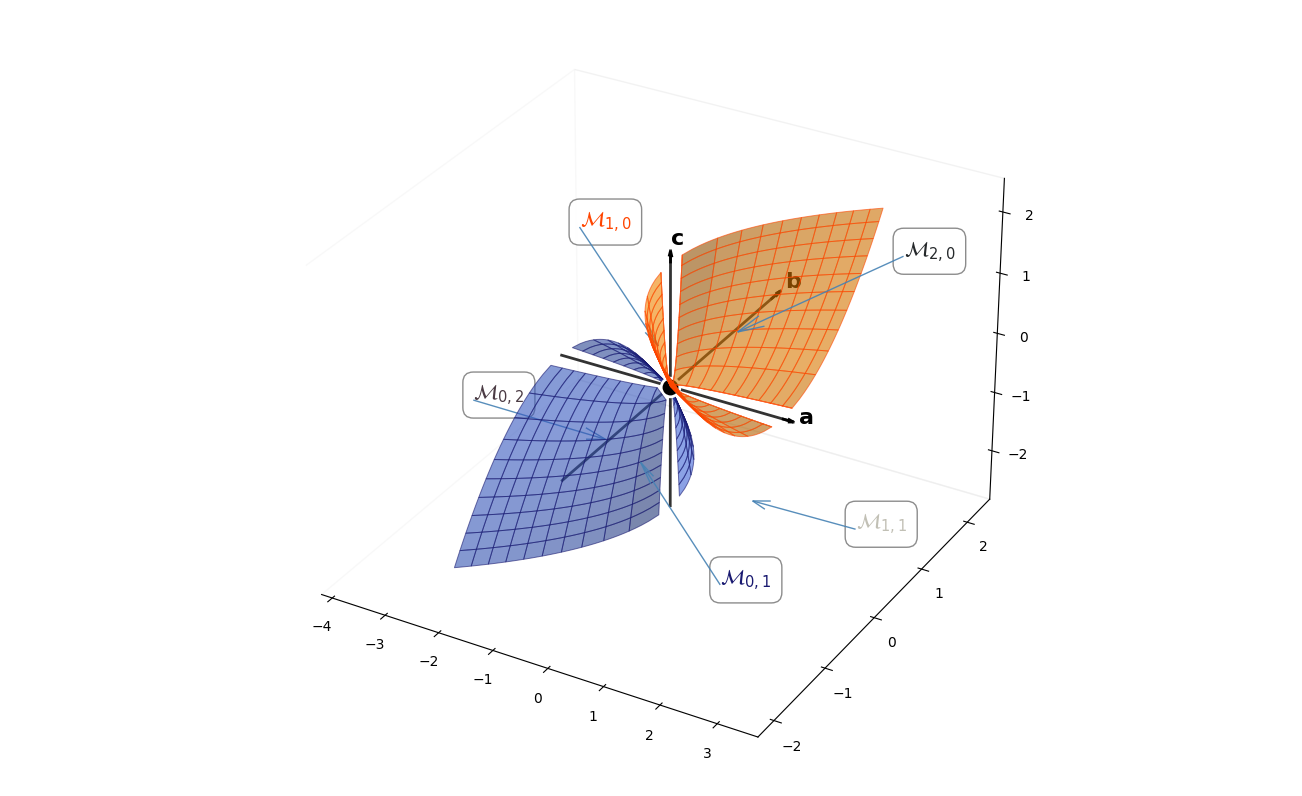}
    \caption{The inertia stratification of $\mathbb{S}^2$.}
    \label{fig:S2-stra}
\end{figure}

\begin{remark}\label{remark:add_SDPs}
    The inertia stratification framework and our subsequent analyses extend seamlessly to NLSDPs with multiple SDP constraints $g_i(x) \in \mathbb{S}^{n_i}_+$ for $i=1,\ldots,\ell$ (which inherently encompass polyhedral constraints of the form $h_E(x)=0$ and $h_I(x)\ge 0$). Indeed, the product space $\mathbb{X} \times \mathbb{S}^{n_1} \times \cdots \times \mathbb{S}^{n_\ell}$ inherits a stratification whose strata take the form
    \begin{equation*}
        \widetilde{\mathcal{M}}_{p_1,\ldots,p_\ell,q_1,\ldots,q_\ell} = \bigl\{ (x,y_1,\ldots,y_\ell) \mid g_i(x) + y_i \in \mathcal{M}_{p_i,q_i},\ i=1,\ldots,\ell \bigr\}.
    \end{equation*}
    Because this set is the preimage of the product manifold $\mathcal{M}_{p_1,q_1} \times \cdots \times \mathcal{M}_{p_\ell,q_\ell}$ under the smooth submersion $(x,y_1,\ldots,y_\ell) \mapsto (g_1(x)+y_1,\ldots,g_\ell(x)+y_\ell)$, it is a smooth embedded submanifold. Geometric objects such as tangent spaces and manifold derivatives can be computed componentwise. To streamline the presentation, we focus exclusively on~\eqref{prog:SDP}.
\end{remark}

Before examining further properties of the inertia stratification, we recall a convenient representation of the tangent space $\mathcal{T}_{A}\mathcal{M}_{p,q}$ at a point $A \in \mathcal{M}_{p,q}$. It is well established (see, e.g.,~\cite{HELMKE19951}) that
\begin{equation}\label{eq:TAM-0}
    \mathcal{T}_{A}\mathcal{M}_{p,q} = \bigl\{ H = X^{\top}A + AX \in \mathbb{S}^n \mid X \in \mathbb{R}^{n\times n} \bigr\}.
\end{equation}
The next proposition provides an alternative characterization based on the eigenvalue decomposition~\eqref{eq:eig-de}.

\begin{proposition}\label{prop:TAM}
Let $(\alpha,\beta,\gamma,p,q,P,\lambda)$ be an IED of $A \in \mathcal{M}_{p,q}$. Then, we have
\begin{equation}\label{eq:TAM}
    \mathcal{T}_{A} \mathcal{M}_{p,q} = \bigl\{ H \in \mathbb{S}^n \mid (P^{\top} H P)_{\beta\beta} = 0 \bigr\}.
\end{equation}
\end{proposition}
\begin{proof}
Let $\mathcal{R}$ denote the right-hand side of~\eqref{eq:TAM}. For any given $H \in \mathcal{T}_{A} \mathcal{M}_{p,q}$, it follows from~\eqref{eq:TAM-0} that $H = X^{\top}A + AX$ for some $X \in \mathbb{R}^{n\times n}$. Defining $\widetilde{H} := P^{\top}HP$, we obtain
\[
    \widetilde{H} = P^{\top}X^{\top}P
    \begin{bmatrix}
        \Lambda_{\alpha\alpha} & 0 & 0 \\
        0 & \Lambda_{\beta\beta} & 0 \\
        0 & 0 & \Lambda_{\gamma\gamma}
    \end{bmatrix}
    +
    \begin{bmatrix}
        \Lambda_{\alpha\alpha} & 0 & 0 \\
        0 & \Lambda_{\beta\beta} & 0 \\
        0 & 0 & \Lambda_{\gamma\gamma}
    \end{bmatrix}
    P^{\top}XP.
\]
Since $\Lambda_{\beta\beta} = 0$, this yields $\widetilde{H}_{\beta\beta} = 0$, implying $\mathcal{T}_{A}\mathcal{M}_{p,q} \subseteq \mathcal{R}$. Conversely, the dimension of $\mathcal{R}$ is clearly $\dim(\mathcal{R}) = n(p+q) - \frac{1}{2}(p+q)(p+q-1)$. Because $\dim(\mathcal{T}_{A}\mathcal{M}_{p,q})$ precisely matches this value~\cite[Proposition 2.1]{HELMKE19951}, we conclude that $\mathcal{T}_{A}\mathcal{M}_{p,q} = \mathcal{R}$. \ep
\end{proof}

Having established the stratified structure and characterized its tangent spaces, we shall analyze the differentiability of the metric projection $\Pi_{\mathbb{S}^n_+}$ along the strata. Specifically, when restricted to a fixed stratum $\mathcal{M}_{p,q}$, the mapping $\Pi_{\mathbb{S}^n_+}: \mathcal{M}_{p,q} \to \mathcal{M}_{p,0}$ is not only differentiable with an explicit closed-form differential, but is $C^\infty$-smooth, strengthening previously known $C^1$-differentiability results (see Remark~\ref{remark:C2-PiK}). These regularity properties fundamentally underpin the variational analysis and algorithmic developments in the sequel. As the proof relies on analytic functional calculus for symmetric matrices, it is deferred to Appendix~\ref{appendix:smooth}.

\begin{theorem}\label{thm:diff-PiK}
    For any stratum $\mathcal{M} = \mathcal{M}_{p,q}$, the restriction of $\Pi_{\mathbb{S}^n_+}$ to $\mathcal{M}$, denoted by $\Pi_{\mathbb{S}^n_+}|_{\mathcal{M}} : \mathcal{M} \to \mathbb{S}^n$, is a $C^\infty$-smooth map. Moreover, for an arbitrary $A \in \mathcal{M}$ with an IED $(\alpha,\beta,\gamma,p,q,P,\lambda)$, the manifold differential $\xi_A := d(\Pi_{\mathbb{S}^n_+}|_{\mathcal{M}})_A : \mathcal{T}_{A} \mathcal{M} \to \mathcal{T}_{\Pi_{\mathbb{S}^n_+}(A)} \mathbb{S}^n$ at $H \in \mathcal{T}_{A} \mathcal{M}$ is explicitly given by
    \begin{equation}\label{eq:xi_A}
        \xi_A(H) = P
        \begin{bmatrix}
            \widetilde{H}_{\alpha \alpha} & \widetilde{H}_{\alpha \beta} & \Xi_{\alpha \gamma} \circ \widetilde{H}_{\alpha \gamma} \\
            \widetilde{H}_{\alpha \beta}^\top & 0 & 0 \\
            \widetilde{H}_{\alpha \gamma}^\top \circ \Xi_{\alpha \gamma}^\top & 0 & 0
        \end{bmatrix}
        P^\top,
    \end{equation}
    where $\widetilde{H} := P^\top H P$.
\end{theorem}

\begin{remark}
    It is straightforward to verify that $\Pi_{\mathbb{S}^n_+}$ maps $\mathcal{M}_{p,q}$ onto $\mathcal{M}_{p,0}$, which is smoothly embedded in $\mathbb{S}^n$. Because $\Pi_{\mathbb{S}^n_+} : \mathcal{M}_{p,q} \to \mathbb{S}^n$ is a $C^\infty$ map, standard properties of smooth embeddings ensure that the corestriction $\Pi_{\mathbb{S}^n_+} : \mathcal{M}_{p,q} \to \mathcal{M}_{p,0}$ is also $C^\infty$-smooth. Hereafter, we suppress the explicit restriction notation $|_{\mathcal{M}_{p,q}}$ and simply write $\Pi_{\mathbb{S}^n_+}$ depending on the context.
\end{remark}

\begin{remark}\label{remark:C2-PiK}
    The $C^\infty$-smoothness established in Theorem~\ref{thm:diff-PiK}, alongside the closed-form differential~\eqref{eq:xi_A}, refines existing results for the metric projection $\Pi_{\mathbb{S}^n_+}$ along the strata. While the $C^1$-differentiability of $\Pi_{\mathbb{S}^n_+}$ relative to $\mathcal{M}_{p,q}$ can be deduced via L\"owner functional calculus or the theory of $C^1$-partly smooth mappings (see, e.g.,~\cite[Example~4.14]{lewis2002active} and~\cite[Proposition~9.7]{drusvyatskiy2014optimality}), Theorem~\ref{thm:diff-PiK} advances this literature by establishing full $C^\infty$-smoothness and providing an explicit manifold differential formula.
\end{remark}

The nonsmoothness of the KKT mapping $F(z)$ in~\eqref{eq:kkt-mapping} stems exclusively from the projection term $\Pi_{\mathbb{S}^n_+}(G(z))$. To isolate and exploit the hidden smooth structure, we lift the inertia stratification of $\mathbb{S}^n$ to the primal-dual space $\mathbb{X} \times \mathbb{S}^n$ via the mapping $G$. For each pair $(p,q)$ with $0 \le p+q \le n$, we define
\begin{equation}\label{eq:tilde-Mpq}
    \widetilde{\mathcal{M}}_{p,q} := G^{-1}(\mathcal{M}_{p,q}) = \bigl\{ z=(x,y) \in \mathbb{X} \times \mathbb{S}^n \mid G(z) = g(x)+y \in \mathcal{M}_{p,q} \bigr\}.
\end{equation}
Since $G$ is a smooth submersion (see, e.g.,~\cite[Chapter~4]{lee2012introduction}), it follows from~\cite[Corollary~6.31]{lee2012introduction} that each $\widetilde{\mathcal{M}}_{p,q}$ is a smoothly embedded submanifold of $\mathbb{X} \times \mathbb{S}^n$. Consequently, the finite disjoint union
\[
    \mathbb{X} \times \mathbb{S}^n = \bigcup_{0\le p+q\le n} \widetilde{\mathcal{M}}_{p,q}
\]
constitutes a stratification of $\mathbb{X} \times \mathbb{S}^n$, which we continue to call the \emph{inertia stratification}. When the indices $p$ and $q$ are clear from context, we simply write $\widetilde{\mathcal{M}}$ for $\widetilde{\mathcal{M}}_{p,q}$.

 By Theorem~\ref{thm:diff-PiK}, while $F$ is globally nonsmooth on $\mathbb{X} \times \mathbb{S}^n$, its restriction to any fixed stratum $\widetilde{\mathcal{M}}_{p,q}$ is {$C^1$ under the standing $C^2$ assumptions on $f$ and $g$}. The following lemma characterizes the tangent space of $\widetilde{\mathcal{M}}_{p,q}$ and facilitates the computation of the manifold differential of $F$. For brevity, we omit the straightforward proof.

\begin{lemma}\label{lemma:TzMt}
    Fix $(p,q)$ with $0 \le p+q \le n$, and let $\mathcal{M} := \mathcal{M}_{p,q}$ and $\widetilde{\mathcal{M}} := \widetilde{\mathcal{M}}_{p,q} = G^{-1}(\mathcal{M})$. For any $z = (x,y) \in \widetilde{\mathcal{M}}$, the tangent space $\mathcal{T}_{z}\widetilde{\mathcal{M}}$ can be represented as
    \begin{equation}\label{eq:TzMt-repr}
        \mathcal{T}_{z}\widetilde{\mathcal{M}} = \bigl\{ (v_x,v_y) \in \mathbb{X} \times \mathbb{S}^n \mid \nabla g(x)^{*}v_x + v_y \in \mathcal{T}_{G(z)}\mathcal{M} \bigr\}.
    \end{equation}
    In particular, the linear mapping $\phi_z : \mathcal{T}_{z}\widetilde{\mathcal{M}} \to \mathbb{X} \times \mathcal{T}_{G(z)}\mathcal{M}$ given by
    \begin{equation}\label{eq:phi_z}
        \phi_z(v_x,v_y) := \bigl(v_x,\ \nabla g(x)^{*}v_x + v_y\bigr)
    \end{equation}
    is a linear isomorphism with inverse $\phi_z^{-1} : \mathbb{X} \times \mathcal{T}_{G(z)}\mathcal{M} \to \mathcal{T}_{z}\widetilde{\mathcal{M}}$ given by
    \begin{equation}\label{eq:phi_z_inv}
        \phi_z^{-1}(v_x,H) := \bigl(v_x,\ H - \nabla g(x)^{*}v_x\bigr).
    \end{equation}
    For notational convenience, we identify a tangent vector $v \in \mathcal{T}_z \widetilde{\mathcal{M}}$ with the pair $(v_x, v_y)$ when viewed within the ambient space $\mathbb{X} \times \mathbb{S}^n$, and with $(v_x, H)$ when mapped via the isomorphism $\phi_z$. Moreover, the assignments $z \mapsto \phi_z$ and $z \mapsto \phi_z^{-1}$ are smooth (in the standard sense of smooth maps into the space of linear operators).
\end{lemma}

Applying the chain rule alongside the differential of $\Pi_{\mathbb{S}^n_+}$ restricted to $\mathcal{M}$ (Theorem~\ref{thm:diff-PiK}), Lemma~\ref{lemma:TzMt} yields an explicit representation for the manifold differential of $F$ along $\widetilde{\mathcal{M}}$. Specifically, for $z=(x,y) \in \widetilde{\mathcal{M}}$ and $v = (v_x,v_y) \in \mathcal{T}_z\widetilde{\mathcal{M}}$, we set
\[
    H := \nabla g(x)^{*}v_x + v_y \in \mathcal{T}_{G(z)}\mathcal{M}.
\]
Then, the manifold differential evaluates to
\begin{equation}\label{eq:dF_z(v)}
\begin{aligned}
    dF_z(v_x,v_y)
    &=
    \begin{bmatrix}
        \nabla^2_{xx}L(z)v_x + \nabla g(x)v_y \\
        -\nabla g(x)^{*}v_x + \xi_{G(z)}\bigl(dG_z (v_x,v_y)\bigr)
    \end{bmatrix} \\[3pt]
    &=
    \begin{bmatrix}
        \nabla^2_{xx}L(z)v_x + \nabla g(x)\bigl(H-\nabla g(x)^{*}v_x\bigr) \\
        -\nabla g(x)^{*}v_x + \xi_{G(z)}(H)
    \end{bmatrix} \\[3pt]
    &=
    \begin{bmatrix}
        \bigl(\nabla^2_{xx}L(z)-\nabla g(x)\nabla g(x)^{*}\bigr)v_x + \nabla g(x)H \\
        -\nabla g(x)^{*}v_x + \xi_{G(z)}(H)
    \end{bmatrix}.
\end{aligned}
\end{equation}
Equivalently, with respect to the coordinates $(v_x,H) \in \mathbb{X} \times \mathcal{T}_{G(z)}\mathcal{M}$ induced by~\eqref{eq:phi_z}, this linear operator admits the following block matrix representation:
\begin{equation}\label{eq:dF_z_block}
    dF_z(v_x,H) =
    \begin{bmatrix}
        \nabla^2_{xx}L(z)-\nabla g(x)\nabla g(x)^{*} & \nabla g(x) \\
        -\nabla g(x)^{*} & \xi_{G(z)}
    \end{bmatrix}
    \begin{bmatrix}
        v_x\\
        H
    \end{bmatrix}.
\end{equation}

The remainder of this section comprises two subsections. We first conduct a variational analysis restricted to a fixed stratum. Subsequently, we investigate the propagation of these regularity properties across adjacent strata. Together, these theoretical developments lay the foundation for designing globally convergent algorithms with rapid local convergence rates under conditions significantly weaker than classical requirements.

\subsection{Variational properties on stratum}\label{subsec:variational-properties}

In this subsection, we first characterize \emph{stratum-restricted strong metric regularity} (Definition~\ref{def:stratum-restricted strongly metric regular}) on a fixed stratum. To describe the geometric behavior of the KKT mapping along directions confined to that stratum, we subsequently analyze a related transversality condition, which yields a geometric interpretation of the proposed regularity framework.

\subsubsection{Stratum-restricted strong metric regularity}\label{sec:subsub-srsmr}

We introduce a natural stratum-restricted extension of strong metric regularity (Definition~\ref{def:strongly metric regular}). This localized property is significantly weaker than its classical ambient-space counterpart, particularly in the presence of degeneracy.

\begin{definition}\label{def:stratum-restricted strongly metric regular}
     Let $\mathbb{X}$ be a stratified space. A set-valued mapping $\Phi: \mathbb{X} \rightrightarrows \mathbb{Y}$ is said to be \emph{stratum-restricted strongly metrically regular} on a stratum $\mathcal{M} \subseteq \mathbb{X}$ at $(\overline{x}, \overline{y})$ if $\overline{x} \in \mathcal{M}$, $\overline{y} \in \Phi(\overline{x})$, and there exist a constant $\kappa > 0$ and neighborhoods $\mathcal{U}$ of $\overline{x}$ and $\mathcal{V}$ of $\overline{y}$ such that for all $x \in \mathcal{U} \cap \mathcal{M}$ and $y \in \mathcal{V} \cap {\Phi(\mathcal{U} \cap \mathcal{M})}$, the following inequality holds:
    \begin{equation*}
        \operatorname{dist}\bigl(x, \Phi^{-1}(y) \cap \mathcal{M} {\cap \mathcal{U}}\bigr) \leq \kappa \cdot \operatorname{dist}\bigl(y, \Phi(x)\bigr),
    \end{equation*}
    and $(\Phi^{-1}(y) \cap \mathcal{M}) \cap \mathcal{U}$ is a singleton for every $y \in \mathcal{V} \cap {\Phi(\mathcal{U} \cap \mathcal{M})}$.
\end{definition}

\begin{remark}\label{remark:lip-homeo}
    In our setting, because $F$ is single-valued and smooth on $\widetilde{\mathcal{M}}$, the stratum-restricted strong metric regularity of $F$ on $\widetilde{\mathcal{M}}$ at $(\overline{z}, F(\overline{z}))$ implies that the restriction $F|_{\widetilde{\mathcal{M}}}$ acts locally as a Lipschitz homeomorphism near $\overline{z}$.
\end{remark}

The next lemma enables us to characterize the second order form $\mathcal{Q}_z$ in second order conditions by controlling the behavior of $(v_x, H)$ on the $\alpha\gamma$-blocks, which will be essential throughout this paper for connecting nonsingularity with second order conditions.

\begin{lemma}\label{lemma:identity-Q(v)}
    For $z = (x, y)\in \Xbb\times\Sbb^n$ with an IED $(\alpha, \beta, \gamma, p, q, P, \lambda)$ of $G(z)$. Then, for any $v_x \in \mathbb{X}$ and $H \in \mathbb{S}^n$ satisfying
    \[
        \bigl(-\nabla g(x)^* v_x + \xi_{G(z)}(H)\bigr)_{\alpha\gamma} = 0,
    \]
    the following identity holds:
    \begin{align}\nonumber
        \big\langle \xi_{G(z)}(H),\, H - \xi_{G(z)}(H) \big\rangle
        = 2 \sum_{i \in \alpha} \sum_{j \in \gamma} \frac{-\lambda_j}{\lambda_i}
        \left( P^\top \bigl(\nabla g(x)^* v_x\bigr) P \right)_{ij}^2.
    \end{align}
\end{lemma}

\begin{proof}
    Let $\widetilde{H}=P^{\top}HP$, we have
    \begin{equation}\nonumber
    \begin{aligned}
        &\la  \xi_{G(z)}  H,\,H- \xi_{G(z)}  H\ra \\[3pt]
        =&\Big\la P\begin{bmatrix}
        \widetilde{H}_{\alpha \alpha} & \widetilde{H}_{\alpha \beta} & \Xi_{\alpha \gamma} \circ \widetilde{H}_{\alpha \gamma} \\
        \widetilde{H}_{\alpha \beta}^\top & 0 & 0 \\
        \widetilde{H}_{\alpha \gamma}^\top \circ \Xi_{\alpha \gamma}^\top & 0 & 0
        \end{bmatrix}P^{\top},\,
        P\begin{bmatrix}
        0 & 0 & (E_{\alpha\gamma}-\Xi_{\alpha \gamma}) \circ \widetilde{H}_{\alpha \gamma} \\
        0 & \widetilde{H}_{\beta \beta} & \widetilde{H}_{\beta \gamma} \\
        \widetilde{H}_{\alpha \gamma}^\top \circ (E_{\alpha\gamma}-\Xi_{\alpha \gamma})^\top & \widetilde{H}_{\beta \gamma}^\top & \widetilde{H}_{\gamma \gamma}
        \end{bmatrix}P^{\top} \Big\ra \\[3pt]
        =&2\la \Xi_{\alpha \gamma} \circ \widetilde{H}_{\alpha \gamma},\,(E_{\alpha\gamma}-\Xi_{\alpha \gamma}) \circ \widetilde{H}_{\alpha \gamma}\ra
        =2\sum_{i\in\alpha}\sum_{j\in\gamma}\frac{-\lambda_i\,\lambda_j}{(\lambda_i-\lambda_j)^2}\,\widetilde{H}_{ij}^2 \\[3pt]
        =&2\sum_{i\in\alpha}\sum_{j\in\gamma}\frac{-\lambda_i\,\lambda_j}{(\lambda_i-\lambda_j)^2}
        \left(\frac{\lambda_i-\lambda_j}{\lambda_i}\right)^2
        \bigl(P^{\top}(-\nabla g(x)^* v_x)P\bigr)_{ij}^2
        = 2\sum_{i\in\alpha}\sum_{j\in\gamma}\frac{-\lambda_j}{\lambda_i}
        \bigl(P^{\top}(\nabla g(x)^* v_x)P\bigr)_{ij}^2.
    \end{aligned}
    \end{equation}
    Here, $E_{\alpha\gamma}\in\mathbb{R}^{|\alpha|\times|\gamma|}$ denotes the all-ones matrix, and the fourth equality follows from the condition $(-\nabla g(x)^* v_x + \xi_{G(z)} H)_{\alpha\gamma} = 0$. \ep
\end{proof}

Building on the definitions of the W-SOC and W-SRCQ, and utilizing the manifold differential of the KKT mapping $F$ given in~\eqref{eq:dF_z_block}, we present the main result of this subsection. This theorem provides a comprehensive characterization of the stratum-restricted strong metric regularity of $F$ on $\widetilde{\mathcal{M}}$.

\begin{theorem}\label{thm:WR-dF-nons}
For $z = (x, y) \in \mathbb{X} \times \mathbb{S}^n$, the following statements hold:
\begin{itemize}
    \item[(a)] if the W-SOC and the W-SRCQ hold at $z$, then the manifold differential $dF_{z} : \mathcal{T}_{z}\widetilde{\mathcal{M}} \to \mathbb{X} \times \mathbb{S}^n$ given by~\eqref{eq:dF_z_block} is injective;
    \item[(b)] if $dF_{z}$ is injective, then the W-SRCQ holds at $z$;
    \item[(c)]  if $dF_{z}$ is injective and {the SRCQ and the SONC (Definition~\ref{def:SONC-nonKKT}) hold at $z$}, then the W-SOC holds at $z$.
\end{itemize}
\end{theorem}

\begin{proof}
    Let $(\alpha,\beta,\gamma,p,q,P,\lambda)$ be an IED of $G(z)$.

    \textbf{(a)} Suppose the W-SOC and the W-SRCQ hold at $z$. Assume that $(v_x, v_y) \in \mathcal{T}_{z}\widetilde{\mathcal{M}}$ satisfies $dF_{z}(v_x, v_y) = 0$. By~\eqref{eq:dF_z_block}, we have
    \begin{equation}\label{eq:null-dF}
        \begin{bmatrix}
            \nabla^2_{xx} L(z) - \nabla g(x) \nabla g(x)^* &  \nabla g(x)  \\
            -\nabla g(x)^* &  \xi_{G(z)}
        \end{bmatrix}
        \begin{bmatrix}
            v_x  \\
            H
        \end{bmatrix} = 0,
    \end{equation}
    where $H = \nabla g(x)^*v_x + v_y \in \mathcal{T}_{G(z)}\mathcal{M}$. The second block row implies $-\nabla g(x)^*v_x + \xi_{G(z)}(H) = 0$. Utilizing the explicit formula for $\xi_{G(z)}$ in~\eqref{eq:xi_A} and the matrix $\Xi$ in~\eqref{eq:Xi}, this relation translates to
    \begin{equation*}
        \begin{cases}
            \bigl[P^\top(\nabla g(x)^*v_x)P\bigr]_{ij} = \bigl[P^\top H P\bigr]_{ij}   &  \text{if } i \in \alpha \text{ and } j \in \alpha \cup \beta, \\[3pt]
            \bigl[P^\top(\nabla g(x)^*v_x)P\bigr]_{ij} = 0  & \text{if } i,j \in \beta \cup \gamma, \\[3pt]
            \bigl[P^\top(\nabla g(x)^*v_x)P\bigr]_{ij} = \Xi_{ij}\bigl[P^\top H P\bigr]_{ij} & \text{if } i \in \alpha \text{ and } j \in \gamma.
        \end{cases}
    \end{equation*}
    By \eqref{eq:def-appl-nonKKT}, this immediately guarantees $v_x \in \operatorname{appl}(z)$. Taking the inner product of the first block row of~\eqref{eq:null-dF} with $v_x$ yields
    \begin{equation}\label{eq:null-dF-W-SOC}
    \begin{aligned}
        0 &= \langle v_x, \nabla^2_{xx} L(z)v_x - \nabla g(x) \nabla g(x)^*v_x + \nabla g(x)H \rangle \\
        &= \langle v_x, \nabla^2_{xx} L(z)v_x \rangle + \langle \nabla g(x)^*v_x, -\nabla g(x)^*v_x + H \rangle \\
        & = \langle v_x, \nabla^2_{xx} L(z)v_x \rangle + \langle\xi_{G(z)}(H), H-\xi_{G(z)}(H) \rangle\\
        &= \mQ_{z}(v_x),
    \end{aligned}
    \end{equation}
    where the last equation follows from Lemma \ref{lemma:identity-Q(v)} as $-\nabla g(x)^*v_x + \xi_{G(z)}(H) = 0$. Because the W-SOC holds at $z$, this equality strictly enforces $v_x = 0$. Consequently, the system~\eqref{eq:null-dF} collapses to
    \begin{equation} \label{eq:null-dF-W-SRCQ}
        \begin{cases}
            \nabla g(x) H = 0, \\
            -\xi_{G(z)}(H) = 0.
        \end{cases}
    \end{equation}
    Together with the structure of $\xi_{G(z)}$ in~\eqref{eq:xi_A} and the W-SRCQ condition at $z$, this implies $H = 0$. Thus, $dF_{z}$ is injective.

    \textbf{(b)} Suppose $dF_{z}$ is injective. If the W-SRCQ fails at $z$, Definition~\ref{def:W-SRCQ-nonKKT} ensures the existence of a nonzero matrix $H \in \mathcal{T}_{G(z)}\mathcal{M}$ satisfying
    \begin{equation} \label{eq:W-SRCQ-violated}
        \begin{cases}
            \nabla g(x) H = 0, \\
            -\xi_{G(z)}(H) = 0.
        \end{cases}
    \end{equation}
    This directly implies that $(0,H) \in \mathbb{X} \times \mathcal{T}_{G(z)}\mathcal{M}$ is a nontrivial solution to~\eqref{eq:null-dF}, contradicting the injectivity of $dF_{z}$.

    {
     \textbf{(c)} Suppose $dF_z$ is injective and the SRCQ and the SONC (Definition~\ref{def:SONC-nonKKT}) hold at $z$. Set $A=G(z)$. The SRCQ and the definitions give
    \[
        \operatorname{span}\mathcal C(z)=\operatorname{app}(z)
        =\bigl(\nabla g(x)(\ker\xi_A)\bigr)^\perp.
    \]
    Let $\mathfrak q_z$ be the symmetric bilinear form associated with $\mQ_z$. Define $B(v)=B_0(\nabla g(x)^*v)$, where
    \[
        B_0(H)=P\begin{bmatrix}
            \widetilde H_{\alpha\alpha}&\widetilde H_{\alpha\beta}&\Xi^{-1}_{\alpha\gamma}\circ\widetilde H_{\alpha\gamma}\\
            \widetilde H_{\alpha\beta}^{\top}&0&0\\
            \widetilde H_{\alpha\gamma}^{\top}\circ(\Xi^{-1}_{\alpha\gamma})^{\top}&0&0
        \end{bmatrix}P^{\top},
    \]
    $\widetilde H=P^{\top}HP$, and $\Xi^{-1}_{\alpha\gamma}$ denotes the elementwise reciprocal. For $v\in\operatorname{appl}(z)$, direct calculation gives $\xi_A(B(v))=\nabla g(x)^*v$ and
    \begin{equation}\label{eq:IB}
        \left\langle d,\nabla^2_{xx}L(z)v+\nabla g(x)\bigl(B(v)-\nabla g(x)^*v\bigr)\right\rangle
        =\mathfrak q_z(d,v)\quad\forall\,d\in\mathbb X.
    \end{equation}
    If the W-SOC fails, some $0\ne v\in\operatorname{appl}(z)$ satisfies $\mQ_z(v)=0$. Since $d+tv\in\mathcal C(z)$ for $d\in\mathcal C(z)$ and $t\in\mathbb R$, the SONC implies $\mathfrak q_z(d,v)=0$ on $\mathcal C(z)$, hence on $\operatorname{app}(z)$. By~\eqref{eq:IB}, the vector $r=\nabla^2_{xx}L(z)v+\nabla g(x)(B(v)-\nabla g(x)^*v)$ belongs to $\nabla g(x)(\ker\xi_A)$. Choose $K\in\ker\xi_A$ with $\nabla g(x)K=-r$. Then $(v,B(v)+K)$ is a nonzero solution of~\eqref{eq:null-dF}, contradicting injectivity. Therefore the W-SOC holds. \ep
    \par}
\end{proof}

When $\overline{z} = (\overline{x}, \overline{y})$ is a KKT pair of~\eqref{prog:SDP}, the following result characterizes the stratum-restricted strong metric regularity of $F$ at $\overline{z}$.

\begin{theorem}\label{thm:WR-dF-nons-KKT}
Let $\overline{z} = (\overline{x}, \overline{y})$ be a KKT pair of~\eqref{prog:SDP} with an IED $(\alpha, \beta, \gamma, p, q, \overline{P}, \overline{\lambda})$ of $G(\overline{z})$. The following two statements are equivalent:
\begin{itemize}
    \item[(a)] the manifold differential $dF_{\overline{z}}$ is injective;
    \item[(b)] the KKT mapping $F$ defined by~\eqref{eq:kkt-mapping} is stratum-restricted strongly metrically regular on $\widetilde{\mathcal{M}}$ at $(\overline{z}, F(\overline{z}))$.
\end{itemize}
 Moreover, if {the SRCQ and the SONC (Definition~\ref{def:SONC-nonKKT}) hold at $\overline{z}$}, they are further equivalent to:
\begin{itemize}
    \item[(c)] The W-SOC and the W-SRCQ hold at $\overline{z}$.
\end{itemize}
\end{theorem}

\begin{proof}
{\bf (a)$\Rightarrow$(b)}: By \cite[Proposition~5.18]{lee2012introduction}, there exists a neighborhood $\mathcal{U} \subseteq \tMcal$ of $\overline{z}$ and a submanifold $\mathcal{U}'$ containing $F(\overline{z})$, embedded in $\Xbb \times \Sbb^n$, such that $F$ is a diffeomorphism from $\mathcal{U}$ onto $\mathcal{U}'$. Denote the restriction of $F$ to $\mathcal{U}$ by $F_{\mathcal{U}}$. Consequently, $(F_{\mathcal{U}})^{-1}$ is single-valued and smooth on $\mathcal{U}'$, which yields (b).

{\bf (b)$\Rightarrow$(a)}:
By Remark \ref{remark:lip-homeo}, statement (b) implies there is a neighborhood, still denoted by $\Ucal\subseteq\tMcal$, of $\overline z$ that $F_{\Ucal}$ is bijective and $F_{\Ucal}^{-1}$ is Lipschitz. We show that $dF_{\overline z}$ is injective by contradiction. Suppose not. Then there exists $0 \neq v \in \Tcal_{\overline z} \widetilde{\Mcal}$ such that $dF_{\overline z}(v)=0$. Choose a smooth curve $c:(-\varepsilon,\varepsilon)\to \Ucal \subseteq \widetilde{\Mcal}$ such that $c(0)=\overline z$ and $c'(0)=v$. Since $F|_{\widetilde{\Mcal}}$ is smooth, we have
\[
F(c(t))-F(\overline z)=t\,dF_{\overline z}(v)+o(t)=o(t).
\]
So
\begin{equation}\label{eq:ct-barz}
    \|c(t)-\overline z\| = \|F_{\Ucal}^{-1}(F(c(t)))- F_{\Ucal}^{-1}( F(\overline z) ) \|
    \le L\,\|F(c(t))-F(\overline z)\|
    = o(t),
\end{equation}
where the inequality follows from that $F_{\Ucal}^{-1}$ is Lipschitz.

However, the smoothness of $c(t)$ implies that \( c(t)-\overline z=t\,v+o(t) \), which contradicts \eqref{eq:ct-barz} as $v\neq 0$. Therefore, $dF_{\overline z}$ is injective. 

 Finally, the equivalence to (c) under {the SRCQ and the SONC} follows directly from Theorem~\ref{thm:WR-dF-nons}.
\ep
\end{proof}

Next, by Theorem~\ref{thm:WR-dF-nons-KKT}, we immediately obtain the stratum-restricted analogue to the classical local error bound (Definition~\ref{def:local error bound}).

\begin{corollary}\label{coro:REB}
    Let $\overline{z}$ be a KKT pair of~\eqref{prog:SDP}. If the W-SOC and the W-SRCQ hold at $\overline{z}$, then there exist a constant $C_{\mathrm{REB}} > 0$ and a neighborhood $\mathcal{U} \subseteq \widetilde{\mathcal{M}}$ of $\overline{z}$ such that $(F|_{\widetilde{\mathcal{M}}})^{-1}(0) \cap \mathcal{U} = \{\overline{z}\}$ and
    \begin{equation}\label{eq:res-local-EB}
        \|F(z)\| \ge C_{\mathrm{REB}} \|z - \overline{z}\|, \quad \forall\, z \in \mathcal{U}.
    \end{equation}
    We refer to~\eqref{eq:res-local-EB} as the \emph{stratum-restricted local error bound}.
\end{corollary}

At KKT pairs, the examples provided in~\cite{feng2025quadratically} illustrate that the W-SOC and W-SRCQ are significantly weaker than classical regularity assumptions. Specifically,~\cite[Example~2]{feng2025quadratically} demonstrates that even when the W-SOC and W-SRCQ hold, every matrix within the Bouligand generalized Jacobian may still be singular. Furthermore,~\cite[Examples~3--7]{feng2025quadratically} detail several scenarios where classical strong regularity fails completely while the W-SOC and W-SRCQ remain valid.

The following simple semidefinite program highlights that the W-SOC and W-SRCQ are strictly weaker than problem-level regularity conditions conventionally imposed in the literature. Crucially, this example also demonstrates that while the standard local error bound condition (Definition~\ref{def:local error bound}) fails, the stratum-restricted local error bound~\eqref{eq:res-local-EB} remains active.

\begin{example}\label{example:9.2}
    Consider the following instance of~\eqref{prog:SDP}:
    \begin{equation}\label{prog:example-1}
        \begin{aligned}
            \min_{x\in \mathbb{R}^5} \quad & x_1 \\
            \mathrm{s.t.} \quad &
            \begin{bmatrix}
                1 & 0 & 0 & x_4+x_5\\
                0 & x_4-x_5 & 0 & x_3 \\
                0 & 0 & 0 &x_2 \\
                x_4+x_5 & x_3 & x_2 & x_1
            \end{bmatrix} \in \mathbb{S}_{+}^4.
        \end{aligned}
    \end{equation}
    Clearly, $\overline{x} = (0,0,0,0,0)^\top$ is an optimal solution with a corresponding Lagrange multiplier $\overline{y} = \operatorname{Diag}(0,0,0,-1)$. Direct verification confirms that both the W-SOC and W-SRCQ hold at $(\overline{x}, \overline{y})$, whereas the SOSC and SRCQ fail. Consequently, neither the S-SOSC nor constraint nondegeneracy holds. As shown in~\cite[Proposition~5]{feng2025quadratically}, this causes the standard generalized Jacobians $\mathcal{U}_0$ and $\mathcal{U}_{\mathcal{I}}$ of $F$ to be degenerate, severely compromising the stability of standard semismooth Newton methods.
    \begingroup
     Moreover, consider the perturbation trajectory
    \[
        y(t)=\begin{bmatrix}
        0&0&0&0\\
        0&-t^3&t^2&0\\
        0&t^2&-t&0\\
        0&0&0&-1
        \end{bmatrix},\qquad t>0.
    \]
    It converges to $\overline y$, and direct calculation gives $\|F(\overline x,y(t))\|=\sqrt2\,t^3$ and $\operatorname{dist}((\overline x,y(t)),F^{-1}(0))=\sqrt{t^6+2t^4}$. Indeed, every KKT multiplier has $y_{22}=y_{23}=0$, and the KKT pair $(\overline x,\operatorname{Diag}(0,0,-t,-1))$ attains this distance. The distance-to-residual ratio therefore tends to infinity, proving failure of the classical local error bound at $(\overline x,\overline y)$.
    \par\endgroup

     Conversely, let $\overline{z} = (\overline{x}, \overline{y})$. By the differential formula~\eqref{eq:dF_z(v)}, any tangent vector $v \in \mathcal{T}_{\overline{z}}\widetilde{\mathcal{M}}$ satisfying $dF_{\overline{z}}(v) = 0$ must satisfy the coupled system:
    \begin{equation*}
        \begin{bmatrix}
            -v_1 + H_{44} \\
            -2v_2 + 2H_{34} \\
            -2v_3 + 2H_{24} \\
            -3v_4 - v_5 + H_{22} + 2H_{14} \\
            -v_4 - 3v_5 - H_{22} + 2H_{14}
        \end{bmatrix} = 0
        \quad \text{and} \quad
        \begin{bmatrix}
            H_{11} & H_{12} & H_{13} & -v_4 - v_5 + \frac{1}{2}H_{14} \\
            H_{21} & -v_4 + v_5 & 0 & -v_3 \\
            H_{31} & 0 & 0 & -v_2 \\
            -v_4 - v_5 + \frac{1}{2}H_{14} & -v_3 & -v_2 & -v_1
        \end{bmatrix} = 0,
    \end{equation*}
    where we utilized the isomorphism $v \mapsto (v_x, H)$ from Lemma~\ref{lemma:TzMt}, with $v_x = (v_1, v_2, v_3, v_4, v_5)^\top$ and $H = (H_{ij}) \in \mathcal{T}_{G(\overline{z})}\mathcal{M}$. Leveraging the geometric characterization $\mathcal{T}_{G(\overline{z})}\mathcal{M} = \{ H \in \mathbb{S}^4 \mid H_{22} = H_{23} = H_{33} = 0 \}$, one readily verifies that the trivial solution $v=0$ is unique. Thus, $dF_{\overline{z}}$ is strictly injective, which, by~\cite[Proposition~5.18]{lee2012introduction} and Corollary~\ref{coro:REB}, rigorously guarantees the validity of the stratum-restricted local error bound.
\end{example}

By continuity, the injectivity of $dF_z$ persists within a local neighborhood. This localized robustness establishes two critical properties for our subsequent algorithmic analysis: the uniform boundedness of the inverse of the associated normal matrix, and the local uniqueness of the solution along the stratum.

\begin{corollary}\label{coro:WR-uniform-bounded}
Let $(\alpha,\beta,\gamma,p,q,P,\lambda)$ be an IED of $G(z)$ at $z = (x, y) \in \mathbb{X} \times \mathbb{S}^n$. If the W-SOC and the W-SRCQ hold at $z$, there exists a neighborhood $\mathcal{U} \subseteq \widetilde{\mathcal{M}}$ of $z$ such that for all $z' \in \mathcal{U}$,
\begin{itemize}
    \item[(a)] the manifold differential $dF_{z'}$ is injective;
    \item[(b)] there exists a constant $C_{\mathrm{ub}} > 0$ depending only on $z$ such that $\|(dF_{z'}^* dF_{z'})^{-1}\| \le C_{\mathrm{ub}}$.
\end{itemize}
\end{corollary}

\begin{proof}
    \textbf{(a):} Because $z' \in \widetilde{\mathcal{M}}$, the active manifold $\mathcal{M}$ at $z'$ is identical to that at $z$, preserving the index sets $\alpha$, $\beta$, and $\gamma$. By Theorem~\ref{thm:WR-dF-nons}, $dF_{z}$ is injective. A standard continuity argument guarantees that $dF_{z'}$ remains injective for all $z' \in \widetilde{\mathcal{M}}$ sufficiently close to $z$.

    \textbf{(b):} The injectivity of $dF_{z}$ guarantees that the self-adjoint operator $dF_{z}^* dF_{z}$ is invertible. By continuity, the inverse $(dF_{z'}^* dF_{z'})^{-1}$ exists and its operator norm remains uniformly bounded over a sufficiently small neighborhood $\mathcal{U} \subseteq \widetilde{\mathcal{M}}$. \ep
\end{proof}

The Clarke generalized Jacobian~\cite{clarke1990optimization} plays a central role in the convergence theory of classical semismooth Newton methods. We now demonstrate that at a KKT pair $\overline{z}$, the failure of the manifold differential $dF_{\overline{z}}$ to be injective renders \emph{every} matrix in the Clarke Jacobian $\partial_C F(\overline{x}, \overline{y})$ singular. This confirms that stratum-restricted strong regularity is an essential prerequisite for the superlinear convergence of such methods.

To formalize this, we recall the structural characterization of $\partial_C F$. By~\cite[Lemma~11]{pang2003semismooth}, given an IED $(\alpha,\beta,\gamma,p,q, \overline{P}, \overline{\lambda})$ of $A \in \mathbb{S}^n$, an operator $\mathcal{V} \in \partial_C \Pi_{\mathbb{S}^n_+}(A)$ if and only if there exists $\mathcal{W} \in \partial_C\Pi_{\mathbb{S}^{|\beta|}_+}(0)$ such that for any $H \in \mathbb{S}^n$,
\begin{equation}\label{eq:VH}
\mathcal{V}(H) = \overline{P}
\begin{bmatrix}
    \widetilde{H}_{\alpha\alpha} & \widetilde{H}_{\alpha\beta} & \overline{\Xi}_{\alpha\gamma}\circ\widetilde{H}_{\alpha\gamma} \\
    \widetilde{H}_{\beta\alpha} & \mathcal{W}(\widetilde{H}_{\beta\beta}) & 0 \\
    \overline{\Xi}_{\gamma\alpha}\circ\widetilde{H}_{\gamma\alpha} & 0 & 0
\end{bmatrix}
\overline{P}^\top,
\end{equation}
where $\overline{\Xi}$ is defined via~\eqref{eq:Xi} and $\widetilde{H} = \overline{P}^\top H \overline{P}$. Furthermore,~\cite[Lemma~2.1]{sun2006strong} states that $\overline{\mathcal{U}} \in \partial_C F(\overline{x},\overline{y})$ if and only if there exists $\overline{\mathcal{V}} \in \partial_C\Pi_{\mathbb{S}^n_+}(g(\overline{x}) + \overline{y})$ such that for any vector $v = (v_x,v_y) \in \mathbb{X} \times \mathbb{S}^n$,
\begin{equation}\label{eq:UH}
    \overline{\mathcal{U}}v =
    \begin{bmatrix}
    \nabla_{xx}^2L(\overline{x},\overline{y})v_x + \nabla g(\overline{x})v_y \\
    -\nabla g(\overline{x})^*v_x + \overline{\mathcal{V}}\bigl(\nabla g(\overline{x})^*v_x + v_y\bigr)
    \end{bmatrix}.
\end{equation}

\begin{proposition}\label{prop:dF_pCF}
For $z = (x,y) \in \mathbb{X} \times \mathbb{S}^n$, if the manifold differential
$dF_z$ is not injective, then there exists a nonzero vector
$v \in \mathcal{T}_z\widetilde{\mathcal{M}}$ such that
\[
    \mathcal{U}v = 0
    \qquad \mbox{for every $\mathcal{U} \in \partial_C F(x,y)$.} 
\]
\end{proposition}

\begin{proof}
    Because $dF_z$ is not injective, there exists a nonzero tangent vector
    $v = (v_x,v_y) \in \mathcal{T}_z\widetilde{\mathcal{M}}$ such that
    $dF_z v = 0$. By~\eqref{eq:UH}, for any arbitrary element
    $\mathcal{U} \in \partial_C F(x,y)$, there exists a corresponding
    $\mathcal{V} \in \partial_C\Pi_{\mathbb{S}^n_+}(G(z))$ such that
    \begin{align*}
        \mathcal{U}v
        &= \begin{bmatrix}
            \nabla_{xx}^2L(x,y)v_x + \nabla g(x)v_y \\
            -\nabla g(x)^*v_x + \mathcal{V}\bigl(\nabla g(x)^*v_x + v_y\bigr)
        \end{bmatrix} \\
        &= \begin{bmatrix}
            \bigl(\nabla_{xx}^2L(x,y) - \nabla g(x)\nabla g(x)^*\bigr)v_x + \nabla g(x)H \\
            -\nabla g(x)^*v_x + \mathcal{V}(H)
        \end{bmatrix},
    \end{align*}
    where $H = \nabla g(x)^*v_x + v_y$. Given an IED
    $(\alpha,\beta,\gamma,p,q,P,\lambda)$ of $G(z)$,~\eqref{eq:VH}
    implies the existence of
    $\mathcal{W} \in \partial_C\Pi_{\mathbb{S}^{|\beta|}_+}(0)$ such that
    \begin{equation*}
        \mathcal{V}(H) = P
        \begin{bmatrix}
        \widetilde{H}_{\alpha\alpha} & \widetilde{H}_{\alpha\beta} & \Xi_{\alpha\gamma}\circ\widetilde{H}_{\alpha\gamma} \\
        \widetilde{H}_{\beta\alpha} & \mathcal{W}(\widetilde{H}_{\beta\beta}) & 0 \\
        \Xi_{\gamma\alpha}\circ\widetilde{H}_{\gamma\alpha} & 0 & 0
        \end{bmatrix}
        P^\top,
    \end{equation*}
    with $\widetilde{H} = P^\top H P$. Since
    $v \in \mathcal{T}_z \widetilde{\mathcal{M}}$, Lemma~\ref{lemma:TzMt}
    and Proposition~\ref{prop:TAM} imply that
    $\widetilde{H}_{\beta\beta} = 0$. Consequently, the action of
    $\mathcal{V}$ simplifies identically to the manifold differential operator:
    \begin{equation*}
        \mathcal{V}(H) = \xi_{G(z)}(H).
    \end{equation*}
    Invoking~\eqref{eq:dF_z_block}, we conclude that
    \begin{equation*}
        \mathcal{U}(v_x,v_y) = dF_z(v_x,v_y) = 0.
    \end{equation*}
    Thus, the same nonzero vector $v$ lies in the null space of every
    $\mathcal{U} \in \partial_C F(x,y)$, completing the proof. \ep
\end{proof}

 {Together with the injectivity characterization above, Proposition~\ref{prop:dF_pCF} connects the Clarke generalized Jacobian with the weak second-order conditions. Under the SRCQ and the SONC, the W-SOC used here, which is formulated by non-vanishing, is equivalent to strict positivity on $\operatorname{appl}(\overline{z})\setminus\{0\}$. This yields the following equivalence.}

\begin{corollary}\label{coro:WW_pCF}
     At a KKT pair $\overline{z} = (\overline{x},\overline{y})$ of~\eqref{prog:SDP}, if {the SRCQ and the SONC (Definition~\ref{def:SONC-nonKKT}) hold at $\overline{z}$}, then the following statements are equivalent:
    \begin{itemize}
        \item[(a)] the W-SOC and the W-SRCQ hold at $\overline{z}$;
        \item[(b)] the manifold differential $dF_{\overline{z}}$ is injective;
        \item[(c)] the KKT mapping $F$ is stratum-restricted strongly metrically regular on $\widetilde{\mathcal{M}}$ at $(\overline{z}, F(\overline{z}))$;
        \item[(d)] there exists a nonsingular element in $\partial_C F(\overline{z})$.
    \end{itemize}
     {For convex instances of~\eqref{prog:SDP}, the above equivalence holds without the additional SRCQ assumption.}
\end{corollary}

{
\begin{proof}
 The general case follows from Theorems~\ref{thm:WR-dF-nons} and~\ref{thm:WR-dF-nons-KKT}, Proposition~\ref{prop:dF_pCF}, and~\cite[Proposition~4]{feng2025quadratically}. For convex instances, $\mQ_{\overline z}\ge0$ on $\mathbb X$ by Remark~\ref{remark:convex-W-SONC}. If $v\in\operatorname{appl}(\overline z)$ and $\mQ_{\overline z}(v)=0$, then $\nabla^2_{xx}L(\overline z)v=0$ and the $\alpha\gamma$ block of $P^{\top}(\nabla g(\overline x)^*v)P$ vanishes. Thus $(v,\nabla g(\overline x)^*v)$ solves~\eqref{eq:null-dF}, so injectivity forces $v=0$. Hence (b)$\Rightarrow$(a) also holds without SRCQ; the other implications are unchanged. \ep
\end{proof}
\par}

\begin{remark}\label{remark:dF_pCF}
     Proposition~\ref{prop:dF_pCF} and Corollary~\ref{coro:WW_pCF} yield an important algorithmic consequence. At a KKT pair $\overline{z}$, if {the SRCQ and the SONC hold or~\eqref{prog:SDP} is convex}, but the W-SOC and the W-SRCQ do not hold simultaneously, then the stratum-restricted strong metric regularity fails. As a result, \emph{every} element of the Clarke generalized Jacobian $\partial_C F(\overline{z})$ is singular. Therefore, any semismooth Newton system built from an element of $\partial_C F(\overline{z})$ is necessarily degenerate. This shows that the stratum-restricted strong metric regularity is a necessary condition for Newton-type methods to achieve superlinear convergence.
\end{remark}

\subsubsection{Transversality and stability}\label{subsubsec: trans}

Having established the exact correspondence between the W-SOC, the W-SRCQ, and the stratum-restricted strong metric regularity of the KKT mapping of the NLSDP~\eqref{prog:SDP}, we now explore the underlying geometric properties of the W-SRCQ on a fixed stratum.

For the NLSDP~\eqref{prog:SDP}, it is well-known that constraint nondegeneracy holds at a feasible point $x$ if and only if the constraint mapping $g$ intersects the manifold $\Mcal_{p,0}$ transversally (see~\cite[(4.182)]{bonnans2013perturbation}). In the same spirit, we geometrically interpret the W-SRCQ (Definition~\ref{def:W-SRCQ-nonKKT}) in terms of a transversality relation and investigate its stability within a prescribed stratum.

To cast the W-SRCQ as a transversality condition (Definition~\ref{def:transversal}), we first define a parametric family of submanifolds over a local neighborhood. For each $z = (x,y) \in \mathbb{X} \times \mathbb{S}^n$, suppose $z \in \tMpq$, we define
\begin{equation}\label{eqn:def-of-Uz}
    \Ucal_z := \Mpo \cap \Bcal \big( {\Pi_{\mathbb{S}^n_+}(G(z))}, \frac{\delta(z)}{2} \big)
\end{equation}
 where $\delta(z)$ is the the minimal non-vanishing eigenvalue modulus of $G(z)$ defined by
\begin{equation}\label{eqn: minieigen}
    \delta(z) := \min \bigl\{ |\lambda_i(G(z))| \mid \lambda_i(G(z)) \neq 0 \bigr\}{\quad\text{if }G(z)\neq0},
\end{equation}
{while for $G(z)=0$ we choose any $\delta(z)>0$. The} open ball $\Bcal(A,r)$ is defined by $\Bcal(A,r) := \{ A' \in \Sbb^n \mid \| A-A' \|_2 < r \}$. So $\Ucal_z$ is actually an open submanifold of $\Mpo$. Then we define
\begin{equation}\label{eq:N_z}
    \mathcal{N}_{z} := \Ucal_z + \Pi_{\mathbb{S}^n_-}\bigl(G(z)\bigr) - y,
\end{equation}
where $\mathbb{S}^n_-$ denotes the cone of negative semidefinite matrices. The geometric interpretation of classical constraint nondegeneracy (Definition~\ref{def:CN-nonKKT}) via $\mathcal{N}_z$ is established in the following proposition. After that, we will construct our target manifolds based on $\Ncal_z$ to give the geometric interpretation of W-SRCQ.

\begin{proposition}\label{prop:trans=CN}
    Let $(\alpha,\beta,\gamma,p,q, P, \lambda)$ be an IED of $G(z)$ at $z=(x, y) \in \mathbb{X}\times\mathbb{S}^n$. The following statements are equivalent:
    \begin{itemize}
        \item[(a)] $g \pitchfork_{x} \mathcal{N}_{z}$ in $\mathbb{S}^n$;
        \item[(b)] the constraint nondegeneracy holds at $z = (x,y)$.
    \end{itemize}
\end{proposition}
\begin{proof}
   Directly computing the tangent space $\mathcal{T}_{g(x)} \mathcal{N}_z$ yields
    \begin{align*}
        \mathcal{T}_{g(x)} \mathcal{N}_z
        &= \mathcal{T}_{\Pi_{\Sbb^n_+}(G(z))} \Ucal_z  = \mathcal{T}_{\Pi_{\Sbb^n_+}(G(z))} \Mpo  \\
        &= \bigl\{ P B P^\top \mid B \in \mathbb{S}^n,\ B_{\beta \beta} = 0,\ B_{\beta \gamma} = 0,\ B_{\gamma\gamma} = 0 \bigr\}.
    \end{align*}
    Therefore, the transversality condition $g \pitchfork_x \mathcal{N}_z$ evaluates exactly to
    \begin{equation*}
        g'(x) \mathbb{X} + \bigl\{ P B P^\top \mid B \in \mathbb{S}^n,\ B_{\beta \beta} = 0,\ B_{\beta \gamma} = 0,\ B_{\gamma\gamma} = 0 \bigr\} = \mathbb{S}^n,
    \end{equation*}
    which recovers Definition~\ref{def:CN-nonKKT}.
    \ep
\end{proof}

\begin{remark}
    When $z = (x, y)$ is a KKT pair of~\eqref{prog:SDP}, one can verify that $\mathcal{N}_{z} = \Ucal_z$, which is an open set in $\Mpo$. In this case, the condition $g \pitchfork_{x} \Ncal_z$ $\mathbb{S}^n$ or equivalently, $g \pitchfork_{x} \Mpo$ in $\mathbb{S}^n$ coincides with the classical constraint nondegeneracy condition formulated in~\cite[Definition~5.70]{bonnans2013perturbation}.
\end{remark}

The subsequent lemma demonstrates that for $z \in \widetilde{\mathcal{M}}_{p,q}$, the set $\mathcal{N}_z$ is embedded in the translated stratum $\mathcal{M}_{p,q} - y$ near $g(x)$. The proof is deferred to Appendix~\ref{appendix:transverse}.
\begin{lemma}\label{lemma:N_z-in-Mpq}
    Let $z \in \widetilde{\mathcal{M}}_{p,q}$ and set $A := G(z) \in \mathcal{M}_{p,q}$. Denote
    \begin{equation*}
        A_+ := \Pi_{\mathbb{S}^n_+}(A) \in \mathcal{M}_{p,0} \quad \text{and} \quad A_- := \Pi_{\mathbb{S}^n_-}(A) \in \mathcal{M}_{0,q},
    \end{equation*}
    so that $A = A_+ + A_-$. Then $\Ucal_z$ defined in \eqref{eqn:def-of-Uz} by
    \begin{equation}\nonumber
        \Ucal_z := \Mpo \cap \Bcal \big( {\Pi_{\mathbb{S}^n_+}(G(z))}, \frac{\delta(z)}{2} \big)
    \end{equation}
    is an open neighborhood of $A_+$ in $\mathcal{M}_{p,0}$ such that $\Ucal_z + A_-$ is smoothly embedded in $\mathcal{M}_{p,q}$.
\end{lemma}

To encompass the W-SRCQ, we expand $\mathcal{N}_z$ by attaching the normal bundle of $\Mpq$. For $z = (x,y) \in \widetilde{\mathcal{M}}_{p,q}$ and each point $A \in \mathcal{U} - y \subseteq \mathcal{N}_z \cap (\mathcal{M}_{p,q} - y)$, the normal space to $\mathcal{M}_{p,q} - y$ in $\mathbb{S}^n$ is given by
\begin{equation*}
    \mathcal{N}_A (\mathcal{M}_{p,q}-y) = \bigl\{ H \in \mathbb{S}^n \mid (P^{\top}HP)_{ij} = 0 \text{ for all } (i,j) \notin \beta \times \beta \bigr\},
\end{equation*}
where $(\alpha,\beta,\gamma,p,q,P,\lambda)$ is an IED of $A+y$.  We define the corresponding geometric union:
\begin{equation}\label{eq:DN_z}
    \mathcal{D}\mathcal{N}_z := {\bigl\{ A + W \mid A \in \mathcal{N}_z,\ W \in \mathcal{N}_A (\mathcal{M}_{p,q}-y),\ \|W\|_2<\delta(z)/4 \bigr\}} \subseteq \mathbb{S}^n.
\end{equation}
It turns out that $\Dcal\Ncal_z$ forms a smooth manifold. The main result of this subsection formally links the W-SRCQ to transversality against the family $\{\mathcal{D}\mathcal{N}_z\}$.

\begin{theorem}\label{thm:trans=WSRCQ}
    For the NLSDP~\eqref{prog:SDP}, let $z = (x, y) \in \mathbb{X} \times \mathbb{S}^n$ possess an IED $(\alpha, \beta, \gamma, p, q, P, \lambda)$ of $G(z)$ associated with $\mathcal{M}_{p,q}$. The following statements are equivalent:
    \begin{itemize}
        \item[(a)] $g \pitchfork_{x} \mathcal{D}\mathcal{N}_{z}$ in $\mathbb{S}^n$;
        \item[(b)] the W-SRCQ (Definition~\ref{def:W-SRCQ-nonKKT}) holds at $z = (x,y)$.
    \end{itemize}
\end{theorem}

For clarity, we informally outline why the transversality condition reduces to the W-SRCQ, deferring the rigorous differential-topological arguments (including the verification that $\mathcal{D}\mathcal{N}_z$ constitutes a smooth manifold) to Appendix~\ref{appendix:transverse}. By construction, $\mathcal{D}\mathcal{N}_z$ attaches the normal space $\mathcal{N}_A\mathcal{M}_{p,q}$ to each point $A \in \mathcal{N}_z$. At the base point $g(x) = G(z) - y$, the tangent space splits as
\begin{align*}
    \mathcal{T}_{g(x)} \mathcal{D}\mathcal{N}_{z} &= \mathcal{T}_{G(z)}(\mathcal{D}\mathcal{N}_{z}+y) = \mathcal{T}_{G(z)}(\mathcal{N}_{z}+y) \oplus \mathcal{N}_{G(z)}\mathcal{M}_{p,q} \\
    &= \bigl\{ PBP^{\top} \mid B \in \mathbb{S}^n,\ B_{\beta\gamma} = 0,\ B_{\gamma\gamma} = 0 \bigr\}.
\end{align*}
Consequently, the transversality relation $g \pitchfork_{x} \mathcal{D}\mathcal{N}_{z}$ in $\mathbb{S}^n$ algebraically evaluates to
\begin{equation}\label{eq:transv-to-WSRCQ}
    dg_x(\mathbb{X}) + \bigl\{ PBP^{\top} \mid B \in \mathbb{S}^n,\ B_{\beta\gamma} = 0,\ B_{\gamma\gamma} = 0 \bigr\} = \mathbb{S}^n,
\end{equation}
which is exactly the W-SRCQ condition.

Figure~\ref{fig:trans} conceptually visualizes this transversality mechanism. The ambient space depicts $\mathbb{S}^n$. The light blue plane represents the translated stratum $\mathcal{M}_{p,q} - y$, while the dark blue line denotes the embedded submanifold $\mathcal{N}_z \subseteq \mathcal{M}_{p,q} - y$ (see~\eqref{eq:N_z}). The yellow vertical extensions illustrate $\mathcal{D}\mathcal{N}_z$ (see~\eqref{eq:DN_z}), formed by restricting the normal bundle of $\mathcal{M}_{p,q} - y$ in $\Sbb^n$ to its submanifold $\mathcal{N}_z$. The red and green planes represent the images of two distinct constraint mappings, $g_1$ and $g_2$. In the left subfigure, $g_1 \pitchfork_x \mathcal{N}_z$ guarantees that both classical constraint nondegeneracy and the W-SRCQ hold at $z$. In the right subfigure, $g_2 \pitchfork_x \mathcal{D}\mathcal{N}_z$ holds, yet $g_2$ fails to intersect $\mathcal{N}_z$ transversally. This geometry captures a scenario where the W-SRCQ is satisfied despite the failure of constraint nondegeneracy.

\begin{figure}[htbp]
    \centering
    \includegraphics[width=0.8\textwidth]{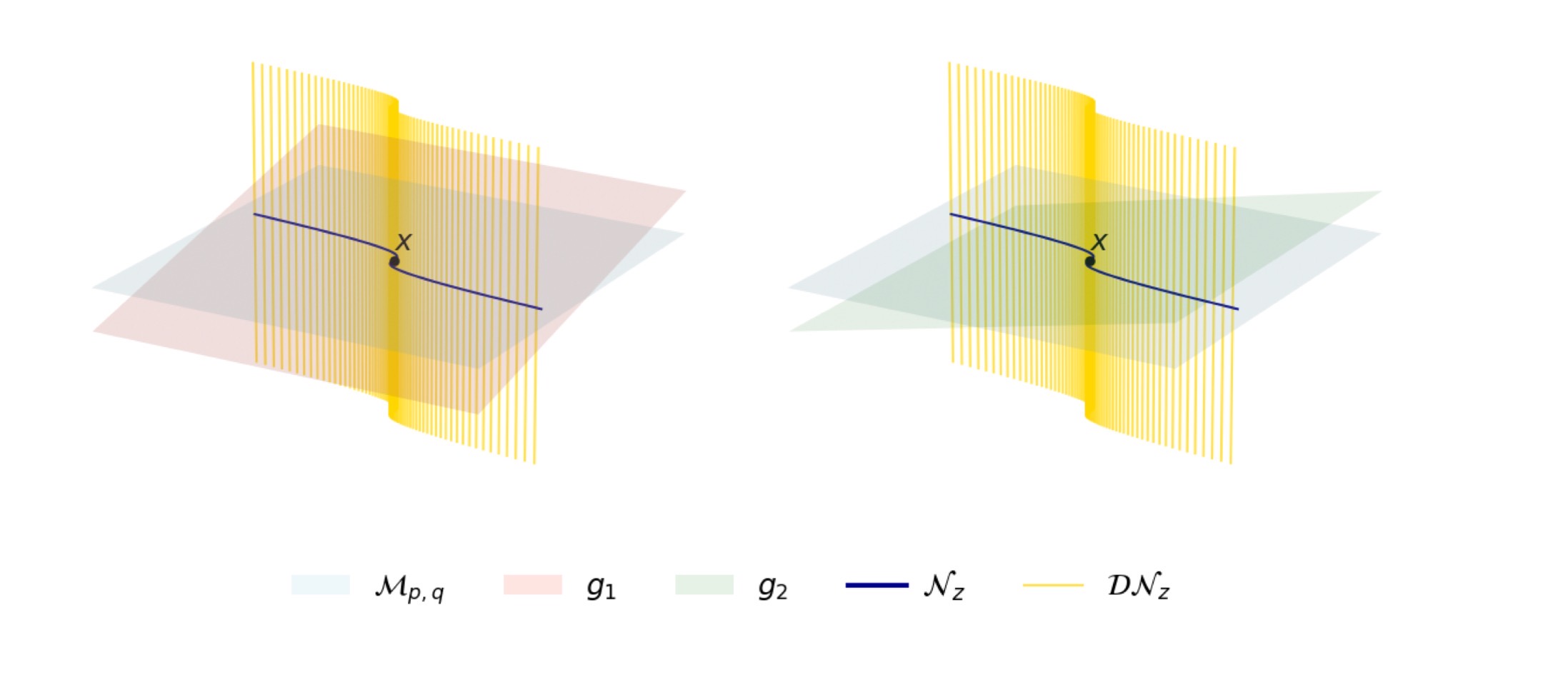}
    \caption{Geometric visualization of transversality: W-SRCQ vs. Constraint Nondegeneracy.}
    \label{fig:trans}
\end{figure}

The following corollary formalizes the geometric intuition that constraint nondegeneracy is a boundary case of the W-SRCQ when the stratum achieves maximum dimension.
\begin{corollary}
    Suppose $z = (x,y) \in \widetilde{\mathcal{M}}_{p,q}$ with $p+q = n$. Then, the W-SRCQ holds at $z$ if and only if the constraint nondegeneracy holds at $z$.
\end{corollary}
\begin{proof}
    When $p+q = n$, the stratum $\mathcal{M}_{p,q}$ is an open submanifold of $\mathbb{S}^n$. Consequently, the normal space collapses to $\mathcal{N}_{g(x)}(\mathcal{M}_{p,q} - y) = \mathcal{N}_{G(z)}\mathcal{M}_{p,q} = \{ 0 \}$. This dictates that $\mathcal{D}\mathcal{N}_z = \mathcal{N}_z$. The equivalence follows immediately from Proposition~\ref{prop:trans=CN} and Theorem~\ref{thm:trans=WSRCQ}. \ep
\end{proof}

This corollary shows that constraint nondegeneracy can be viewed as a special case of the W-SRCQ. Nevertheless, constraint nondegeneracy occupies a distinguished position in the theory. In the next subsection we shall see that it plays a particularly important role when one investigates variational properties \emph{across} different strata.
On the other hand, the classical transversality characterization of constraint nondegeneracy~\cite[(4.182)]{bonnans2013perturbation} implies that constraint nondegeneracy is stable under small perturbations. 
The following theorem establishes the local stability of the W-SRCQ along each lifted stratum. As in the across-strata analysis, this stability follows from perturbations that preserve the negative inertia index of $G(z)$.

\begin{theorem}\label{thm:stability-on-stratum}
    If the W-SRCQ holds at $z = (x, y) \in \widetilde{\mathcal{M}}_{p,q}$, then there exists a neighborhood $\widetilde{\mathcal{U}} \subseteq \widetilde{\mathcal{M}}_{p,q}$ of $z$ such that the W-SRCQ holds at every point in $\widetilde{\mathcal{U}}$.
\end{theorem}

\begin{remark}

This statement follows directly from Proposition~\ref{prop:open-W-SRCQ}, which establishes the stronger stability under perturbations that preserve the negative inertia index of $G(z)$. The short proof is given in Appendix~\ref{appendix:transverse}.

\end{remark}

\begin{remark}
    Having securely grounded the W-SRCQ in a transversality framework, an analogous geometric interpretation for the W-SOC can be derived via the dual characterization established in~\cite[Section~4]{feng2025quadratically}, provided the primal problem~\eqref{prog:SDP} is a convex quadratic program. While a primal geometric formulation for the general W-SOC and S-SOSC could theoretically be developed, its complexity would likely preclude practical utility. We therefore leave this as a direction for future research.
\end{remark}

\subsection{Variational properties across strata}\label{sec:sub-vp-across-stra}

Having characterized the variational behavior \emph{within} a fixed stratum, this subsection investigates how these properties propagate \emph{across} different strata. We demonstrate that the problem- and solution-level regularity conditions exhibit natural topological openness and closedness properties, governed strictly by the adjacency relations among strata. Furthermore, we establish that classical strong-form regularity conditions are mathematically equivalent to the local \emph{uniform validity} of their weak-form counterparts.

Notice that in the previous section we used $\overline{z}$ to denote KKT pairs and $z$ for an arbitrary point in $\mathbb{X} \times \mathbb{S}^n$. In this section, we adopt a different convention: we use $\overline{z}$ to denote a target point in $\mathbb{X} \times \mathbb{S}^n$, and $z$ for points near $\overline{z}$, for the sake of simplicity, as we are concerned with perturbations of arbitrary points in this section.

\subsubsection{The W-SRCQ across strata}\label{section:subsub-CQ}

We first examine the perturbation properties of the W-SRCQ across strata. Proposition~\ref{prop:open-W-SRCQ} establishes a stratified openness property: the W-SRCQ remains locally stable under small perturbations that preserve the negative inertia index.

\begin{proposition}\label{prop:open-W-SRCQ}
    Let $\overline{z} = (\overline{x}, \overline{y}) \in \mathbb{X} \times \mathbb{S}^n$, and let $q$ be the negative inertia index of $G(\overline{z})$. If the W-SRCQ holds at $\overline{z}$, then there exists a neighborhood $\mathcal{U}$ of $\overline{z}$ such that the W-SRCQ holds at all $z \in \mathcal{U}$ for which the negative inertia index of $G(z)$ is $q$.
\end{proposition}

\begin{proof}
    Let $(\alpha, \beta, \gamma, p, q, P, \lambda)$ be an IED of $G(\overline{z})$. Suppose, for the sake of contradiction, that the assertion fails. Then, there exists a sequence $z^\nu = (x^\nu, y^\nu) \in \mathbb{X} \times \mathbb{S}^n$ such that $z^\nu \to \overline{z}$, the negative inertia index of $G(z^\nu)$ is exactly $q$, and the W-SRCQ fails at each $z^\nu$. Let $(\alpha^\nu, \beta^\nu, \gamma, p^\nu, q, P^\nu, \lambda^\nu)$ be an IED of $G(z^\nu)$. Because $z^\nu \to \overline{z}$, we have $G(z^\nu) \to G(\overline{z})$. By \cite[Lemma 3]{chen2003non}, passing to a subsequence if necessary, we may assume $P^\nu \to P'$ for some
    $P' \in \mathcal{O}^n(G(\overline{z}))$, and that the index set satisfies $\beta^\nu = \beta'$ for all $\nu$ with some fixed index set $\beta'$. By eigenvalue continuity, we must have $\beta' \subseteq \beta$.

    Since the W-SRCQ fails at $z^\nu$, there exists $0\neq W^\nu\in\mathbb S^n$ such that
    \[
    \langle W^\nu,\, g'(x^\nu)u\rangle =0\quad \forall\,u\in\mathbb X
    \quad\mbox{and}\quad
    \langle W^\nu,\, P^\nu B(P^\nu)^\top\rangle =0\quad \forall\, B\in\mathbb{S}^n:\ B_{\beta'\gamma}=0,\ B_{\gamma\gamma}=0.
    \]
    By compactness, we may assume $\|W^\nu\|=1$ and $W^\nu\to W\neq 0$, and hence, letting $\nu\to\infty$ and using
    $g'(x^\nu)\to g'(\overline x)$ and $P^\nu\to P'$, we obtain
    \[
    \langle W,\, g'(\overline x)u\rangle =0\quad \forall\,u\in\mathbb X\quad\mbox{and}\quad
    \langle W,\, P' B(P')^\top\rangle =0\quad \forall\,B\in\mathbb{S}^n:\ B_{\beta'\gamma}=0,\ B_{\gamma\gamma}=0,
    \]
    which implies
    \[
    g'(\overline{x})\mathbb{X} + \{ P' B (P')^\top \mid B_{\beta'\gamma}=0,\ B_{\gamma\gamma}=0 \}\neq \mathbb{S}^n.
    \]
    Since $P' \in \mathcal{O}^n(G(\overline{z}))$ and $\beta' \subseteq \beta$, this contradicts the assumption that the W-SRCQ holds at $\overline{z}$. The proof is then completed.\ep
\end{proof}

This proposition shows that, locally, the validity of the W-SRCQ propagates to higher-dimensional strata as long as the negative inertia index remains fixed, a behavior we interpret as a form of openness. By contrast, an analogous openness property for constraint nondegeneracy (Definition~\ref{def:CN-nonKKT}) holds without imposing any restriction on the negative inertia index.

This proposition demonstrates that the validity of the W-SRCQ propagates locally to higher-dimensional strata, provided the negative inertia index is fixed. In contrast, an analogous openness property for constraint nondegeneracy (Definition~\ref{def:CN-nonKKT}) holds unconditionally regarding inertia indices.

\begin{corollary}\label{coro:open-CN}
    Let $\overline  z = (\overline  x, \overline  y) \in \mathbb{X} \times \Sbb^n$. If the constraint nondegeneracy holds at $\overline  z$, then there exists a neighborhood $\mU$ of $\overline  z$ such that the constraint nondegeneracy holds at all $z \in \mU$.
\end{corollary}

\begin{proof}
    Let $(\alpha,\beta,\gamma,\overline p,\overline q,P,\lambda)$ be an IED of $G(\overline z)$.
    Suppose, for the sake of contradiction, that the assertion fails. Then there exists a sequence
    $z^\nu=(x^\nu,y^\nu)\to\overline z$ such that constraint nondegeneracy fails at each $z^\nu$.
    Let $(\alpha^\nu,\beta^\nu,\gamma^\nu,p^\nu,q^\nu,P^\nu,\lambda^\nu)$ be an IED of $G(z^\nu)$.
    Because $z^\nu\to\overline z$, we have $G(z^\nu)\to G(\overline z)$.
    By \cite[Lemma 3]{chen2003non}, passing to a subsequence if necessary, we may assume $P^\nu\to P'$ for some
    $P'\in\mathcal O^n(G(\overline z))$, and that $\alpha^\nu=\alpha'$ for all $\nu$ with some fixed index set $\alpha'$.
    By eigenvalue continuity, we must have $\alpha'\supseteq \alpha$.

    By~\eqref{eq:CN-nonKKT}, the failure of constraint nondegeneracy at $z^\nu$ implies there exists $0\neq W^\nu\in\mathbb{S}^n$ such that
    \[
    \langle W^\nu,\, g'(x^\nu)u\rangle =0\quad \forall\, u\in\mathbb{X} \quad \mbox{and} \quad
    \big\langle W^\nu,\, P^\nu B(P^\nu)^\top\big\rangle =0\quad  \forall\,B\in\mathbb{S}^n:\
    B_{\beta^\nu\beta^\nu}=0,\ B_{\beta^\nu\gamma}=0,\ B_{\gamma\gamma}=0 .
    \]
    Passing to a subsequence, assume $\|W^\nu\|=1$, $W^\nu\to W\neq 0$, and letting $\nu\to\infty$ yields
    \[
    \langle W,\, g'(\overline{x})u\rangle =0\quad \forall\, u\in\mathbb{X} \quad \mbox{and} \quad
    \big\langle W,\, P' B(P')^\top\big\rangle =0\quad \forall\,B\in\mathbb{S}^n:\
    B_{\beta'\beta'}=0,\ B_{\beta'\gamma}=0,\ B_{\gamma\gamma}=0 ,
    \]
    and hence
    \[
    g'(\overline{x})\mathbb{X}
    + \{ P' B (P')^\top \mid B_{\beta'\beta'}=0,\ B_{\beta'\gamma}=0,\ B_{\gamma\gamma}=0 \}
    \neq \mathbb{S}^n .
    \]
    Since $P'\in \mathcal{O}^n(G(\overline{z}))$ and $\beta'\subseteq \beta$ (equivalently, $\alpha'\supseteq\alpha$), this contradicts constraint nondegeneracy holds at $\overline{z}$. The proof is then completed.\ep
\end{proof}

The following result is an immediate consequence of Corollary~\ref{coro:open-CN}. It demonstrates that constraint nondegeneracy holds if and only if the W-SRCQ holds collectively across all relevant strata in a local neighborhood.

\begin{theorem}\label{thm:CN-W-SRCQ}
    Let $\overline{z} = (\overline{x}, \overline{y}) \in \mathbb{X} \times \Sbb^n$. The following statements are equivalent:
    \begin{itemize}
        \item[(a)] the constraint nondegeneracy holds at $(\overline  x,\overline  y)$;
        \item[(b)] the W-SRCQ holds at every pair $(x, y)$ sufficiently close to $(\overline  x, \overline  y)$;
        \item[(c)] the W-SRCQ holds at every pair $(\overline  x, y)$ with $y$ sufficiently close to $\overline  y$.
    \end{itemize}
\end{theorem}

\begin{proof}
    \textbf{(a) $\Rightarrow$ (b):} By Corollary~\ref{coro:open-CN}, \eqref{eq:CN-nonKKT} holds for all $(x, y)$ sufficiently close to $(\overline{x}, \overline{y})$, which directly implies the validity of the W-SRCQ~\eqref{eq:W-SRCQ-nonKKT}.

    \textbf{(b) $\Rightarrow$ (c):} This implication is trivial.

    \textbf{(c) $\Rightarrow$ (a):} Let $(\alpha,\beta,\gamma,p,q,P,\lambda)$ be an IED of $G(\overline{z})$, and define the perturbed multiplier
    \begin{equation*}
        y^t = \overline{y} - P\begin{bmatrix}
            0 & 0 & 0\\
            0 & tI & 0\\
            0 & 0 & 0
        \end{bmatrix}P^\top,
    \end{equation*}
    where $t > 0$ is sufficiently small to ensure $y^t$ is close to $\overline{y}$. By assumption, the W-SRCQ holds at $(\overline{x}, y^t)$. Invoking~\eqref{eq:W-SRCQ-nonKKT} at $y^t$ immediately recovers~\eqref{eq:CN-nonKKT} at $\overline{y}$, proving the constraint nondegeneracy at $(\overline{x}, \overline{y})$. \ep
\end{proof}

Because the classical constraint nondegeneracy (\cite[Definition~5.70]{bonnans2013perturbation}) is a special case of the extended formulation (Definition~\ref{def:CN-nonKKT}), the preceding theorem yields an analogous characterization for the classical setting. We omit the proof, which follows directly from the proof of Theorem~\ref{thm:CN-W-SRCQ}.

\begin{corollary}
    Let $\overline x$ be a feasible point of \eqref{prog:SDP} and $\overline y\in M(\overline x)$. Then the following  statements are equivalent:
    \begin{itemize}
        \item[(a)] the constraint nondegeneracy (\cite[Definition~5.70]{bonnans2013perturbation}) holds at $\overline  x$;
        \item[(b)] the W-SRCQ holds at every pair $(x, y)$ sufficiently close to $(\overline  x, \overline  y)$;
        \item[(c)] there exists $y$ strictly complementary to $g(\overline x)$, such that the W-SRCQ holds at $(\overline  x, y)$.
    \end{itemize}
\end{corollary}

Complementing the openness property in Proposition~\ref{prop:open-W-SRCQ}, the next result establishes a stratified closedness property for the W-SRCQ. Specifically, the validity of the W-SRCQ within a local neighborhood of a stratum extends to the boundary along directions that preserve the positive inertia index.

\begin{proposition}\label{prop:close-W-SRCQ}
    Let $\mU$ be an open set in $\mathbb{X} \times \Sbb^n$ and let $\tMcal_{p,q}$ be a stratum of $\mathbb{S}^n$. If the W-SRCQ holds at every $z \in \mU \cap \tMcal_{p,q}$ for some $p, q$, then the W-SRCQ holds at every $z \in \mU \cap \tMcal_{p,q'}$ for all $q'$ such that $0 \le q' \le q$.
\end{proposition}

\begin{proof}
    Fix $q'$ such that $0 \le q' \le q$. For any $z = (x, y) \in \mathcal{U} \cap \widetilde{\mathcal{M}}_{p,q'}$ with an IED $(\alpha, \beta', \gamma', p, q', P, \lambda)$ of $G(z)$, we define the perturbation direction
    \begin{equation*}
        \Delta = P \begin{bmatrix}
            0 & 0 & 0 & 0 \\
            0 & 0 & 0 & 0 \\
            0 & 0 & -I & 0 \\
            0 & 0 & 0 & 0
        \end{bmatrix} P^\top,
    \end{equation*}
    where the block partition matches dimensions $p$, $n-p-q$, $q-q'$, and $q'$, and $I$ is the identity matrix of order $q-q'$. The trajectory $z^t = (x, y + t\Delta)$ resides strictly in the stratum $\widetilde{\mathcal{M}}_{p,q}$ for all $t > 0$, and $z^t \to z$ as $t \downarrow 0$. Because $\mathcal{U}$ is an open neighborhood of $z$, there exists $t_0 > 0$ such that $z^t \in \mathcal{U}$ for all $t \in (0, t_0]$. By hypothesis, the W-SRCQ holds at each $z^t$ on this segment.

    According to~\eqref{eq:W-SRCQ-nonKKT}, the W-SRCQ at $z^t$ asserts that
    \begin{equation}\label{eq:char-W-SRCQ-z-t}
        g'(x)\mathbb{X} + \bigl\{ P B P^\top \mid B \in \mathbb{S}^n,\ B_{\beta\gamma} = 0,\ B_{\gamma\gamma} = 0 \bigr\} = \mathbb{S}^n,
    \end{equation}
    where $\beta = \{p+1, \ldots, n-q\}$ and $\gamma = \{n-q+1, \ldots, n\}$. Because $\beta \subseteq \beta'$ and $\gamma \supseteq \gamma'$, the subspace corresponding to the W-SRCQ at $z$,
    \begin{equation*}
        \bigl\{ P B P^\top \mid B \in \mathbb{S}^n,\ B_{\beta'\gamma'} = 0,\ B_{\gamma'\gamma'} = 0 \bigr\},
    \end{equation*}
    is strictly contained within the corresponding subspace in~\eqref{eq:char-W-SRCQ-z-t}. Therefore, the subspace addition naturally spans $\mathbb{S}^n$, verifying that the W-SRCQ holds at $z$. \ep
\end{proof}

\subsubsection{The W-SOC across strata}\label{section:subsub-SOC}

By exploiting the primal-dual relationships between the W-SOC and the W-SRCQ established in~\cite[Section~4]{feng2025quadratically}, a dual counterpart to the results in Section~\ref{section:subsub-CQ} can be developed. For general (possibly nonconvex) instances of~\eqref{prog:SDP}, analogous conclusions follow via similar arguments. Specifically, Proposition~\ref{prop:open-W-SOC} establishes a stratified openness property for the W-SOC. As the dual analogue to Proposition~\ref{prop:open-W-SRCQ}, it requires the \emph{positive} inertia index $p$ to remain invariant, rather than the negative inertia index.

\begin{proposition}\label{prop:open-W-SOC}
    Let $\overline{z} = (\overline{x}, \overline{y}) \in \mathbb{X} \times \mathbb{S}^n$, and let $p$ be the positive inertia index of $G(\overline{z})$. If the W-SOC holds at $\overline{z}$, then there exists a neighborhood $\mathcal{U}$ of $\overline{z}$ such that the W-SOC holds at every $z \in \mathcal{U}$ for which $G(z)$ has positive inertia index $p$.
\end{proposition}

\begin{proof}
    Let $(\alpha, \beta, \gamma, p, q, P, \lambda)$ be an IED of $G(\overline{z})$. Suppose, for the sake of contradiction, that the assertion fails. There exists a sequence $z^\nu = (x^\nu, y^\nu) \to \overline{z}$ such that the positive inertia index of $G(z^\nu)$ is $p$, but the W-SOC fails at each $z^\nu$. Let $(\alpha, \beta^\nu, \gamma^\nu, p, q^\nu, P^\nu, \lambda^\nu)$ be an IED of $G(z^\nu)$. Because $z^\nu \to \overline{z}$, we have $\operatorname{dist}\bigl(P^\nu, \mathcal{O}^n(G(\overline{z}))\bigr) \to 0$. By \cite[Lemma 3]{chen2003non}, passing to a subsequence if necessary, we assume $P^\nu \to P'$ for some $P' \in \mathcal{O}^n(G(\overline{z}))$, and $(\beta^\nu, \gamma^\nu) = (\beta', \gamma')$ for all $\nu$. By eigenvalue continuity, $\gamma' \supseteq \gamma$.

    The failure of the W-SOC at $z^\nu$ implies the existence of a vector $d^\nu \in \mathbb{X}$ with $\|d^\nu\| = 1$ such that
    \begin{equation*}
        \left\langle d^\nu, \nabla^2_{xx} L(z^\nu) d^\nu \right\rangle + 2 \sum_{i \in \alpha} \sum_{j \in \gamma'} \frac{-\lambda^\nu_j}{\lambda^\nu_i} \big[(P^\nu)^\top (\nabla g(x^\nu)^* d^\nu) P^\nu\big]_{ij}^2 = 0,
    \end{equation*}
    while satisfying the index block conditions:
    \begin{equation*}
        \bigl[(P^\nu)^\top (\nabla g(x^\nu)^*d^\nu)P^\nu\bigr]_{\beta'\beta'} = 0, \quad \bigl[(P^\nu)^\top (\nabla g(x^\nu)^*d^\nu)P^\nu\bigr]_{\beta'\gamma'} = 0 \quad \text{and} \quad \bigl[(P^\nu)^\top (\nabla g(x^\nu)^*d^\nu)P^\nu\bigr]_{\gamma'\gamma'} = 0.
    \end{equation*}
    Taking a further subsequence such that $d^\nu \to d^\infty \in \mathbb{X}$ (with $\|d^\infty\| = 1$), and passing to the limit as $\nu \to \infty$, we obtain
    \begin{equation*}
        \left\langle d^\infty, \nabla^2_{xx} L(\overline{z}) d^\infty \right\rangle + 2 \sum_{i \in \alpha} \sum_{j \in \gamma'} \frac{-\overline{\lambda}_j}{\overline{\lambda}_i} \big[(P')^\top (\nabla g(\overline{x})^* d^\infty) P'\big]_{ij}^2 = 0,
    \end{equation*}
    along with the limit block conditions:
    \begin{equation*}
        \bigl[(P')^\top (\nabla g(\overline{x})^*d^\infty)P'\bigr]_{\beta'\beta'} = 0, \quad \bigl[(P')^\top (\nabla g(\overline{x})^*d^\infty)P'\bigr]_{\beta'\gamma'} = 0 \quad \text{and} \quad \bigl[(P')^\top (\nabla g(\overline{x})^* d^\infty)P'\bigr]_{\gamma'\gamma'} = 0.
    \end{equation*}
    Because $P' \in \mathcal{O}^n(G(\overline{z}))$, $\gamma' \supseteq \gamma$, $\beta' \cup \gamma' = \{1, \ldots, n\} \setminus \alpha$, and $\overline{\lambda}_j = 0$ for all $j \in \gamma' \setminus \gamma$, this strictly contradicts the assumption that the W-SOC holds at $\overline{z}$. \ep
\end{proof}

Recall from Section~\ref{section:subsub-CQ} that the openness property of the W-SRCQ (Proposition \ref{prop:open-W-SRCQ}) and the constraint nondegeneracy (Proposition \ref{coro:open-CN}) lead to Theorem~\ref{thm:CN-W-SRCQ}, which characterizes constraint nondegeneracy via the local collective validity of the W-SRCQ. For second order conditions, however, the picture is more subtle: the strong second order sufficient condition (S-SOSC) is not, in general, stable in the same stratified sense. Consequently, establishing a dual analogue of Theorem~\ref{thm:CN-W-SRCQ}, that is, relating the S-SOSC to the local collective validity of the W-SOC, requires additional assumptions and more delicate technical tools.
Observe that if a self-adjoint linear operator $\mathcal{Q}$ on a Euclidean space $\mathbb{X}$ satisfies
\[
\langle \mathcal{Q}x, x\rangle \neq 0 \quad \forall\, x\in \mathcal{H}\setminus\{0\},
\]
for some subspace $\mathcal{H}\subseteq \mathbb{X}$, then by continuity of the mapping
$x\mapsto \langle \mathcal{Q}x, x\rangle$ and compactness of the unit sphere in $\mathcal{H}$, there exists a constant $c>0$ such that
\[
|\langle \mathcal{Q}x, x\rangle| \ge c\|x\|^2 \quad \forall\, x\in \mathcal{H}.
\]
Hence, if the W-SONC and W-SOC hold at some $z$, then the right-hand side in \eqref{eq:W-SOC-nonKKT} can be changed to $c\|v_x\|^2$ for some fixed $c>0$.
Motivated by this elementary fact, we introduce a uniform version of the W-SOC in a neighborhood of a given point $\overline{z}\in \mathbb{X}\times\mathbb{S}^n$.

\begin{definition}\label{def:W-SOC-nonKKT-uniform}
    Let $\overline{z} = (\overline{x}, \overline{y}) \in \mathbb{X} \times \mathbb{S}^n$. The W-SOC is said to hold \emph{uniformly on a neighborhood} $\mathcal{U}$ of $\overline{z}$ if there exists a constant $c > 0$ such that, for every $z = (x,y) \in \mathcal{U}$ with an IED $(\alpha,\beta,\gamma,p,q,P,\lambda)$ of $G(z)$, we have
    \begin{equation}\label{eq:W-SOC-nonKKT-uniform}
        \mQ_{z}(v_x) \ge c\|v_x\|^2 \quad \forall\, v_x \in \operatorname{appl}(x,y),
    \end{equation}
    where $\operatorname{appl}(x,y)$ is defined in~\eqref{eq:def-appl-nonKKT}.
\end{definition}

With Definitions~\ref{def:W-SOC-nonKKT-uniform} in place, we can now establish the desired relationship between the S-SOSC and the uniform validity of the W-SOC.

\begin{theorem}\label{thm:S-SOSC-W-SOC}
Let $\overline{z}=(\overline{x},\overline{y})\in \mathbb{X}\times\Sbb^n$. The following statements are equivalent:
\begin{itemize}
\item[(a)] the S-SOSC holds at $\overline{z}$;
\item[(b)] the W-SOC holds uniformly on a neighborhood of $\overline{z}$;
\item[(c)] the W-SOC holds uniformly for $(\overline{x},y)$ with $y$ in a neighborhood of $\overline{y}$.
\end{itemize}
\end{theorem}

\begin{proof}
    Let $(\overline{\alpha}, \overline{\beta}, \overline{\gamma}, \overline{p}, \overline{q}, \overline{P}, \overline{\lambda})$ be an IED of $G(\overline{z})$.

    \textbf{(a) $\Rightarrow$ (b):} We argue by contradiction. Suppose that there exist a sequence $z^\nu = (x^\nu, y^\nu) \to \overline{z}$ and scalars $c^\nu \downarrow 0$ such that the uniform W-SOC fails at each $z^\nu$. Let $(\alpha^\nu, \beta^\nu, \gamma^\nu, p^\nu, q^\nu, P^\nu, \lambda^\nu)$ be an IED of $G(z^\nu)$. By \cite[Lemma 3]{chen2003non}, passing to a subsequence if necessary, we may assume $P^\nu \to P \in \mathcal{O}^n(G(\overline{z}))$, $(\alpha^\nu, \beta^\nu, \gamma^\nu) = (\alpha, \beta, \gamma)$ for all $\nu$, and $\lambda^\nu \to \overline{\lambda}$. Without loss of generality, $\alpha \supseteq \overline{\alpha}$, $\beta \subseteq \overline{\beta}$, and $\gamma \supseteq \overline{\gamma}$.

    The failure of the uniform W-SOC guarantees the existence of $d^\nu \in \mathbb{X}$ with $\|d^\nu\| = 1$ such that
    \begin{equation}\label{eq:W-SOC-c^nu-fail}
        \left\langle d^\nu, \nabla^2_{xx} L(z^\nu) d^\nu \right\rangle + 2 \sum_{i \in \alpha} \sum_{j \in \gamma} \frac{-\lambda^\nu_j}{\lambda^\nu_i} \big[ (P^\nu)^\top (\nabla g(x^\nu)^* d^\nu) P^\nu\big]_{ij}^2 < c^\nu
    \end{equation}
    and $d^\nu \in \operatorname{appl}(z^\nu)$, satisfying:
    \begin{equation}\label{eq:appl-nu}
        \bigl[(P^\nu)^\top (\nabla g(x^\nu)^*d^\nu) P^\nu\bigr]_{\beta\beta} = 0,\quad
        \bigl[(P^\nu)^\top (\nabla g(x^\nu)^*d^\nu) P^\nu\bigr]_{\beta\gamma} = 0,\quad
        \bigl[(P^\nu)^\top (\nabla g(x^\nu)^*d^\nu) P^\nu\bigr]_{\gamma\gamma} = 0.
    \end{equation}
    For $i \in \alpha \setminus \overline{\alpha}$ and $j \in \overline{\gamma}$, we have $\lambda_i^\nu \downarrow 0$ and $\lambda_j^\nu \to \overline{\lambda}_j < 0$, making the ratio $-\lambda_j^\nu / \lambda_i^\nu \to +\infty$. Because the quadratic sequence in~\eqref{eq:W-SOC-c^nu-fail} is bounded, this enforces
    \begin{equation}\label{eq:alpha-setminus-gamma}
        \big[ (P^\nu)^\top (\nabla g(x^\nu)^* d^\nu) P^\nu\big]_{ij} \to 0 \quad \forall\, i \in \alpha\setminus\overline{\alpha},\ j \in \overline{\gamma}.
    \end{equation}
    Restricting the sum to $\overline{\alpha}$ and $\overline{\gamma}$, we maintain
    \begin{equation}\label{eq:W-SOC-nu-bar-alpha-gamma}
        \left\langle d^\nu, \nabla^2_{xx} L(z^\nu) d^\nu \right\rangle + 2 \sum_{i \in \overline{\alpha}} \sum_{j \in \overline{\gamma}} \frac{-\lambda^\nu_j}{\lambda^\nu_i} \big[ (P^\nu)^\top (\nabla g(x^\nu)^* d^\nu) P^\nu\big]_{ij}^2 < c^\nu.
    \end{equation}
    Assuming $d^\nu \to d^\infty$ with $\|d^\infty\| = 1$, we pass to the limit in~\eqref{eq:appl-nu},~\eqref{eq:alpha-setminus-gamma}, and~\eqref{eq:W-SOC-nu-bar-alpha-gamma} to obtain
    \begin{align*}
        &\left\langle d^\infty, \nabla^2_{xx} L(\overline{z}) d^\infty \right\rangle + 2 \sum_{i \in \overline{\alpha}} \sum_{j \in \overline{\gamma}} \frac{-\overline{\lambda}_j}{\overline{\lambda}_i} \big[ P^\top (\nabla g(\overline{x})^* d^\infty) P\big]_{ij}^2 \le 0,\\
        &\bigl[P^\top (\nabla g(\overline{x})^*d^\infty) P\bigr]_{\beta\beta} = 0, \quad \bigl[P^\top (\nabla g(\overline{x})^*d^\infty) P\bigr]_{\beta\gamma} = 0, \quad \bigl[P^\top (\nabla g(\overline{x})^*d^\infty) P\bigr]_{\gamma\gamma} = 0,\\
        \text{and} \quad &\bigl[ P^\top (\nabla g(\overline{x})^*d^\infty) P\bigr]_{ij} = 0 \quad \forall\, i \in \alpha \setminus \overline{\alpha},\ j \in \overline{\gamma}.
    \end{align*}
    These conditions guarantee that $d^\infty \in \operatorname{app}(\overline{z})$. However, the resulting nonpositive quadratic form directly contradicts the S-SOSC at $\overline{z}$, confirming (a) $\Rightarrow$ (b).

    \textbf{(b) $\Rightarrow$ (c):} This implication is trivial.

    \textbf{(c) $\Rightarrow$ (a):} Define the block matrix
    \begin{equation*}
        \Delta = \overline{P} \begin{bmatrix}
            0 & 0 & 0 \\
            0 & I & 0 \\
            0 & 0 & 0
        \end{bmatrix} \overline{P}^\top,
    \end{equation*}
    corresponding to dimensions $\overline{p}$, $n-\overline{p}-\overline{q}$, and $\overline{q}$, where $I$ is the relevant identity matrix. The sequence $z^t = (\overline{x}, \overline{y} + t\Delta) \to \overline{z}$ as $t \downarrow 0$. By assumption, there exists $t_0 > 0$ such that the W-SOC holds uniformly at $z^t$ for all $t \in (0, t_0]$. Thus, there exists a constant $c > 0$ such that
    \begin{equation}\label{eq:char-quadform-zt2}
        \left\langle v_x, \nabla^2_{xx} L(z^t) v_x \right\rangle + 2 \sum_{i\in\overline{\alpha}} \sum_{j \in \overline{\gamma}} \frac{-\overline{\lambda}_j}{\overline{\lambda}_i} \big[\overline{P}^\top (\nabla g(\overline{x})^* v_x) \overline{P}\big]_{ij}^2 + 2 \sum_{i=\overline{p}+1}^{n-\overline{q}} \sum_{j \in \overline{\gamma}} \frac{-\overline{\lambda}_j}{t} \big[\overline{P}^\top (\nabla g(\overline{x})^* v_x) \overline{P}\big]_{ij}^2 \ge c\|v_x\|^2
    \end{equation}
    for all $v_x \in \operatorname{appl}(z^t) = \bigl\{v_x \in \mathbb{X} \mid \big[ \overline{P}^\top (\nabla g(\overline{x})^*v_x) \overline{P}\big]_{\overline{\gamma}\,\overline{\gamma}} = 0\bigr\}$.

    If the S-SOSC fails at $\overline{z}$, there exists a nonzero $d \in \operatorname{app}(\overline{z})$ such that
    \begin{equation}\label{eq:char-quadform-bar-z}
        \left\langle d, \nabla^2_{xx} L(\overline{z}) d \right\rangle + 2 \sum_{i \in \overline{\alpha}} \sum_{j \in \overline{\gamma}} \frac{-\overline{\lambda}_j}{\overline{\lambda}_i} \big[\overline{P}^\top (\nabla g(\overline{x})^* d)\overline{P}\big]_{ij}^2 \le 0.
    \end{equation}
    Because $d \in \operatorname{app}(\overline{z})$, it satisfies $\bigl[\overline{P}^\top (\nabla g(\overline{x})^*d) \overline{P}\bigr]_{\overline{\gamma}\,\overline{\gamma}} = 0$ and $\bigl[\overline{P}^\top (\nabla g(\overline{x})^*d) \overline{P}\bigr]_{\overline{\beta}\,\overline{\gamma}} = 0$. Consequently, $d \in \operatorname{appl}(z^t)$. Substituting $v_x = d$ into~\eqref{eq:char-quadform-zt2} and letting $t \downarrow 0$ yields
    \begin{equation*}
        \left\langle d, \nabla^2_{xx} L(\overline{z}) d \right\rangle + 2 \sum_{i\in\overline{\alpha}} \sum_{j \in \overline{\gamma}} \frac{-\overline{\lambda}_j}{\overline{\lambda}_i} \big[\overline{P}^\top (\nabla g(\overline{x})^* d) \overline{P}\big]_{ij}^2 \ge c\|d\|^2 > 0,
    \end{equation*}
    which contradicts~\eqref{eq:char-quadform-bar-z}. Thus, the S-SOSC must hold at $\overline{z}$. \ep
\end{proof}

We finalize this subsection by establishing the stratified closedness of the uniform W-SOC.

\begin{proposition}\label{prop:close-W-SOC}
    Let $\mathcal{U}$ be an open set in $\mathbb{X} \times \mathbb{S}^n$ and let $\widetilde{\mathcal{M}}_{p,q}$ be a stratum. If the W-SOC holds uniformly at every $z \in \mathcal{U} \cap \widetilde{\mathcal{M}}_{p,q}$ for some $p, q$, then the W-SOC holds uniformly at every $z \in \mathcal{U} \cap \widetilde{\mathcal{M}}_{p',q}$ for all $0 \le p' \le p$.
\end{proposition}

\begin{proof}
    For $0 \le p' \le p$ and $z = (x, y) \in \mathcal{U} \cap \widetilde{\mathcal{M}}_{p',q}$ with an IED $(\alpha', \beta', \gamma, p', q, P, \lambda)$ of $G(z)$, define the perturbation direction
    \begin{equation*}
        \Delta = P \begin{bmatrix}
            0 & 0 & 0 & 0 \\
            0 & I & 0 & 0 \\
            0 & 0 & 0 & 0 \\
            0 & 0 & 0 & 0
        \end{bmatrix} P^\top,
    \end{equation*}
    partitioned according to dimensions $p'$, $p-p'$, $n-p-q$, and $q$, where $I$ is the identity matrix of order $p-p'$. The sequence $z^t = (x, y + t\Delta)$ resides strictly in $\widetilde{\mathcal{M}}_{p,q}$ for $t > 0$, converging to $z$ as $t \downarrow 0$. Since $\mathcal{U}$ is open, $z^t \in \mathcal{U}$ for sufficiently small $t \in (0, t_0]$. Let $\alpha = \{1, \ldots, p\}$ and $\beta = \{p+1, \ldots, n-q\}$. The tuple $(\alpha, \beta, \gamma, p, q, P, \lambda + t\operatorname{diag}(\Delta))$ serves as a valid IED for $G(z^t)$. Note that $\alpha \supseteq \alpha'$ and $\beta \subseteq \beta'$.

    By assumption, the uniform W-SOC at $z^t$ guarantees a constant $c > 0$ such that
    \begin{equation}\label{eq:W-SOC-z^t}
        \langle v_x, \nabla^2_{xx} L(z^t) v_x \rangle + 2 \sum_{i \in \alpha'} \sum_{j \in \gamma} \frac{-\lambda_j}{\lambda_i} \bigl[ P^\top (\nabla g(x)^* v_x) P\bigr]_{ij}^2 + 2 \sum_{i \in \alpha \setminus \alpha'} \sum_{j \in \gamma} \frac{-\lambda_j}{t} \bigl[ P^\top (\nabla g(x)^* v_x) P\bigr]_{ij}^2 \ge c\|v_x\|^2
    \end{equation}
    for all nonzero $v_x \in \operatorname{appl}(z^t)$, where
    \begin{equation*}
        \operatorname{appl}(z^t) = \bigl\{ v_x \in \mathbb{X} \mid \bigl[P^\top (\nabla g(x)^*v_x) P\bigr]_{\beta\beta} = 0,\ \bigl[P^\top (\nabla g(x)^*v_x) P\bigr]_{\beta\gamma} = 0,\ \bigl[P^\top (\nabla g(x)^*v_x) P\bigr]_{\gamma\gamma} = 0 \bigr\}.
    \end{equation*}
    Because $\beta \subseteq \beta'$, we inherently have $\operatorname{appl}(z) \subseteq \operatorname{appl}(z^t)$.

    If the W-SOC fails at $z$, there exists a nonzero vector $d \in \operatorname{appl}(z)$ such that
    \begin{equation}\label{eq:W-SOC-z3}
        \langle d, \nabla^2_{xx} L(z) d \rangle + 2 \sum_{i \in \alpha'} \sum_{j \in \gamma} \frac{-\lambda_j}{\lambda_i} \bigl[ P^\top (\nabla g(x)^* d) P \bigr]_{ij}^2 < c\|d\|^2.
    \end{equation}
    Because $d \in \operatorname{appl}(z)$, the cross terms satisfy $\bigl[ P^\top (\nabla g(x)^* d) P \bigr]_{ij} = 0$ for all $i \in \alpha \setminus \alpha'$ and $j \in \gamma$. Substituting $v_x = d$ into~\eqref{eq:W-SOC-z^t} and letting $t \downarrow 0$ yields a strict contradiction against~\eqref{eq:W-SOC-z3}. This verifies that the W-SOC uniformly extends to $z$. \ep
\end{proof}

\subsubsection{The problem-level regularity across strata}\label{section:subsub-regularity}

A foundational result established in~\cite{sun2006strong} elegantly links problem-level constraint qualifications to solution-level strong metric regularity for the NLSDP. In Section~\ref{subsec:variational-properties}, we derived a purely stratum-restricted analogue to this equivalence. We now integrate the \emph{across-strata} perturbation results from Sections~\ref{section:subsub-CQ} and~\ref{section:subsub-SOC}, demonstrating that strong-form problem-level conditions at a point naturally equate to the \emph{local uniform validity} of their weak-form counterparts across a neighborhood.

\begin{theorem}\label{thm:regularity-across}
    Let $\overline{x}$ be a local minimizer of~\eqref{prog:SDP} satisfying the RCQ, and let $\overline{y} \in M(\overline{x})$. The following statements (a)--(d) are equivalent, and each implies (e):
    \begin{itemize}
        \item[(a)] the S-SOSC and the constraint nondegeneracy hold simultaneously at $(\overline{x}, \overline{y})$;
        \item[(b)] every element in the Clarke generalized Jacobian $\partial_C F(\overline{x}, \overline{y})$ is nonsingular;
        \item[(c)] the KKT mapping $F$ is strongly metrically regular at $(\overline{x}, \overline{y})$;
        \item[(d)] the W-SOC holds uniformly, and the W-SRCQ holds at every pair $(x, y)$ sufficiently close to $(\overline{x}, \overline{y})$;
        \item[(e)] for every $(x,y)$ sufficiently close to $(\overline{x},\overline{y})$, the manifold differential $dF_{(x, y)}$ is injective.
    \end{itemize}
\end{theorem}

\begin{proof}
    The equivalence (a) $\Leftrightarrow$ (b) $\Leftrightarrow$ (c) is classically established in~\cite[Theorem~4.1]{sun2006strong}. The equivalence (a) $\Leftrightarrow$ (d) follows directly from Theorem~\ref{thm:CN-W-SRCQ} and Theorem~\ref{thm:S-SOSC-W-SOC}. Finally, the implication (d) $\Rightarrow$ (e) was proved in Theorem~\ref{thm:WR-dF-nons}. \ep
\end{proof}

\begin{remark}
    If statement~(e) is strengthened to the following \emph{uniform} version:
    \begin{quote} 
        \emph{(e$'$)} \emph{for every $(x,y)$ sufficiently close to $(\overline{x},\overline{y})$, the manifold differential $dF_{(x,y)}$ is uniformly injective and positively (or negatively) oriented,}
    \end{quote}
    then this uniform condition~(e$'$) is in fact equivalent to statements~(a)--(d) in Theorem~\ref{thm:regularity-across}. Here, ``uniformly injective'' means that there exist a constant $\kappa>0$ such that for $z$ sufficient close to $(\overline{x},\overline{y})$:
    \[
        \|dF_z h\|\ge \kappa \|h\|
        \qquad
        \forall\, h\in \mathcal T_z\widetilde{\mathcal M}.
    \]
    Roughly speaking, ``positively (or negatively) oriented'' means the determinants have the same sign. A precise formulation, however, is more subtle in the present setting: it must characterize the orientation of $dF_z$ as a map onto its image space, while also accounting for the compatibility of these orientations across neighboring strata.     We therefore do not pursue the exact definitions or the corresponding proofs here, as they would require additional topological machinery and would lead us too far from the main line of development. We leave this refinement for future work. What is most important for the purposes of this paper is that Theorem~\ref{thm:regularity-across} shows the classical strong-form regularity conditions to be equivalent to the local uniform validity of their stratum-restricted weak counterparts.
\end{remark}

\begin{remark}\label{remark:condition-uniform-validity}
    If the constraint mapping $g$ is affine, the requirement of uniform validity over a local neighborhood can be safely relaxed to pointwise validity. In such cases, the quadratic bound~\eqref{eq:W-SOC-z^t} is satisfied globally, immediately yielding the requisite theoretical equivalences.
\end{remark}

\begin{remark}
    From a primal-dual perspective, it is natural to ask why statement~(d) demands a \emph{uniform} version of the W-SOC, whereas no such uniformity is required for the W-SRCQ. We offer an intuitive geometric explanation here, leaving a rigorous analysis for future work. The W-SRCQ is governed by the affine structure of the linearized constraint mapping $\nabla g(x)^{*}$; as noted in Remark~\ref{remark:condition-uniform-validity}, this affine nature automatically enforces uniformity once the condition holds pointwise locally. In contrast, the W-SOC (the dual analogue of the W-SRCQ) loses this affine characteristic due to the inherent spectral nonlinearity of the nonpolyhedral cone $\mathbb{S}^n_+$. Consequently, an explicit uniform formulation is mathematically indispensable.
\end{remark}

\section{A globalized Gauss--Newton method based on stratification}\label{section:global}

In this section, we leverage the stratification framework to develop a globally convergent Newton-type method with rapid local convergence for solving the KKT system~\eqref{eq:kkt-mapping}. By restricting the nonsmooth KKT mapping $F$ to individual strata, our approach systematically resolves the inherent nonsmoothness. Specifically, to compute a KKT pair of~\eqref{prog:SDP}, we consider the unconstrained least-squares reformulation:
\begin{equation}\label{prog:global-P}
\begin{aligned}
    \min \quad & \varphi(z) := \frac{1}{2}\|F(z)\|^2 \\
    \mathrm{s.t.} \quad & z \in \mathbb{X}\times\mathbb{S}^n.
\end{aligned}
\end{equation}

Because the nonsmoothness of $F$ is resolved only by restricting it to a fixed stratum, any global method must be capable of navigating seamlessly between strata of varying dimensions. Accordingly, our algorithm integrates three complementary mechanisms:
\begin{itemize}
    \item \emph{Tangential steps along a stratum:} Compute a Levenberg--Marquardt step tangent to the current stratum, implemented via a line-search-compatible retraction to ensure monotonic descent of $\varphi$.
    \item \emph{Normal steps for stratum escape:} Exploit explicitly constructed normal directions to escape from stationary behavior confined solely to a local stratum. These steps are accepted without a line search.
    \item \emph{Eigenvalue-triggered stratum correction:} When a sufficiently small eigenvalue is detected, deploy a correction mechanism to identify a lower-dimensional stratum harboring a more appropriate inertia.
\end{itemize}

Under mild assumptions, we establish the global convergence of the proposed method and rigorously prove its superlinear or quadratic local convergence.  {The local result assumes D-stationarity, SOSC, and SRCQ at an accumulation point and a suitable correction threshold; D-stationarity follows from the global convergence result, and the KKT property follows from Theorem~\ref{thm:dstat-implies-kkt}. The rate is superlinear under $C^2$ and quadratic under additional $LC^2$ regularity.}

\subsection{The algorithm formulation}

We first outline the proposed algorithm at a conceptual level, deferring precise analytical definitions to subsequent discussions. We refer to the method as the \emph{Stratified Gauss--Newton method} (SGN; see Algorithm~\ref{alg:SGN}). At each iteration, the current iterate $z^k$ is first potentially corrected to $\widehat{z}^k$, aiming to project it onto a stratum where the desired solution point satisfying the requisite regularity properties may stay. Starting from $\widehat{z}^k$, we compute a stratum LM-normal step, denoted by $\mathrm{SLMN}(\widehat{z}^k)$ (see Algorithm~\ref{alg:SLMN}), and apply an acceptance test based on the sufficient decrease of the merit function $\varphi$. This mechanism strictly enforces a monotonic decrease in the sequence $\{\varphi(z^k)\}$, rigorously precluding any potential residual inflation induced by the correction step.

\begin{algorithm}[H]
\caption{Stratified Gauss--Newton Method with correction (SGN)}
\label{alg:SGN}
\begin{algorithmic}[1]
\State \textbf{Initialize:} $z^0 \in \mathbb{X} \times \mathbb{S}^n$, correction threshold $\delta > 0$, and tolerance $\epsilon > 0$.
\For{$k = 0, 1, \ldots$ while $s(z^k) > \epsilon$ (Definition~\ref{def:d-sta-mea})}
    \State Compute the corrected point $\widehat{z}^k$ from $z^k$ via~\eqref{eq:hat-z} using the threshold $\delta$.
    \If{$\varphi\bigl( \mathrm{SLMN}(\widehat{z}^k) \bigr) \leq \varphi(z^k)$}
        \State $z^{k+1} \leftarrow \mathrm{SLMN}(\widehat{z}^k)$
    \Else
        \State $z^{k+1} \leftarrow \mathrm{SLMN}(z^k)$
    \EndIf
\EndFor
\end{algorithmic}
\end{algorithm}

Algorithm~\ref{alg:SGN} relies on three functional subroutines: a stationarity-measure function $s(\cdot)$ (see Definition \ref{def:d-sta-mea}), a stratum LM-normal step $\mathrm{SLMN}(\cdot)$ (see Algorithm \ref{alg:SLMN}), and a correction step $\widehat{(\cdot)}$ (see \eqref{eq:hat-z}). We detail and analyze these components in the following subsections.

\subsubsection{The directional stationarity and stationarity measure}

A suitable stopping criterion is essential for practical iterative algorithms. Conventionally, the gradient norm of the merit function serves this purpose. Due to the nonsmoothness of $\varphi$, however, we must adopt the broader concept of \emph{directional stationarity}. We first recall its standard definition in nonsmooth optimization (see, e.g.,~\cite[Section~8.2]{facchinei2003finite}).

\begin{definition}\label{def:D-stationary}
    Let $\phi : \mathbb{R}^n \to \mathbb{R}$ be a locally Lipschitz continuous and directionally differentiable function. A point $z^* \in \mathbb{R}^n$ is called a \emph{directionally stationary point} (or simply \emph{D-stationary point}) of $\phi$ if
    \begin{equation*}
        \phi'(z^*; v) \ge 0 \quad \forall\, v \in \mathbb{R}^n.
    \end{equation*}
\end{definition}

Since problem~\eqref{prog:global-P} is unconstrained, it is natural to use its objective function $\varphi$ as the merit function. It follows directly from the definition that every KKT pair of~\eqref{prog:SDP}, or equivalently, every global minimizer of~\eqref{prog:global-P}, is a D-stationary point of $\varphi$. We now compute the directional derivative of $\varphi$. While previous sections restricted the tangent vector $v = (v_x,v_y)$ to $\mathcal{T}_z\widetilde{\mathcal{M}}_{p,q}$, we now consider an arbitrary direction $v \in \mathbb{X} \times \mathbb{S}^n$  (not necessarily tangent to a fixed stratum), retaining the notation $v = (v_x, v_y)$ for uniformity.

Fix $z = (x,y) \in \mathbb{X} \times \mathbb{S}^n$, and let $G(z)$ admit an IED $(\alpha,\beta,\gamma,p,q,P,\lambda)$. Recall the orthogonal decomposition of the ambient matrix space:
\begin{equation}\label{eq:decomp-of-Y}
    \mathbb{S}^n = \mathcal{T}_{G(z)}\mathcal{M}_{p,q} \oplus \mathcal{N}_{G(z)}\mathcal{M}_{p,q},
\end{equation}
where
\begin{equation*}
    \mathcal{T}_{G(z)}\mathcal{M}_{p,q} = \bigl\{ P\widetilde{H}P^{\top} \mid \widetilde{H} \in \mathbb{S}^n,\ \widetilde{H}_{\beta\beta} = 0 \bigr\}
\end{equation*}
and
\begin{equation*}
    \mathcal{N}_{G(z)}\mathcal{M}_{p,q} = \bigl\{ P\widetilde{H}P^{\top} \mid \widetilde{H} \in \mathbb{S}^n,\ \widetilde{H}_{ij} = 0 \text{ for all } (i,j) \notin \beta \times \beta \bigr\}.
\end{equation*}

The following lemma lifts this  decomposition to $\mathbb{X} \times \mathbb{S}^n$ in a manner compatible with the coordinate isomorphism $\phi_z$ defined in~\eqref{eq:phi_z}.

\begin{lemma}\label{lemma:decomp-of-XxY}
    For an arbitrary $z \in \mathbb{X} \times \mathbb{S}^n$ with $G(z)$ admitting an IED $(\alpha,\beta,\gamma,p,q,P,\lambda)$, and for any direction $v = (v_x,v_y) \in \mathbb{X} \times \mathbb{S}^n$, the isomorphism $\phi_z$ defined in~\eqref{eq:phi_z} induces an orthogonal decomposition $\phi_z(v) = u_1 + u_2$, where
    \begin{equation*}
        u_1 = \begin{bmatrix} v_x \\ H_1 \end{bmatrix} \in \mathbb{X} \times \mathcal{T}_{G(z)} \mathcal{M}_{p,q}, \quad \text{with } H_1 \in \mathcal{T}_{G(z)} \mathcal{M}_{p,q},
    \end{equation*}
    and
    \begin{equation*}
        u_2 = \begin{bmatrix} 0 \\ H_2 \end{bmatrix} \in \{ 0 \} \times \mathcal{N}_{G(z)} \mathcal{M}_{p,q}, \quad \text{with } H_2 \in \mathcal{N}_{G(z)} \mathcal{M}_{p,q}.
    \end{equation*}
    Defining the normal pullback $\mathcal{N}_z\widetilde{\mathcal{M}}_{p,q} := \phi_z^{-1} \bigl(\{ 0 \} \times \mathcal{N}_{G(z)}\mathcal{M}_{p,q}\bigr)$, the ambient space uniquely decomposes as the direct sum
    \begin{equation*}
        \mathbb{X} \times \mathbb{S}^n = \mathcal{T}_z \widetilde{\mathcal{M}}_{p,q} \oplus \mathcal{N}_z \widetilde{\mathcal{M}}_{p,q}.
    \end{equation*}
    Consequently, we obtain the unique decomposition $v = v_1 + v_2$, with $v_1 \in \mathcal{T}_z\widetilde{\mathcal{M}}_{p,q}$ and $v_2 \in \mathcal{N}_z\widetilde{\mathcal{M}}_{p,q}$.
\end{lemma}

\begin{proof}
    The verification is straightforward that the mapping $\phi_z$ established in Lemma~\ref{lemma:TzMt} naturally extends to the entire ambient space $\mathbb{X} \times \mathbb{S}^n$ via identical mapping rules:
    \begin{align*}
        \phi_z(v_x, v_y) &= \bigl(v_x,\ \nabla g(x)^*v_x + v_y\bigr), \\
        \phi_z^{-1}(v_x, H) &= \bigl(v_x,\ H - \nabla g(x)^*v_x\bigr).
    \end{align*}
    These constitute a pair of linear isomorphisms on $\mathbb{X} \times \mathbb{S}^n$. Leveraging~\eqref{eq:decomp-of-Y}, the product space decomposes identically as
    \begin{equation}\label{eq:decom-XxY}
        \mathbb{X} \times \mathbb{S}^n = \bigl(\mathbb{X} \times \mathcal{T}_{G(z)} \mathcal{M}_{p,q}\bigr) \oplus \bigl(\{ 0 \} \times \mathcal{N}_{G(z)} \mathcal{M}_{p,q}\bigr).
    \end{equation}
    Projecting $\phi_z(v)$ onto these respective subspaces yields the vectors $u_1$ and $u_2$. Applying $\phi_z^{-1}$ yields $v_1 = \phi_z^{-1}(u_1) \in \mathcal{T}_z \widetilde{\mathcal{M}}_{p,q}$ and $v_2 = \phi_z^{-1}(u_2) \in \mathcal{N}_z \widetilde{\mathcal{M}}_{p,q}$. \ep
\end{proof}

\begin{remark}
    For readers familiar with Riemannian geometry, the above construction can be interpreted as twisting the Riemannian metric on $\Xbb \times \Sbb^n$ via the map $g$. Consequently, $\Ncal_z \tMcal$ is indeed normal to $\Tcal_z \tMcal$ with respect to this twisted metric. This perspective also highlights a general guiding principle: an appropriate choice of Riemannian metric can induce an appropriate decomposition that aligns naturally with the underlying structure of the problem.
\end{remark}

We now are able to compute the directional derivative $\varphi'(z;v)$ at $z=(x,y)\in \Xbb\times\mathbb{S}^n$ for an arbitrary direction $v\in \Xbb\times\mathbb{S}^n$. The next proposition provides an explicit formula and its proof is deferred to Appendix~\ref{appendix:alg-global}.

\begin{proposition}\label{prop:dir-der-of-phi}
    For an arbitrary $z \in \mathbb{X} \times \mathbb{S}^n$ with $G(z)$ admitting an IED $(\alpha,\beta,\gamma,p,q,P,\lambda)$ and any direction $v = (v_x,v_y) \in \mathbb{X} \times \mathbb{S}^n$, the directional derivative of $\varphi$ at $z$ along $v$ evaluates to
    \begin{equation}\label{eq:dir-der-of-phi}
        \varphi'(z; v) = \left\langle dF_z^* \bigl(F(z)\bigr), \begin{bmatrix} v_x \\ H_1 \end{bmatrix} \right\rangle + \left\langle \nabla g(x)^* \bigl(F_1(z)\bigr), \Pi_{\mathbb{S}^n_-}(H_2) \right\rangle + \left\langle \nabla g(x)^* \bigl(F_1(z)\bigr) + F_2(z), \Pi_{\mathbb{S}^n_+}(H_2) \right\rangle,
    \end{equation}
    where $F_1(z) := \nabla_x L(x,y) = \nabla f(x) + \nabla g(x)y$, $F_2(z) := -g(x)+\Pi_{\Sbb^n_+}(G(z)) $, and  $H_1$ and $H_2$ are defined in Lemma~\ref{lemma:decomp-of-XxY}.
\end{proposition}

Proposition~\ref{prop:dir-der-of-phi} provides the algebraic foundation to rigorously characterize the D-stationarity of $\varphi$. We first define two critical normal directions. Let $z = (x, y) \in \widetilde{\mathcal{M}}_{p,q}$ with an IED $(\alpha, \beta, \gamma, p, q, P, \lambda)$. Recalling the tangent-normal orthogonal split of $\mathbb{S}^n$, we define the respective projection operators mapping an arbitrary matrix $H \in \mathbb{S}^n$ onto these subspaces:
\begin{equation}\label{eqn:pi^2}
    \pi_{z, \mathbb{S}^n}^1(H) := P \begin{bmatrix}
        \widetilde{H}_{\alpha \alpha} & \widetilde{H}_{\alpha \beta} & \widetilde{H}_{\alpha \gamma} \\
        \widetilde{H}_{\beta \alpha} & 0 & \widetilde{H}_{\beta \gamma} \\
        \widetilde{H}_{\gamma \alpha} & \widetilde{H}_{\gamma \beta} & \widetilde{H}_{\gamma \gamma}
    \end{bmatrix} P^\top
    \quad \text{and} \quad
    \pi_{z, \mathbb{S}^n}^2(H) := P \begin{bmatrix}
        0 & 0 & 0 \\
        0 & \widetilde{H}_{\beta \beta} & 0 \\
        0 & 0 & 0
    \end{bmatrix} P^\top,
\end{equation}
where $\widetilde{H} = P^\top H P$. Using these projections, we further define two specific normal descent directions $W_1(z)$ and $W_2(z) \in \mathcal{N}_{G(z)} \mathcal{M}_{p,q}$:
\begin{equation}\label{eq:def-W_1}
    W_1(z) := \Pi_{\mathbb{S}^n_-}\Bigl( -\pi_{z, \mathbb{S}^n}^2 \bigl( \nabla g(x)^* F_1(z) \bigr) \Bigr)
\end{equation}
and
\begin{equation}\label{eq:def-W_2}
    W_2(z) := \Pi_{\mathbb{S}^n_+}\Bigl( -\pi_{z, \mathbb{S}^n}^2 \bigl( \nabla g(x)^* F_1(z) + F_2(z) \bigr) \Bigr).
\end{equation}

\begin{proposition}\label{prop:char_of_dstat}
    Recalling the definitions of $F_1$ and $F_2$ in Proposition \ref{prop:dir-der-of-phi}, a point $z = (x,y) \in \mathbb{X} \times \mathbb{S}^n$ is a D-stationary point of $\varphi$ if and only if the following three conditions hold simultaneously:
    \begin{itemize}
        \item[(a)] $dF_z^*\bigl(F(z)\bigr) = 0$;
        \item[(b)] $W_1(z) = 0$, or equivalently, $\pi_{z,\mathbb{S}^n}^2 \bigl(\nabla g(x)^*F_1(z)\bigr) \in \mathbb{S}^n_-$;
        \item[(c)] $W_2(z) = 0$, or equivalently, $\pi_{z,\mathbb{S}^n}^2 \bigl(\nabla g(x)^* F_1(z) + F_2(z)\bigr) \in \mathbb{S}^n_+$.
    \end{itemize}
    
\end{proposition}
\begin{proof}
    Because $dF_z^*\bigl(F(z)\bigr) \in \mathcal{T}_z\widetilde{\mathcal{M}}$, the inner product $\bigl\langle dF_z^* \bigl(F(z)\bigr), [v_x^\top, H_1^\top]^\top \bigr\rangle \ge 0$ holds for all valid tangent pairs if and only if $dF_z^* \bigl(F(z)\bigr) = 0$. Furthermore, by the geometric properties of metric projection onto the closed convex cone $\mathbb{S}^n_+$, the condition $\bigl\langle \nabla g(x)^* F_1(z), \Pi_{\mathbb{S}^n_-}(H_2) \bigr\rangle \ge 0$ holds universally for all $H_2$ if and only if the normal projection $\pi_{z,\mathbb{S}^n}^2 \bigl(\nabla g(x)^* F_1(z)\bigr)$ strictly resides in $\mathbb{S}^n_-$. This geometrically forces $W_1(z) = 0$. A similar argument applied to the third term ensures $W_2(z) = 0$. \ep
\end{proof}

\begin{remark}
    Notice that condition (a) in Proposition~\ref{prop:char_of_dstat} corresponds exactly to the classical first order stationarity condition for the stratum-restricted subproblem of~\eqref{prog:global-P}.
\end{remark}

The next theorem provides conditions under which a D-stationary point is also a KKT pair, thereby connecting our global convergence result with the local convergence rate analysis in Section~\ref{subsection: Convergence analysis}.

\begin{theorem}\label{thm:dstat-implies-kkt}
    Let $z=(x,y)\in \widetilde{\mathcal{M}}_{p,q}$ with an IED $(\alpha,\beta,\gamma,p,q,P,\lambda)$ of $G(z)$. If the SOSC and the SRCQ hold at $z$, then the following are equivalent:
    \begin{itemize}
        \item[(a)] $z$ is a D-stationary point of $\varphi(z)$;

        \item[(b)] $z$ is a KKT pair of~\eqref{prog:SDP}.
    \end{itemize}
\end{theorem}

\begin{proof}
Obviously, a KKT pair is a D-stationary point. Now we suppose $z$ is D-stationary for $\varphi$, Proposition~\ref{prop:char_of_dstat} implies that
\begin{equation}\label{eq:dstat-sys-rewrite}
    dF_z^*\bigl(F(z)\bigr)=0,\quad W_1(z)=0\quad \mbox{and}\quad  W_2(z)=0.
\end{equation}
Recall that we write $F(z) = (F_1(z); F_2(z))$ in Proposition~\ref{prop:dir-der-of-phi}, and here we further abbreviate it as $(F_1; F_2)$ for simplicity. Define $U:=\nabla g(x)^*F_1\in \mathbb{S}^n$. Recalling the orthogonal decomposition
$\mathbb{S}^n=\mathcal{T}_{G(z)}\mathcal{M}_{p,q}\oplus
\mathcal{N}_{G(z)}\mathcal{M}_{p,q}$
given in~\eqref{eq:decomp-of-Y}, we decompose
\[
    U=U_{\mathcal{T}}+U_{\mathcal{N}}
    \quad \mbox{and} \quad
    F_2=(F_2)_{\mathcal{T}}+(F_2)_{\mathcal{N}}.
\]
In view of the IED $(\alpha,\beta,\gamma,p,q,P,\lambda)$ of $G(z)$,
the normal components admit the block representations
\begin{equation}\label{eq:UN-F2N}
U_{\mathcal{N}}
=
P\begin{bmatrix}
    0 & 0 & 0 \\
    0 & N & 0 \\
    0 & 0 & 0
\end{bmatrix}P^\top
\quad \mbox{and} \quad
(F_2)_{\mathcal{N}}
=
P\begin{bmatrix}
    0 & 0 & 0 \\
    0 & M & 0 \\
    0 & 0 & 0
\end{bmatrix}P^\top,
\end{equation}
where $N$ and $M$ are given by
\[
    N:=(P^\top U P)_{\beta\beta}\in \mathbb{S}^{|\beta|}
    \quad \mbox{and} \quad
    M:=(P^\top F_2 P)_{\beta\beta}\in \mathbb{S}^{|\beta|}.
\]
By Proposition~\ref{prop:char_of_dstat}, $W_1(z)=0$ is equivalent to
$\pi_{z,\mathbb{S}^n}^2(U)\in \mathbb{S}^n_-$, i.e., $N\preceq 0$.
Similarly, $W_2(z)=0$ is equivalent to
$\pi_{z,\mathbb{S}^n}^2(U+F_2)\in \mathbb{S}^n_+$, i.e., $N+M\succeq 0$.

Expanding the equation $dF_z^*(F_1,F_2)=0$ using~\eqref{eq:dF_z(v)} (and the definition of the adjoint) gives
\begin{align}
(\nabla^2_{xx}L(z)-\nabla g(x)\nabla g(x)^*)F_1-\nabla g(x)F_2&=0, \label{eq:adj-expanded-1}\\
U_{\mT}+\xi_{G(z)}\bigl(F_2\bigr)&=0. \label{eq:adj-expanded-2}
\end{align}
Taking the inner product of~\eqref{eq:adj-expanded-1} with $F_1$ and noting that $U = \nabla g(x)^* F_1$ yields
\begin{equation}\label{eq:energy-start}
\langle F_1,\nabla^2_{xx}L(z)F_1\rangle-\langle U,\,U+F_2\rangle=0.
\end{equation}
Using the orthogonal decomposition, we have
\begin{equation}\label{eq:energy-start2}
\langle U,\,U+F_2\rangle=\langle U_{\mathcal{T}},\,U_{\mathcal{T}}+(F_2)_{\mathcal{T}}\rangle
+\langle U_{\mathcal{N}},\,U_{\mathcal{N}} + (F_2)_{\Ncal} \rangle.
\end{equation}
Then, substituting \eqref{eq:adj-expanded-2} and \eqref{eq:energy-start2} into \eqref{eq:energy-start}, together with the relations
$\langle U_{\mathcal{T}},(F_2)_{\mathcal{N}}\rangle=0$
and
$\langle U_{\mathcal{N}},U_{\mathcal{N}}+(F_2)_{\mathcal{N}}\rangle=\langle N,N+M\rangle$,
both of which follow from \eqref{eq:UN-F2N}, we arrive at 
\begin{equation}\label{eq:energy-identity}
\langle F_1,\nabla^2_{xx}L(z)F_1\rangle+\langle \xi_{G(z)}(F_2),\,F_2-\xi_{G(z)}(F_2)\rangle
-\langle N,\,N+M\rangle=0.
\end{equation}
Since $N\preceq 0$ and $N+M\succeq 0$, we know
\begin{equation}\label{eq:NNM-nonneg}
\langle -N,\,N+M\rangle\ge 0.
\end{equation}
As \eqref{eq:adj-expanded-2} implies $\bigl(-\nabla g(x)^* (-F_1) + \xi_{G({z})}(F_2)\bigr)_{\alpha\gamma} = 0$, by Lemma~\ref{lemma:identity-Q(v)} and~\eqref{eq:NNM-nonneg}, we deduce from~\eqref{eq:energy-identity} that
\begin{equation}\label{eq:Q-nonpos}
\mQ_z(-F_1) \le 0,
\end{equation}
where $\mQ_z$ is the quadratic form defined in \eqref{eq:def-Qz}.

Next we show that $F_1=0$. We first claim that $-F_1$ lies in the critical cone $\mathcal{C}(z)$  associated with~\eqref{prog:SDP} (see \eqref{eq:def-C-nonKKT} for definition).
Indeed, by~\eqref{eq:adj-expanded-2} and the explicit form of $\xi_{G(z)}$ in~\eqref{eq:xi_A}, we know $[P^\top (\nabla g(x)^*(-F_1))P]_{\beta\gamma}=0$ and
$[P^\top (\nabla g(x)^*(-F_1))P]_{\gamma\gamma}=0$.
Moreover, since $\pi_{z,\mathbb{S}^n}^2(U)\in \mathbb{S}^n_-$, $(P^\top(\nabla g(x)^*(-F_1))P)_{\beta\beta}$ is positive semidefinite. Hence $-F_1\in\mathcal{C}(z)$. Therefore, by SOSC (Definition \ref{def:SOSC-nonKKT}) and \eqref{eq:Q-nonpos}, we have $F_1=0$.

Next we show that $F_2=0$. Suppose to the contrary that $F_2\neq 0$. Since $F_1=0$, we have $U=\nabla g(x)^*F_1=0$ and hence $N=0$.
From $N+M\succeq 0$ it follows that $M\succeq 0$, namely, $(P^\top F_2P)_{\beta\beta}\succeq 0$. Moreover, \eqref{eq:adj-expanded-1} reduces to
\begin{equation}\label{eq:F2-kernel}
    \nabla g(x)F_2=0.
\end{equation}
In addition, \eqref{eq:adj-expanded-2} becomes $\xi_{G(z)}(F_2)=0$.
By the explicit form of $\xi_{G(z)}$ in~\eqref{eq:xi_A},
this implies that the $\alpha$-blocks of $P^\top F_2P$ vanish.

By the SRCQ condition in Definition~\ref{def:SRCQ-nonKKT},
there exist $x_0\in\mathbb X$ and
$Y_0\in \Bigl\{P B P^\top\in\mathbb{S}^n \,\Big|\,
B_{\beta\beta}\succeq 0,\ B_{\beta\gamma}=0,\ B_{\gamma\gamma}=0\Bigr\}$
such that $\nabla g(x)^*x_0+Y_0=-F_2$. Taking the inner product with $F_2$ gives
\[
0>\langle -F_2, F_2\rangle
=
\langle x_0, \nabla g(x)F_2\rangle
+
\langle Y_0, F_2\rangle.
\]
Using \eqref{eq:F2-kernel}, the first term vanishes.
Since $Y_0\in \Bigl\{ P B P^\top\in\mathbb{S}^n \,\Big|\,
B_{\beta\beta}\succeq 0,\ B_{\beta\gamma}=0,\ B_{\gamma\gamma}=0\Bigr\}$
and $F_2$ has zero $\alpha$-blocks, a positive semidefinite (and nonzero) $\beta\beta$-block, and arbitrary other blocks consistent with symmetry, we have $\langle Y_0, F_2 \rangle \ge 0$.
Consequently, $0>\langle -F_2, F_2\rangle\ge 0$. The contradiction proves $F_2=0$.
Hence $F(z)=0$, and the proof is complete. \ep
\end{proof}

\begin{remark}\label{remark:thm10-nontrivial}
    Although the classical SOSC and SRCQ ensure local optimality and uniqueness of the multiplier, here they are formulated for an arbitrary primal--dual pair (see Definition \ref{def:SOSC-nonKKT} and  Definition \ref{def:SRCQ-nonKKT}), and hence do not imply that the pair is a KKT pair. As illustrated in Remark~\ref{remark:extended-not-KKT}, these two conditions can hold simultaneously at non-KKT pairs. The content of Theorem~\ref{thm:dstat-implies-kkt} is that D-stationarity of the merit function $\varphi$ provides precisely the additional first order information needed to close the gap.
\end{remark}

Motivated by the exact geometric equivalence established in Proposition~\ref{prop:char_of_dstat}, we naturally quantify proximity to D-stationarity via the norms of these critical vectors.

\begin{definition}\label{def:d-sta-mea}
    For any $z \in \mathbb{X} \times \mathbb{S}^n$, the \emph{D-stationarity measure} of $\varphi$ at $z$ is the scalar function
    \begin{equation*}
        s(z) := \max \bigl\{ \| W_1(z)\|,\ \| W_2(z) \|,\ \| v^{\mathrm{LM}}(z) \| \bigr\},
    \end{equation*}
    where $W_1(z)$ and $W_2(z)$ are defined in~\eqref{eq:def-W_1} and~\eqref{eq:def-W_2}, and $v^{\mathrm{LM}}(z)$ is the stratum-LM direction formally introduced in~\eqref{eq:SLM-direction}.
\end{definition}

\begin{remark}
    Because the definitions of $W_1$, $W_2$, and $v^{\mathrm{LM}}$ depend intrinsically on the active tangent space $\mathcal{T}_z\widetilde{\mathcal{M}}_{p,q}$, the measure $s(z)$ is not globally continuous. As the eigenvalues of $G(z)$ cross zero, the dimension of the tangent space changes discretely, inducing structural jumps in the stationarity measure.
\end{remark}

\subsubsection{The SLMN descent step}\label{section:sub-SLMN}

Globalizing the algorithm requires identifying descent directions for the unconstrained problem~\eqref{prog:global-P}. Motivated by the stratified geometry of $\mathbb{X} \times \mathbb{S}^n$, we utilize two distinct classes of directions: \emph{tangential directions} that reduce the residual while locked onto the current stratum, and \emph{normal directions} that force transitions across strata. Combining these yields the \emph{stratum LM-normal step} ($\mathrm{SLMN}$). At an iterate $z^k$, the $\mathrm{SLMN}$ step integrates these components to guarantee monotonic descent of the merit function $\varphi(z)=\tfrac12\|F(z)\|^2$. The procedure is formalized in Algorithm~\ref{alg:SLMN}.

\vspace{1ex}
\noindent\textbf{Tangential direction.}
We first compute a search step that navigates the current stratum $\widetilde{\mathcal{M}}_{p,q}$ to minimize the residual. We consider the strictly stratum-restricted subproblem:
\begin{equation}\label{prog:stratum-P}
    \min_{z \in \widetilde{\mathcal{M}}_{p,q}} \quad \varphi(z).
\end{equation}
A standard Gauss--Newton step minimizes the linearized quadratic model of $\varphi$ over the local tangent space:
\begin{equation}\label{prog:GN-P}
\begin{aligned}
    \min_{v \in \mathcal{T}_{z^k}\widetilde{\mathcal{M}}_{p,q}} \quad & \varphi_k(v) := \frac{1}{2}\bigl\|F(z^k) + dF_{z^k}(v)\bigr\|^2.
\end{aligned}
\end{equation}
This convex quadratic program resolves to the normal equation:
\begin{equation}\label{eq:GN-dif}
    dF_{z^k}^* dF_{z^k}(v) = -dF_{z^k}^*F(z^k), \quad v \in \mathcal{T}_{z^k}\widetilde{\mathcal{M}}_{p,q}.
\end{equation}
By Theorem~\ref{thm:WR-dF-nons}, if the W-SOC and W-SRCQ hold at $z^k$, \eqref{eq:GN-dif} admits a unique solution. To ensure global convergence, we apply Levenberg--Marquardt (LM) regularization when far from optimal regions.

\begin{definition}
    The \emph{stratum-LM direction} for $F$ at $z \in \mathbb{X} \times \mathbb{S}^n$ is defined as
    \begin{equation}\label{eq:SLM-direction}
        v^{\mathrm{LM}}(z) := -\bigl(\mu(z) I + dF_{z}^* dF_{z}\bigr)^{-1} dF_{z}^* F(z),
    \end{equation}
    where $\mu(z) := \|F(z)\|^2$ serves as the dynamic regularization parameter, and $I$ denotes the identity operator on $\mathcal{T}_z\widetilde{\mathcal{M}}$.
\end{definition}

\begin{remark}\label{remark:local-rate-LMGN}
     {Under the standing $C^2$ assumption, the stratum LM iteration for~\eqref{prog:stratum-P} converges superlinearly when it starts sufficiently close to a KKT pair $\overline z$ on $\tMcal$ and the W-SOC and W-SRCQ hold at $\overline z$. If $f$ and $g$ are $LC^2$ near $\overline x$, the rate is quadratic. See~\cite{yamashita2001rate} for classical Levenberg--Marquardt methods and Proposition~\ref{prop:quad-rate} for the estimates used below.}
\end{remark}

 {Once the stratum-LM direction is computed, the tangential candidate is $R_{z^k}(t v^k)$, where the step length $t>0$ is chosen so that $(z^k,tv^k)$ lies in the domain of the local retraction defined below.}

For a general manifold $\mathcal{M}$, a classical candidate for retraction is the Riemannian exponential map, which was used in~\cite{smith2014optimization} to establish the local quadratic convergence of the Riemannian Newton method. However, computing the Riemannian exponential directly is often intractable due to the need to solve geodesic equations explicitly. A more practical alternative is the metric projection onto $\mathcal{M}$, which is well-defined and locally smooth, provided the projection is unique~\cite{vandereycken2013lowrank,absil2008optimization}.
 In this work, rather than directly applying the metric projection on $\tMcal_{p,q}$, we introduce an induced retraction. Specifically, we first define a retraction on $\Mcal_{p,q}$ using the metric projection, and then extend it to $\tMpq$. This approach leverages the computational efficiency of the metric projection on $\Mcal_{p,q}$, which can be computed {by local spectral truncation}, instead of projecting onto the more complicated manifolds $\tMcal_{p,q}$. It should be noted that while a direct projection-based retraction is feasible, the induced retraction also offers computational benefits, particularly for evaluating $\varphi(\cdot)$, which is essential for line search procedures.

\begingroup
 For $\mathcal M=\mathcal M_{p,q}$ and $\tMcal=\tMpq$, let $\mathcal U_{\mathcal M}$ be an open neighborhood of the zero section in $\Tcal\mathcal M$ on which the nearest-point projection of $A+H$ onto $\mathcal M$ is unique and smooth~\cite[Lemma~4 and Proposition~5]{absil2012projection}. Define
\[
    R_{\mathcal M}(A,H)=\underset{M\in\mathcal M}{\argmin}\|M-(A+H)\|^2,
    \qquad (A,H)\in\mathcal U_{\mathcal M}.
\]
Set $\mathcal U_{\tMcal}=\{(z,v)\in\Tcal\tMcal\mid(G(z),\nabla g(x)^*v_x+v_y)\in\mathcal U_{\mathcal M}\}$. For $(z,v)\in\mathcal U_{\tMcal}$, define
\begin{equation}\label{retraction_tM}
    R_{\tMcal}(z,v)=(x+v_x,R_{\mathcal M}(A,H)-g(x+v_x)),
\end{equation}
where $A=G(z)$ and $H=\nabla g(x)^*v_x+v_y$. We write $R_z(v)=R_{\tMcal}(z,v)$ when the stratum is clear.
\par\endgroup

\begin{remark}
    For readers familiar with fiber bundles, the retraction $R_{\tMcal}$ on $\tMcal$ can be interpreted as performing retraction along both the fiber and base directions of the fiber bundle induced by the submersion $G: \tMcal \to \Mcal$.
\end{remark}

\begin{proposition}\label{prop:R_M-smooth-retraction}
     {Under the standing assumption $g\in C^2$, the mapping $R_{\tMcal}:\mathcal{U}_{\tMcal}\to\tMcal$ is a local $C^1$ retraction. For each fixed $z\in\tMcal$, the map $v\mapsto R_z(v)$ is $C^2$ on its domain in $\Tcal_z\tMcal$, with $R_z(0)=z$ and $d(R_z)_0=\mathrm{id}_{\Tcal_z\tMcal}$.}
\end{proposition}
\begin{proof}
     {The identities $R_{\mathcal M}(A,0)=A$ and $d(R_{\mathcal M,A})_0=\mathrm{id}$ give $R_z(0)=z$ and $d(R_z)_0=\mathrm{id}$ by~\eqref{retraction_tM}. The same formula is $C^1$ jointly in $(z,v)$ and $C^2$ in $v$ for fixed $z$, since $g\in C^2$.}
    \ep
\end{proof}

\vspace{1ex}
\noindent{\bf Normal direction.}
The stratum LM direction remains within the current stratum, but to move between strata, we shall follow normal directions to escape from a ``bad'' stratum that does not contain a solution of~\eqref{prog:global-P}. These directions come from \eqref{eq:def-W_1} and \eqref{eq:def-W_2} we defined above.

Since $W_i(z) \in \Ncal_{G(z)} \Mpq$, then by the definition of $\phi_z$ in \eqref{eq:phi_z} and definition of $\Ncal_z \tMpq$ in Lemma \ref{lemma:decomp-of-XxY}, an explicit calculation shows
\[
 \left[ \begin{array}{c} 0 \\ W_i(z) \end{array} \right] = \phi_z^{-1}\left( \left[ \begin{array}{c} 0 \\ W_i(z) \end{array} \right] \right) \in \Ncal_z \tMcal_{p,q} \subseteq \mathbb{X} \times \mathbb{S}^n,
\]
and we can discuss the step $z + \left[ \begin{array}{c} 0 \\ W_i(z) \end{array} \right]$. These normal directions $\left[ \begin{array}{c} 0 \\ W_i(z) \end{array} \right]$ are used to transition to higher-dimensional strata. The following lemma characterizes how $\varphi(z)$ varies along these directions.

\begin{lemma}\label{lemma:normal-control}
    Let $z = (x,y) \in \mathbb{X} \times \mathbb{S}^n$. If $W_1(z) \neq 0$, then $\|\nabla g(x)W_1(z)\| \neq 0$, and the exact line search minimum satisfies
    \begin{equation*}
        \mathop{\mathrm{argmin}}_{t \ge 0} \ \varphi \left(z + \begin{bmatrix} 0 \\ tW_1(z) \end{bmatrix} \right) = \frac{\| W_1(z)\|^2}{\| \nabla g(x)W_1(z)\|^2},
    \end{equation*}
    achieving a guaranteed reduction:
    \begin{equation*}
        \varphi(z) - \min_{t \ge 0 } \varphi \left(z + \begin{bmatrix} 0 \\ tW_1(z) \end{bmatrix} \right) = \frac{1}{2} \frac{\| W_1(z)\|^4}{\| \nabla g(x)W_1(z)\|^2}.
    \end{equation*}
    Analogously, if $W_2(z) \neq 0$, the optimal stepsize and the corresponding reduction are
    \begin{equation*}
        \mathop{\mathrm{argmin}}_{t \ge 0} \ \varphi \left(z + \begin{bmatrix} 0 \\ tW_2(z) \end{bmatrix} \right) = \frac{\| W_2(z)\|^2}{\| W_2(z)\|^2 + \| \nabla g(x)W_2(z)\|^2}
    \end{equation*}
    and
    \begin{equation*}
        \varphi(z) - \min_{t \ge 0} \varphi \left(z + \begin{bmatrix} 0 \\ tW_2(z) \end{bmatrix} \right) = \frac{1}{2} \frac{\| W_2(z)\|^4}{\| W_2(z)\|^2 + \| \nabla g(x)W_2(z)\|^2}.
    \end{equation*}
\end{lemma}

Since $\varphi \left(z + t\left[\begin{array}{c}
      0 \\
      W_i (z)
    \end{array}\right] \right)$ is a quadratic function of $t$, the conclusion is obtained by direct computation again. Its detailed proof is deferred to Appendix \ref{appendix:alg-global}. With the construction of the directions, the following is the precise form of our descent step algorithm \SLMN($z^k$).

\begin{algorithm}[H]
\caption{Stratum LM-normal step ($\mathrm{SLMN}$)}
\label{alg:SLMN}
\begin{algorithmic}[1]
\State \textbf{Input:} Current iterate $z^k \in \mathbb{X} \times \mathbb{S}^n$, line search parameters $\eta \in (0, 1/2)$ and $\rho \in (0, 1)$.
\vspace{0.3em}
\State \textbf{Compute Candidate Steps:}
    \begin{itemize}[label=\textbullet, leftmargin=*]
        \item If $W_1^k := W_1(z^k) \neq 0$ (see~\eqref{eq:def-W_1}), construct the first normal step:
            \begin{equation*}
                z^{k+1}_1 = z^k + \frac{\|W_1^k \|^2}{\| \nabla g(x^k)W_1^k \|^2} \begin{bmatrix} 0 \\ W_1^k \end{bmatrix}.
            \end{equation*}
        \item If $W_2^k := W_2(z^k) \neq 0$ (see~\eqref{eq:def-W_2}), construct the second normal step:
            \begin{equation*}
                z^{k+1}_2 = z^k + \frac{\| W_2^k\|^2}{\| W_2^k\|^2 + \| \nabla g(x^k)W_2^k\|^2} \begin{bmatrix} 0 \\ W_2^k \end{bmatrix}.
            \end{equation*}
        \item Perform a manifold Armijo backtracking line search along the tangent direction $v^k := v^{\mathrm{LM}}(z^k)$ (see~\eqref{eq:SLM-direction}).  {Find the smallest integer $j_k\ge0$ for which $(z^k,\rho^{j_k}v^k)\in\mathcal{U}_{\tMcal}$ and the following Armijo inequality holds:}
            \begin{equation*}
                \varphi\bigl(R_{z^k}(\rho^{j_k} v^k)\bigr) - \varphi(z^k) \le \eta \rho^{j_k} \varphi'(z^k; v^k),
            \end{equation*}
            and construct the retracted tangent step $z^{k+1}_3 = R_{z^k}(\rho^{j_k} v^k)$.
    \end{itemize}
\vspace{0.3em}
\State \textbf{Output:} Select the candidate yielding the maximal objective reduction:
\begin{equation*}
    \mathrm{SLMN}(z^k) \leftarrow \mathop{\arg\min}_{z \in \{z^{k+1}_1, z^{k+1}_2, z^{k+1}_3 \}} \varphi(z),
\end{equation*}
where undefined branches are strictly excluded from the comparison.
\end{algorithmic}
\end{algorithm}

In Algorithm~\ref{alg:SLMN}, $z^{k+1}_1$ and $z^{k+1}_2$ are computed based on Lemma~\ref{lemma:normal-control}, resulting in a decrease in the function value. That $W_1^k \ne 0$ implies $\nabla g(x^k) W_1^k \ne 0$ guarantees that $z_1^{k+1}$ is well-defined. The line search step clearly leads to a decrease in $\varphi$, provided that $j_k$ exists, which is proved in Proposition \ref{prop:armijo}.  Therefore, we can conclude that \SLMN\ always results in a decrease of $\varphi$ at non-D-stationary points.

\subsubsection{The correction strategy}

While Algorithm~\ref{alg:SLMN} navigates efficiently within a given stratum, traversing near boundaries naturally obscures the differential information characterizing the optimal adjacent stratum. To overcome this issue, we incorporate an eigenvalue correction strategy proposed in~\cite{feng2025quadratically}, drawing structural inspiration from globalization techniques for smooth functions on stratified spaces~\cite{Olikier2023FirstorderOO}. Provided the target solution satisfies adequate regularity, this operator controls the sequence into the correct active stratum.

Given a threshold $\delta > 0$ and an iterate $z = (x, y) \in \mathbb{X} \times \mathbb{S}^n$ with an IED $(\alpha, \beta, \gamma, p, q, P, \lambda)$ of $G(z)$, we define the corrected point by
\begin{equation}\label{eq:hat-z}
    \widehat{z} := \biggl(x,\ y - \sum_{i \in \theta(z)} \lambda_i\bigl(G(z)\bigr) P_i P_i^\top \biggr),
\end{equation}
where the threshold-active index set is defined as $\theta(z) := \bigl\{ i \mid |\lambda_i(G(z))| \le \delta \bigr\}$.

We observe that the operator $\widehat{\cdot}$ can be viewed as a locally smooth mapping from $\Xbb \times \Sbb^n$ to itself, a property that follows directly from the smoothness of the L\"owner operator.

Crucially, the following lemma establishes that the decrement terms derived in the preceding two parts constitute the dominant contribution, thereby ensuring the validity of our subsequent arguments.

\begin{lemma}\label{lemma:corr-almost-proj}
    Suppose a sequence $\{z^k\}$ converges to a limit $z^\infty \in \widetilde{\mathcal{M}}$.  {Assume that the threshold $\delta>0$ is strictly smaller than the absolute value of every nonzero eigenvalue of $G(z^\infty)$.} Then, for sufficiently large $k$, we have
    \begin{equation*}
        \| \widehat{z}^k - z^k \| \le \mathcal{O}\bigl(\| z^k -z^\infty \|\bigr),
    \end{equation*}
    which directly implies
    \begin{equation*}
        \| \widehat{z}^k - z^\infty \| \le \mathcal{O}\bigl(\|z^k -z^\infty \|\bigr).
    \end{equation*}
\end{lemma}

\begin{proof}
    Define the respective index sets $\alpha^k = \alpha(G(z^k))$, $\widehat{\alpha}^k = \alpha(G(\widehat{z}^k))$, and $\alpha^\infty = \alpha(G(z^\infty))$. Analogous definitions apply for $\beta$ and $\gamma$. Because the correction $\widehat{z}^k$ precisely preserves the $x$-component of $z^k$, their distance reduces to
    \begin{equation}\label{eq:z-zhat-1}
        \|\widehat{z}^k - z^k\| = \bigl\| \bigl(x^k, G(\widehat{z}^k) - g(x^k)\bigr) - \bigl(x^k, G(z^k) - g(x^k)\bigr) \bigr\| = \| G(\widehat{z}^k) - G(z^k) \|.
    \end{equation}
    Decompose the respective matrices using the bases $P^k$ and $P^\infty$:
    \begin{equation*}
        G(z^k) = P^k \begin{bmatrix}
            \Lambda^k_{\alpha^k} & 0 & 0 \\
            0 & 0_{\beta^k} & 0 \\
            0 & 0 & \Lambda^k_{\gamma^k}
        \end{bmatrix} (P^k)^\top \quad \text{and} \quad
        G(z^\infty) = P^\infty \begin{bmatrix}
            \Lambda^{\infty}_{\alpha^\infty} & 0 & 0 \\
            0 & 0_{\beta^\infty} & 0 \\
            0 & 0 & \Lambda^{\infty}_{\gamma^\infty}
        \end{bmatrix} (P^\infty)^\top.
    \end{equation*}
    By construction, $G(\widehat{z}^k)$ matches $G(z^k)$ except that specific elements within $\Lambda^k_{\alpha^k}$ and $\Lambda^k_{\gamma^k}$ falling below the threshold $\delta$ are zeroed out. The residual difference is:
    \begin{equation*}
        G(z^k) - G(\widehat{z}^k) = P^k \begin{bmatrix}
            \widetilde{\Lambda}^k_{\alpha^k} & 0 & 0 \\
            0 & 0_{\beta^k} & 0 \\
            0 & 0 & \widetilde{\Lambda}^k_{\gamma^k}
        \end{bmatrix} (P^k)^\top.
    \end{equation*}
     {By the threshold assumption and eigenvalue continuity, for all sufficiently large $k$, precisely the eigenvalues indexed by $\beta^\infty$ are truncated, so $\widehat\beta^k=\beta^\infty$. This includes $G(z^\infty)=0$. The standard eigenvalue perturbation bound and local Lipschitz continuity of $G$ yield:}
    \begin{align}\label{eq:z-zhat-2}
        \| G(z^k) - G(\widehat{z}^k)\| &= \|\widetilde{\Lambda}^k\| = \bigl\| \Lambda_{\widehat{\beta}^k}\bigl(G(z^k)\bigr) - \Lambda_{\beta^\infty}\bigl(G(z^\infty)\bigr) \bigr\| \nonumber\\
        &\le \|\Lambda\bigl(G(z^k)\bigr) - \Lambda\bigl(G(z^\infty)\bigr)\| \le \| G(z^k) - G(z^\infty) \| \le \mathcal{O}\bigl(\| z^k-z^\infty \|\bigr).
    \end{align}
    Substituting~\eqref{eq:z-zhat-2} into~\eqref{eq:z-zhat-1} establishes the primary bound. Applying the standard triangle inequality $\| \widehat{z}^k - z^\infty \| \le \| z^k - \widehat{z}^k \| + \| z^k - z^\infty \|$ completes the proof.\ep
\end{proof}

\subsection{The convergence analysis}\label{subsection: Convergence analysis}

This subsection establishes the rigorous convergence analysis of Algorithm~\ref{alg:SGN}. The global convergence properties are stated below and we defer the technical proof to Appendix~\ref{appendix:alg-global}. To systematically quantify the spectrum, we define the minimal non-vanishing eigenvalue modulus for any $z \in \mathbb{X} \times \mathbb{S}^n$ in \eqref{eqn: minieigen} by
\begin{equation*}
    \delta(z) := \min \bigl\{ |\lambda_i(G(z))| \mid \lambda_i(G(z)) \neq 0 \bigr\}.
\end{equation*}

\begin{theorem}\label{thm:global-convergence}
    Let $\epsilon = 0$ and let $z^0$ be an arbitrary initial point. Then Algorithm~\ref{alg:SGN} either terminates finitely at a D-stationary point of $\varphi$ or generates an infinite sequence whose every accumulation point $z^\infty$ with $\delta(z^\infty)>\delta$ is a D-stationary point of $\varphi$.
\end{theorem}

\begin{remark}
    Theorem~\ref{thm:global-convergence} theoretically assumes that the hyperparameter $\delta$ is scaled relative to the unknown accumulation point. One could circumvent this dependency by deploying a \emph{partial correction} sweep over all admissible truncation thresholds (paralleling the methodology in~\cite{Olikier2023FirstorderOO}). However, while ensuring parameter-free globalization, such an approach introduces substantial complexity to the local convergence analysis. We prioritize mathematical clarity here, leaving hyperparameter-free implementations for future refinement.
\end{remark}

 We next establish the local {convergence rates} of the SGN method (Algorithm~\ref{alg:SGN}). The detailed proof is deferred to Appendix~\ref{appendix:alg-local}. The analysis relies fundamentally on the stratification geometry developed in Section~\ref{sec:geometric}. In particular, under the {SOSC (Definition~\ref{def:SOSC-nonKKT})} and the SRCQ (Definition~\ref{def:SRCQ-nonKKT}) at a D-stationary accumulation point, with a suitable correction threshold, we show that the iterates eventually identify the active stratum and then enter a locally {superlinearly convergent regime, with quadratic convergence if $f$ and $g$ are additionally of class $LC^2$ near $x^\infty$}. {Example~\ref{example:9.2} illustrates that the W-SOC together with the W-SRCQ alone does not ensure identification of the active stratum.}

\begin{theorem}\label{thm:local-quad-rate}
     Suppose $z^\infty$ is an accumulation point of the sequence generated by Algorithm~\ref{alg:SGN}. If the correction threshold bound satisfies $0 < \delta < \delta(z^\infty)$ and $z^\infty$ is a D-stationary point of $\varphi$ at which the {SOSC} and the SRCQ hold simultaneously, then the whole sequence converges {superlinearly} to $z^\infty$ in the sense that for all sufficiently large $k$,
    \begin{equation*}
        \|z^{k+1} - z^\infty\| \leq {o(\|z^{k} - z^\infty\|)}.
    \end{equation*}
    {If $f$ and $g$ are additionally of class $LC^2$ near $x^\infty$, then the convergence is quadratic.} Moreover, for all sufficiently large $k$, the sequence $\{z^k\}$ lies in the active stratum containing $z^\infty$.
\end{theorem}

\begin{remark}\label{remark:compare}
    The main difference between our method and conventional manifold optimization is that we do not solve a given smooth problem subject to a manifold constraint. Instead, we decompose a nonsmooth problem in Euclidean space into smooth subproblems on strata from a stratification perspective. As mentioned in the introduction, this idea is also central to the concept of partial smoothness~\cite{lewis2002active}, where smoothness is imposed only after restricting to an active manifold (termed ``restricted smoothness'' in~\cite[Definition~2.7(i)]{lewis2002active}). Partial smoothness, however, also requires a ``normal sharpness" property~\cite[Definition~2.7(iii)]{lewis2002active}. For the merit function $\varphi$ given in \eqref{prog:global-P}, this property fails, so the partial smoothness theory is not directly applicable here.
\end{remark}

\begin{remark}
    From a different perspective, Deng et al.~\cite{deng2025efficient} recently proposed a manifold based primal-dual semismooth Newton method for~\eqref{prog:SDP}. Their approach also combines the correction idea of~\cite{feng2025quadratically} to drive the sequence toward the active manifold. Superlinear convergence was established (see~\cite[Theorem~10]{deng2025efficient}) under an error bound condition together with an assumption that the iterates eventually lie on the active manifold. In contrast, the present approach does not require the error bound to hold over the entire ambient space, nor does it rely on the prior existence of iterates on the active stratum. Instead, by investigating the stratification, we have distilled an equivalent condition required for the convergence of Newton’s method applied to the manifold-restricted subproblem, which is shown to be significantly weaker than classical regularity conditions. Furthermore, by exploiting the variational properties induced by the stratification, the proposed algorithm, through a specific design, drives the iterates toward the active stratum and ultimately achieves manifold identification.
\end{remark}

Theorem~\ref{thm:global-convergence} yields D-stationarity of accumulation points satisfying the correction threshold condition. Combining this conclusion with Theorem~\ref{thm:local-quad-rate} gives the following corollary, without a separate stationarity or KKT assumption.

\begin{corollary}\label{cor:quad-conv}
     Suppose that $z^\infty$ is an accumulation point of the sequence generated by Algorithm~\ref{alg:SGN}. If the correction threshold satisfies $0<\delta<\delta(z^\infty)$, and both the SOSC and the SRCQ hold at $z^\infty$, then the whole sequence converges {superlinearly} to $z^\infty$. Moreover, $z^\infty$ is a KKT pair, and, for all sufficiently large $k$, the iterates $\{z^k\}$ lie in the active stratum containing $z^\infty$. {If $f$ and $g$ are additionally of class $LC^2$ near $x^\infty$, then the convergence is quadratic.}
\end{corollary}

\begin{proof}
    By Theorem~\ref{thm:global-convergence}, $z^\infty$ is a D-stationary point of $\varphi$. Together with the SOSC and SRCQ, this verifies the assumptions of Theorem~\ref{thm:local-quad-rate}, which gives active-stratum identification and superlinear convergence of the whole sequence, with quadratic convergence under the additional $LC^2$ assumption. Moreover, Theorem~\ref{thm:dstat-implies-kkt} implies that $z^\infty$ is a KKT pair. \ep
\end{proof}

\begin{remark}
    The significance of Corollary~\ref{cor:quad-conv} is that it bridges the gap between global and local analyses, which are typically treated separately even in the classical smooth Gauss--Newton literature. Indeed, for such methods applied to $\varphi(z) = \frac{1}{2}\|F(z)\|^2$ with smooth $F$, the global convergence usually yields merely the first-order condition $F'(z)^* F(z) = 0$, while the local superlinear or quadratic convergence is typically established only near zeros of $F$ under suitable regularity conditions \cite{dennis1996numerical,nocedal2006numerical,zhou2010global}. In our setting, global convergence likewise yields only D-stationarity, but the stratification framework enables us to characterize when such a point is already a KKT pair.
\end{remark}

\section{Numerical experiments}\label{sec:numerical}

\subsection{{Examples of fast local convergence}}

{In this subsection, we report the observed convergence behavior of Algorithm~\ref{alg:SGN} on two examples. The reported runs exhibit fast local convergence.} As reported in~\cite[Section 6]{feng2025quadratically}, the classical convergence theory for semismooth Newton methods is not applicable to these problems due to the degeneracy of the generalized Jacobian. In fact, even the method itself may not be executed.

The first example is adapted from~\cite[Example~6]{feng2025quadratically}. Although simple, it highlights a typical situation in which the classical SOSC fails because the solution set is not isolated,  {while fast local convergence is observed in the reported run}. In~\cite{feng2025quadratically}, the local quadratic convergence of a semismooth Newton method is demonstrated by selecting a specific Bouligand generalized Jacobian element $\mathcal{U}_0$. In contrast, Algorithm~\ref{alg:SGN} does not require an \emph{a priori} choice of generalized Jacobians and is designed to be globally convergent.

\begin{example}\label{example:block_SDP}

Consider the following convex quadratic SDP:
    \begin{equation}\label{prog:block_SDP}
    \begin{aligned}
        \min_{X \in \mathbb{S}^n} \quad & \frac{1}{2}\| X_{11}-I\|_F^2 + \frac{1}{2}\| X_{12}\|_F^2 + \frac{1}{2}\| X_{21}\|_F^2 \\
        \mathrm{s.t.} \quad & X \in \mathbb{S}^n_+,
    \end{aligned}
    \end{equation}
    where the variable is partitioned as $X = \begin{bmatrix} X_{11} & X_{12} \\ X_{21} & X_{22} \end{bmatrix}$ with $X_{11} \in \mathbb{S}^{l_1}$ and $X_{22} \in \mathbb{S}^{l_2}$. The optimal solution set is given by
    \begin{equation*}
        \mathcal{S} = \left\{ \begin{bmatrix} I & 0 \\ 0 & X_{22} \end{bmatrix} \mathrel{\Bigg|} X_{22} \in \mathbb{S}^{l_2}_+ \right\}.
    \end{equation*}
     At {the} optimal solution $\overline{X} {= \operatorname{Diag}(I,0)}$, the unique Lagrange multiplier is $\overline{Y} = 0$. One can verify that the W-SOC and constraint nondegeneracy hold at $(\overline{X}, \overline{Y})$, whereas the SOSC fails because the solution is not isolated.

    We test this instance with dimensions $l_1 = l_2 = 100$ and a correction threshold $\delta = 0.2$, initialized from a randomly generated starting point. Figure~\ref{fig:block_SDP} explicitly demonstrates both the global convergence and the asymptotic quadratic convergence rate of Algorithm~\ref{alg:SGN}.
    \begin{figure}[H]
    \centering
    \includegraphics[width=1\textwidth]{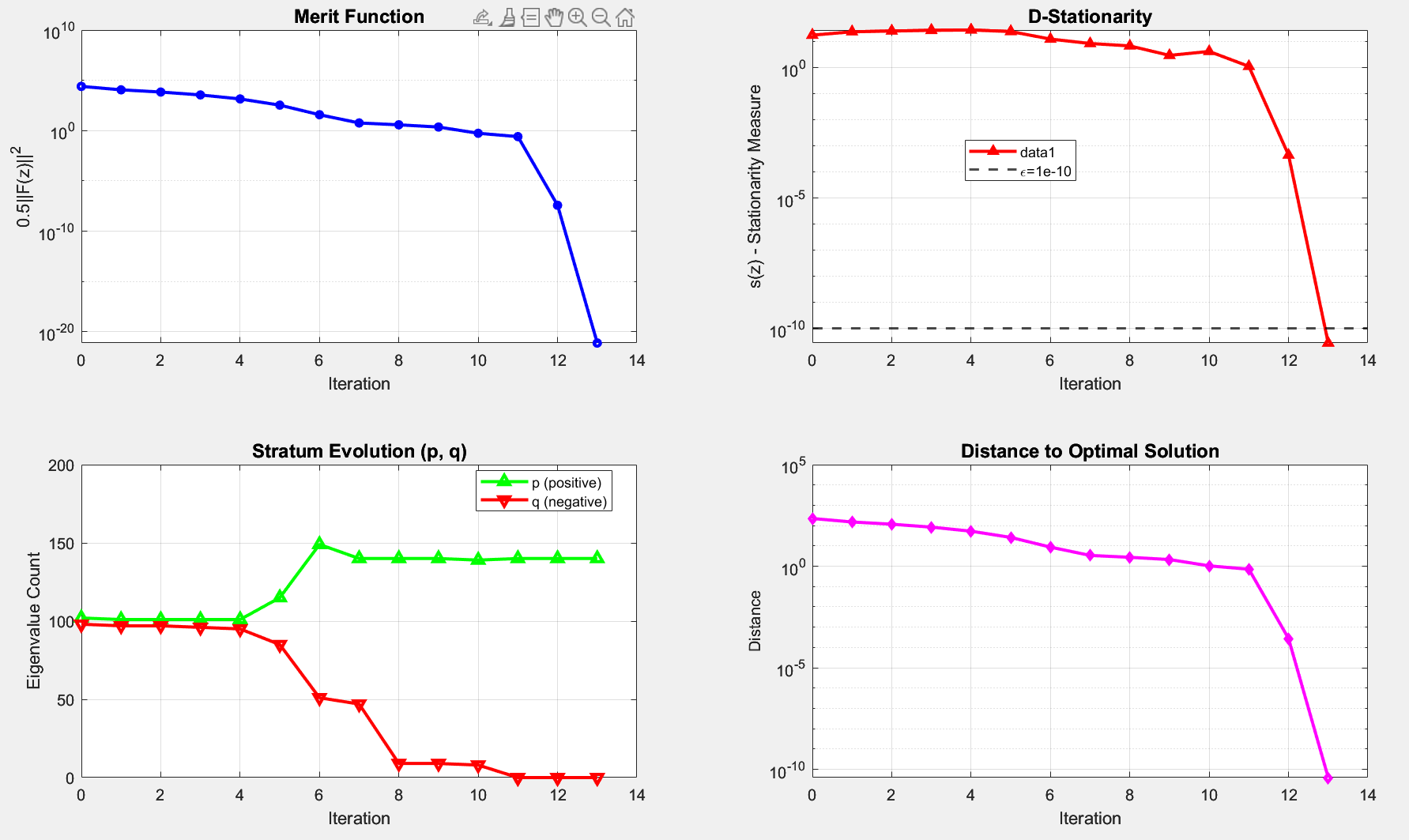}
    \caption{Convergence of Algorithm~\ref{alg:SGN} for Example~\ref{example:block_SDP}.}
    \label{fig:block_SDP}
    \end{figure}
    Two features are worth noting. First, the upper-right panel of Fig. \ref{fig:block_SDP} indicates that the stationarity measure $s(z)$ (Definition~\ref{def:d-sta-mea}) need not be monotone, even though the merit function is forced to decrease by the acceptance test.
    Second, the lower-left panel of Fig. \ref{fig:block_SDP} shows that the stratum-identification phase may exhibit a visible plateau: the iterates move slowly as the algorithm explores nearby strata. Although the iterates evolve only marginally during this phase, the eventual identification of the correct stratum triggers a sharp acceleration in convergence. This behavior matches the design goal of SGN: identification enables the method to exploit stratum-restricted regularity even in the absence of classical regularity.
\end{example}

The second example is adapted from~\cite[Example~5]{feng2025quadratically} (see also~\cite[Example~3]{ding2017characterization} and~\cite{zhou2017unified}). It includes an additional inequality constraint, which is handled by the product stratification discussed in Remark~\ref{remark:add_SDPs}.

\begin{example}\label{example:SI}
Let $B = \begin{bmatrix} 3/2 & -2 \\ -2 & 3 \end{bmatrix}$ and $b = B^{-1/2} \begin{bmatrix} 5/2 \\ -1 \end{bmatrix}$. Consider the following dual SDP:
    \begin{equation}\label{prog:prob_example4_D}
    \begin{aligned}
        \max_{Y \in \mathbb{S}^2} \quad & -\frac{1}{2}\| \mathcal{A} Y - b \|^2 - \langle I, Y \rangle \\
        \mathrm{s.t.} \quad & \langle E, Y \rangle \le 1, \\
        & Y \in \mathbb{S}^2_+,
    \end{aligned}
    \end{equation}
    where the linear operator is defined as $\mathcal{A} Y = B^{1/2}\operatorname{Diag}(Y)$ for all $Y \in \mathbb{S}^2$, and $E$ is the matrix of all ones. As demonstrated in~\cite{ding2017characterization}, the resulting KKT solution mapping for this problem lacks calmness. The unique optimal solution is $\overline{Y} = \begin{bmatrix} 1 & 0 \\ 0 & 0 \end{bmatrix}$ with the unique corresponding multiplier $\overline{y} = (0,0) \in \mathbb{R} \times \mathbb{S}^2$. It is straightforward to verify that both the W-SOC and constraint nondegeneracy hold at the KKT pair $\overline{z} = (\overline{Y}, \overline{y})$, whereas the S-SOSC fails.

    We test this instance using $\delta = 0.2$, initializing from a randomly generated point sufficiently distant from the optimal solution to clearly observe the transition into the asymptotic quadratic convergence regime. Figure~\ref{fig:ex_SI} corroborates the global convergence and local quadratic convergence of Algorithm~\ref{alg:SGN}.
    \begin{figure}[H]
    \centering
    \includegraphics[width=1\textwidth]{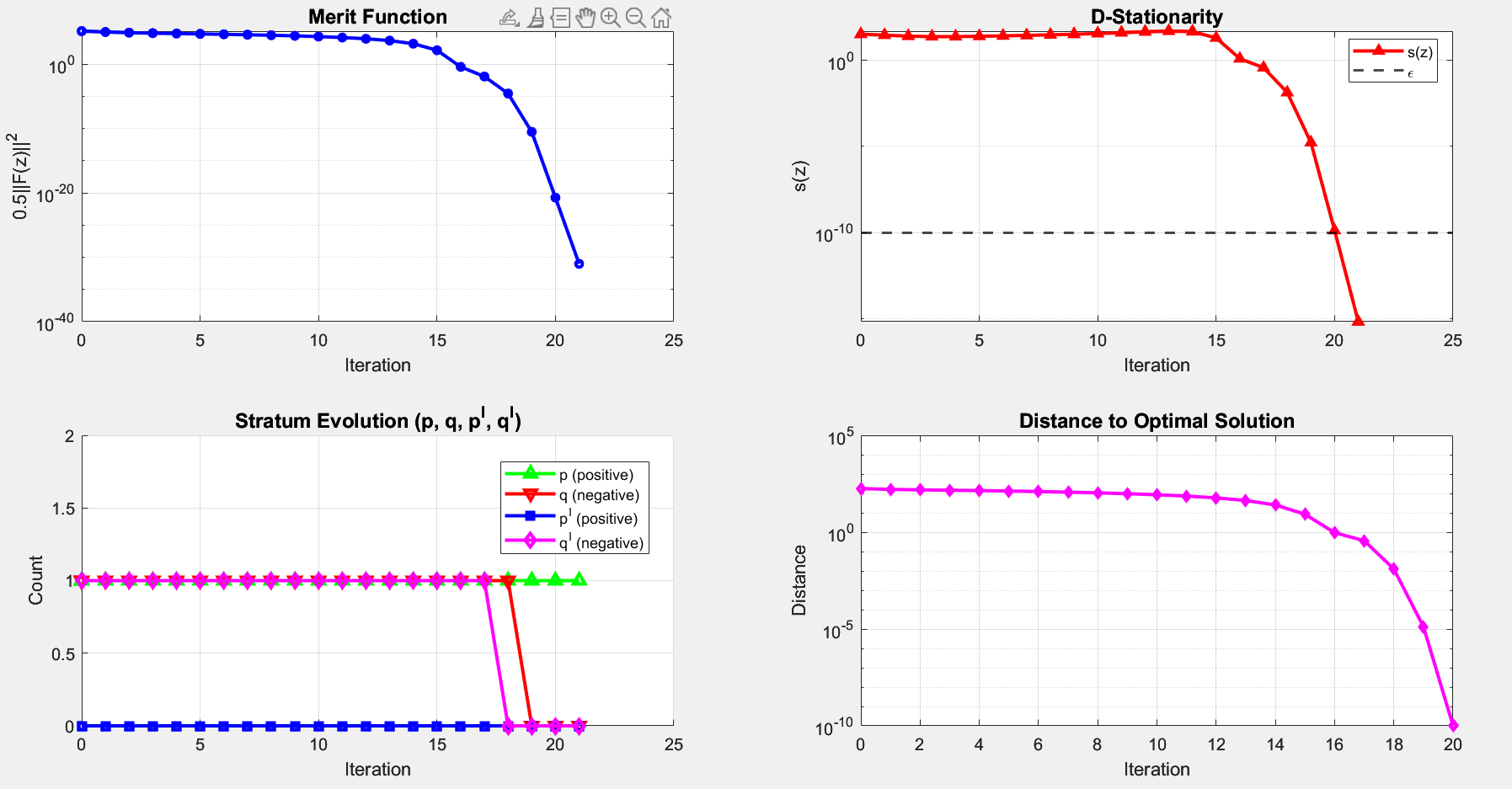}
    \caption{Convergence of Algorithm~\ref{alg:SGN} for Example~\ref{example:SI}.}
    \label{fig:ex_SI}
    \end{figure}

    In particular, the figure tracks the evolution of the active-stratum indices $(p,q)$ associated with the PSD constraint and $(p^I, q^I)$ associated with the inequality constraint. As evidenced by the lower-left panel of Figure~\ref{fig:ex_SI}, the iterates typically exhibit a clear phase transition: convergence is slow while the algorithm remains on incorrect strata, and the quadratic regime appears once the correct stratum is identified.
\end{example}

\subsection{An instance violating the proposed condition}

Next, we apply the proposed algorithm to the instance defined in Example~\ref{example:9.2}. As shown therein, the W-SOC and the W-SRCQ are satisfied at $(\overline{x},\overline{y})$, where the optimal solution is $\overline{x} = (0,0,0,0,0)^\top$ with the corresponding multiplier $\overline{y} = \operatorname{Diag}(0,0,0,-1)$. However, the SOSC and the SRCQ fail at this point. Consequently, the assumptions of Theorem~\ref{thm:local-quad-rate} are not satisfied, so this theorem does not provide a local superlinear or quadratic convergence guarantee for this instance.

More importantly, numerical experiments show that the convergence rate of Algorithm~\ref{alg:SGN} may degenerate to linear even when the initial point is sufficiently close to $(\overline{x},\overline{y})$, as illustrated in Figure~\ref{fig:ex9.2}.

\begin{figure}[H]
    \centering
    \includegraphics[width=1\textwidth]{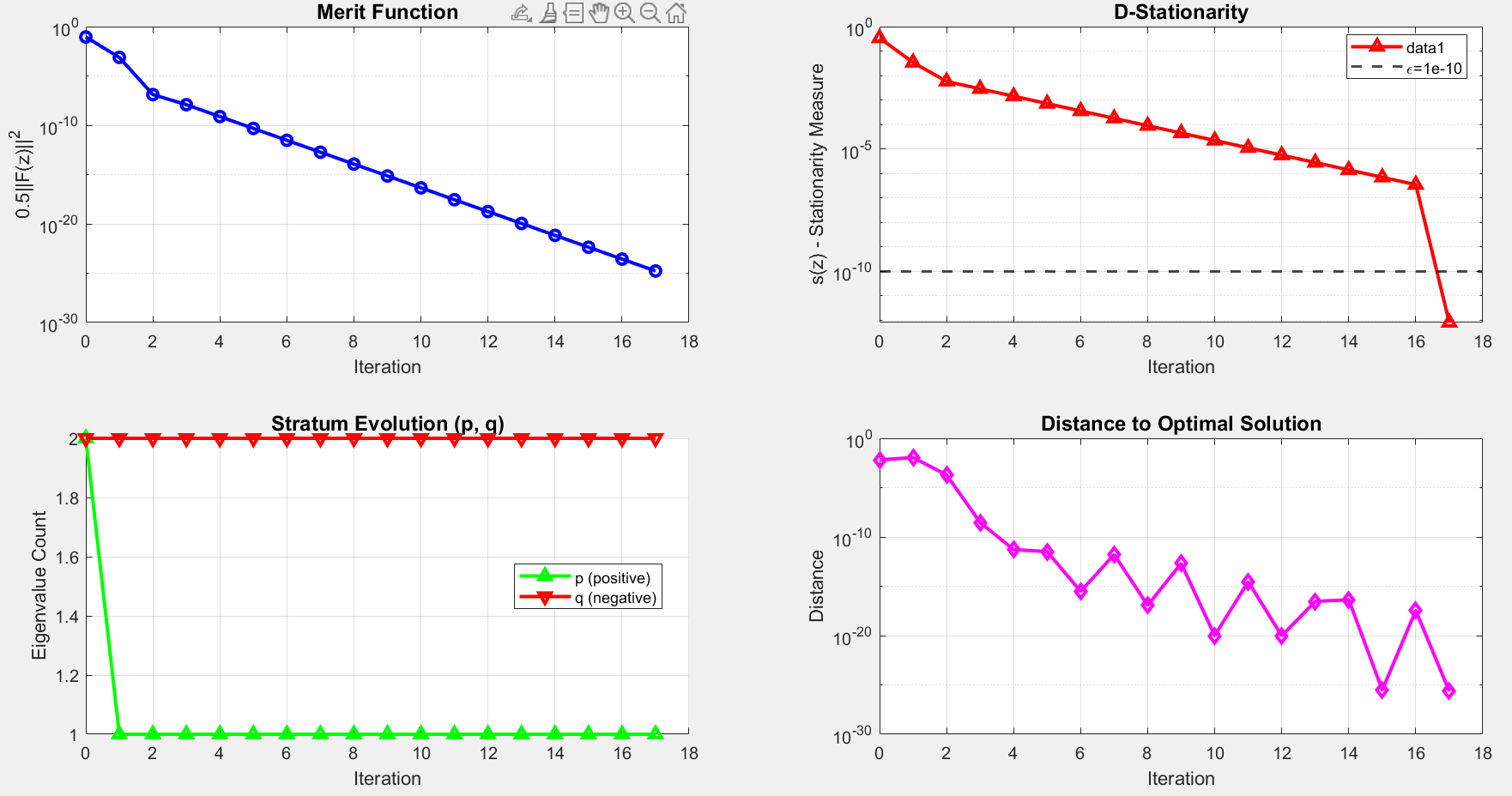}
    \caption{Convergence of Algorithm~\ref{alg:SGN} for Example~\ref{example:9.2}.}
    \label{fig:ex9.2}
\end{figure}

It can be observed that Algorithm~\ref{alg:SGN} fails to identify the correct stratum $\tMcal_{1,1}$ on which $(\overline{x},\overline{y})$ lies. This indicates that W-SRCQ and W-SOC alone are not sufficient to ensure quadratic convergence, even locally.  {The local convergence theorem uses the stronger pair SOSC and SRCQ.}

A more theoretical explanation for this behavior follows from a finer regularity analysis of~\eqref{prog:example-1}. Indeed, there exists a sequence of KKT pairs of~\eqref{prog:example-1} converging to $(\overline{x},\overline{y})$ for which the W-SOC fails. According to Remark~\ref{remark:dF_pCF}, quadratic convergence of semismooth Newton-type methods is generally not available in a neighborhood of such points. Hence, even when the initial point is sufficiently close to $(\overline{x},\overline{y})$, the generated iterates may drift toward these unfavorable nearby KKT pairs, in which case only linear convergence can be expected asymptotically.

Moreover, if the initial point is randomly chosen near $(\overline{x},\overline{y})$ but lies on the correct stratum $\tMcal_{1,1}$, quadratic convergence is also attained, as shown in Figure~\ref{fig:ex9.2_random_2}. This observation is consistent with the statement in Remark~\ref{remark:local-rate-LMGN}.

\begin{figure}[H]
    \centering
    \includegraphics[width=1\textwidth]{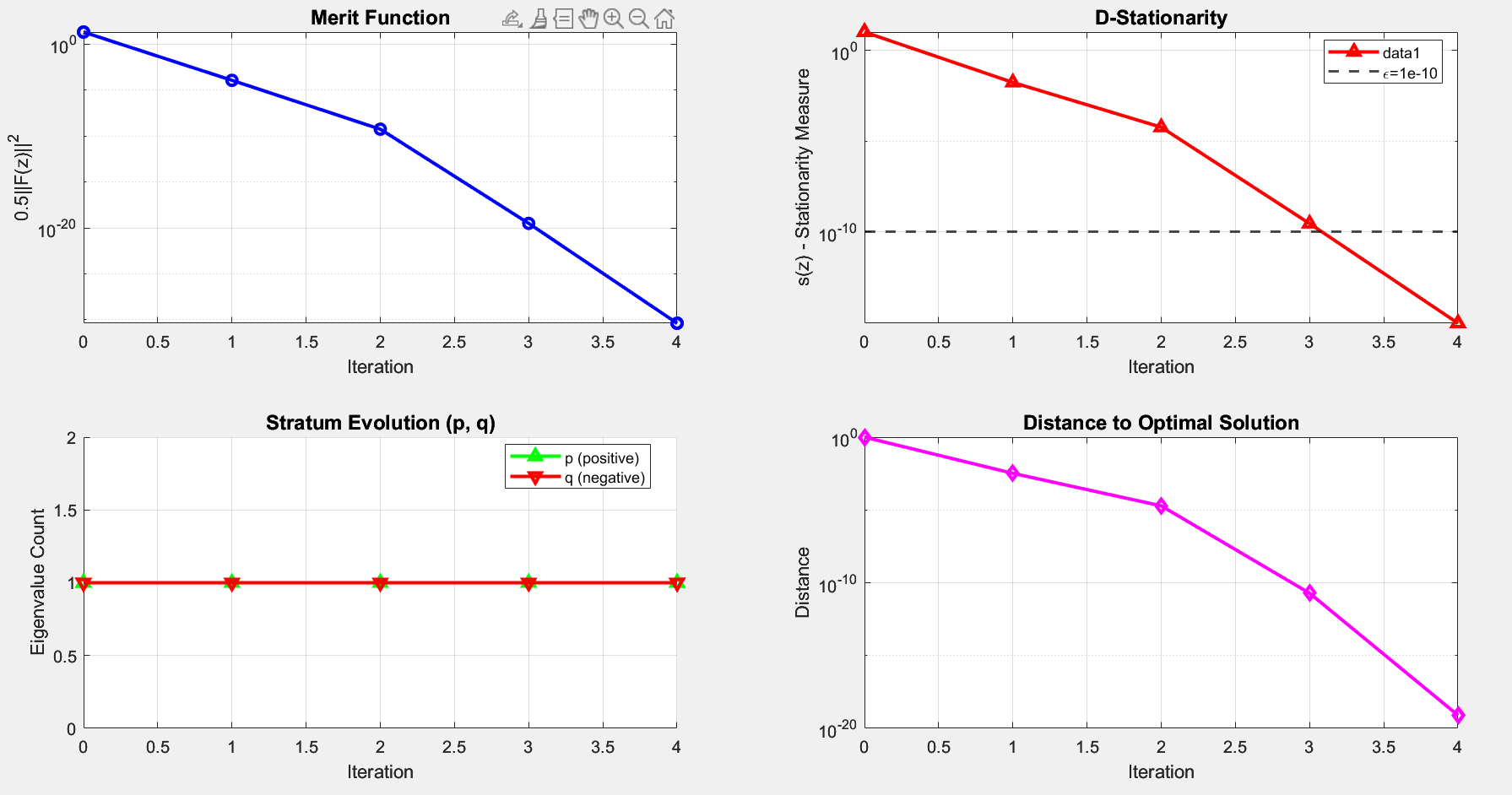}
    \caption{Initial point on stratum $\tMcal_{1,1}$}
    \label{fig:ex9.2_random_2}
\end{figure}

\section{Conclusion}\label{sect:conclusion}

This paper develops a stratification-based framework for nonlinear semidefinite programming. By decomposing $\mathbb{S}^n$ through the inertia stratification and lifting this structure to the primal-dual space, we reveal a hidden smooth geometry within the inherently nonsmooth KKT system: the metric projection onto $\mathbb{S}^n_{+}$ is $C^\infty$-smooth on each stratum, and the KKT mapping admits an explicit manifold differential when restricted to a fixed stratum. This leads to a clear separation between analysis \emph{within a stratum}, where standard smooth-manifold calculus applies, and analysis \emph{across strata}, where perturbations are governed by the adjacency structure of the stratification. Within this framework, we demonstrate that weak and verifiable problem-level regularity conditions are equivalent to solution-level strong metric regularity in a stratum-restricted sense. Furthermore, the classical strong-form regularity conditions are proven to be equivalent to the local uniform validity of these stratum-restricted properties near a local minimizer satisfying the RCQ. The transversality interpretation further clarifies the geometry of W-SRCQ and yields stability results on a fixed stratum that parallel the classical theory for constraint nondegeneracy. The stratification viewpoint also informs algorithm design. We propose a globally convergent stratified Gauss--Newton (SGN) method, combined with Levenberg--Marquardt steps with explicit normal moves and an eigenvalue-triggered correction mechanism. The resulting scheme guarantees descent of the least-squares merit function while allowing controlled transitions between strata.  Under mild assumptions, it converges globally to D-stationary points. If the {SOSC} and SRCQ hold at an accumulation point and the correction threshold is suitable, this point is a KKT pair, the whole sequence converges {superlinearly} to it, and the active stratum is eventually identified. {The local rate is quadratic if $f$ and $g$ are additionally of class $LC^2$ near the primal solution.}

Beyond its specific application to the KKT system, the proposed stratification framework may provide a new geometric perspective on NLSDP and, more broadly, on nonpolyhedral matrix optimization. From a theoretical standpoint, an important direction for future research is revisiting other central concepts in perturbation analysis, including the robust isolated calmness, the Aubin property, and other related concepts. We anticipate that leveraging this stratified geometric perspective, underpinned by stratum-wise smoothness, will yield sharper characterizations of these classical properties. On the algorithmic side, this stratification framework may suggest a new approach for degenerate regimes, where classical nondegeneracy assumptions are too restrictive, and it motivates extending the analysis beyond Newton-type schemes to other NLSDP solvers, notably augmented Lagrangian methods. Finally, it is of practical interest to remove the dependence on the eigenvalue-threshold parameter $\delta$, for instance, through partial-correction strategies that retain both robust globalization and transparent local rate guarantees.

\section*{Acknowledgements}
The authors would like to thank Professor Jianfeng Lin and the undergraduate students Yicheng Yang, Yuxuan Li, and Siyu Chen of Qiuzhen College for helpful discussions on techniques related to transversality. They are also grateful to Professor Nachuan Xiao for helpful discussions that brought \cite{Olikier2023FirstorderOO} to their attention, and to Doctor Shiwei Wang for valuable comments on this work.

\bibliographystyle{spmpsci}
\bibliography{Str-NLSDP-bibliography}

\section*{Declarations}
\vskip 10pt
\noindent{\bf Conflict of interest} All authors certify that they have no affiliations with or involvement in any organization or entity with any financial interest or non-financial interest in the subject matter or materials discussed in this manuscript.

\appendix
\normalsize

\renewcommand{\thesection}{\Alph{section}}
\section{Smoothness of $\Pi_{\Sbb^n_+}$ on the fixed-index manifold}\label{appendix:smooth}

In this appendix, we prove the $C^\infty$-smoothness of the metric projector $\Pi_{\mathbb{S}^n_+}$ when restricted to a fixed-index stratum $\mathcal{M}_{p,q}$, and we derive the differential formula in Theorem~\ref{thm:diff-PiK}.  Equivalently, one may view this as a stratum-wise smoothness result on the constant-rank manifold
\[
\mathcal{M}_r:=\{A\in \mathbb{S}^n \mid \operatorname{rank}(A)=r\},
\qquad r=1,\ldots,n,
\]
since $\mathcal{M}_r$ decomposes into $r+1$ connected components
\[
\mathcal{M}_r=\bigcup_{p=0}^r \mathcal{M}_{p,r-p}.
\]
For notational simplicity, we suppress the indices $p,q$ and write $\mathcal{M}$ for the fixed stratum under consideration.

 {On each fixed stratum $\widetilde{\mathcal{M}}_{p,q}$, the normal mappings $W_1$ and $W_2$ defined in Proposition~\ref{prop:char_of_dstat} are locally Lipschitz, as established in Lemma~\ref{lemma: W-smooth-on-M}. The induced retraction $R_{\widetilde{\mathcal{M}}}$ introduced in Section~\ref{section:sub-SLMN} has the local regularity specified in Proposition~\ref{prop:R_M-smooth-retraction}.} For notational brevity, we suppress the subscripts $p, q$ and denote the relevant manifolds simply as $\mathcal{M}$ and $\widetilde{\mathcal{M}}$, as the specific indices are mathematically immaterial to the local analysis. The subsequent proofs rely on analytic functional calculus. We briefly recall the relevant facts and refer to~\cite[Chapter~VII]{conway1994course} or~\cite[Sections~2.2--2.3]{simon2015operator} for background.

Recall that in complex analysis, the classical Cauchy integral formula provides an integral representation of a scalar analytic function $f(\zeta)$ evaluated at $\zeta \in \mathbb{C}$ via the contour integral:
\begin{equation*}
    f(\zeta) = \frac{1}{2\pi i} \int_{\Gamma} \frac{f(z)}{z - \zeta} \, dz,
\end{equation*}
where the contour $\Gamma$ encloses $\zeta$. This representation facilitates differentiation under the integral sign, thereby guaranteeing infinite differentiability. Analogously, for a symmetric matrix $A \in \mathbb{S}^n$, analytic functional calculus defines the matrix-valued function $f(A)$ for an analytic function $f$ via the contour integral:
\begin{equation*}
    f(A) = \frac{1}{2\pi i} \int_{\Gamma} f(z) (zI - A)^{-1} \, dz,
\end{equation*}
where the contour $\Gamma \subset \mathbb{C}$ strictly encloses the spectrum (i.e., the set of eigenvalues) of $A$, and $(zI - A)^{-1}$ denotes the resolvent matrix. This construction directly mirrors the scalar Cauchy formula. For a rigorous treatment, we refer the reader to the classical text on functional analysis by Conway~\cite[Chapter~VII]{conway1994course} or to Simon~\cite[Sections~2.2 and~2.3]{simon2015operator} for a detailed exposition.

\begin{appthm}[{\cite[Chapter~VII]{conway1994course}}]\label{thm:functional_calculus}
    Let $f$ be an analytic function on an open set $\Omega \subseteq \mathbb{C}$ containing the spectrum of $A \in \mathbb{S}^n$, and let the boundary $\Gamma = \partial \Omega$  enclose this spectrum. Then, the matrix-valued function
    \begin{equation*}
        f(A) = \frac{1}{2\pi i} \int_{\Gamma} f(z) (zI - A)^{-1} \, dz
    \end{equation*}
    depends analytically on $A$. In particular, the mapping $A \mapsto f(A)$ is $C^\infty$-smooth and locally expressible as a convergent power series on the domain of symmetric matrices whose spectra lie within $\Omega$.
\end{appthm}

To establish the smoothness of the metric projection, we construct three auxiliary functions. Fix $A \in \mathcal{M}$. Since the inertia indices $|\alpha(A)|$, $|\beta(A)|$, and $|\gamma(A)|$ are structurally constant on the stratum $\mathcal{M}$, we can select three mutually disjoint open disks $D_1, D_2, D_3 \subset \mathbb{C}$ enclosing the positive eigenvalues, the negative eigenvalues, and the origin, respectively. Define $D := D_1 \cup D_2 \cup D_3$. By the continuity of eigenvalues, there exists a sufficiently small neighborhood of $A$ in $\mathcal{M}$ such that the positive, negative, and zero parts of the spectrum for any matrix in this neighborhood remain within $D_1$, $D_2$, and $D_3$, respectively. We define the following scalar functions on $D$:
\begin{equation*}
    f_+(z) := \begin{cases} z, & z \in D_1, \\ 0, & z \in D_2 \cup D_3, \end{cases} \quad
    f_{\dagger}(z) := \begin{cases} 1/z, & z \in D_1 \cup D_2, \\ 0, & z \in D_3, \end{cases} \quad \text{and} \quad
    p_0(z) := \begin{cases} 0, & z \in D_1 \cup D_2, \\ 1, & z \in D_3. \end{cases}
\end{equation*}
Since $D_1$, $D_2$, and $D_3$ are disjoint, these piecewise definitions are analytic on the open set $D$. Applying Theorem~\ref{thm:functional_calculus}, the induced matrix functions $f_+(A)$, $f_{\dagger}(A)$, and $p_0(A)$ are well-defined and depend analytically on $A$ within this local neighborhood. We are now in a position to prove Theorem~\ref{thm:diff-PiK}.

\begin{proof}[of Theorem \ref{thm:diff-PiK}]
    For any $A \in \Mcal \subseteq \Sbb^n$, consider the local exponential chart with respect to the induced Riemannian metric on the embedded submanifold $\Mcal$ in $\Sbb^n$
    $$\exp_A : \mathcal{U} \to \mathcal{U}'$$
    which maps from a neighborhood $\mathcal{U} \subseteq \mathcal{T}_A\mathcal{M}$ of the origin to a neighborhood $\mathcal{U}' \subseteq \mathcal{M}$ of $A$. Since the induced matrix-valued function $M \mapsto f_+(M)$ is analytic (and hence $C^\infty$-smooth) on an open subset of $\mathbb{S}^n$ containing $\mathcal{U}'$, the composition $v \mapsto f_+\bigl(\exp_A(v)\bigr)$ smoothly maps $\mathcal{U}$ into $\mathbb{S}^n$. For any matrix $\widetilde{A} \in \mathcal{U}' \cap \mathcal{M}$, it is clear that $\Pi_{\mathbb{S}^n_+}(\widetilde{A}) = f_+(\widetilde{A})$, since both operations identically isolate the positive eigenspace of the matrix from others. Consequently, the local coordinate representation $v \mapsto \Pi_{\mathbb{S}^n_+}\bigl(\exp_A(v)\bigr)$ is smooth, which guarantees that the restriction $\Pi_{\mathbb{S}^n_+}|_{\mathcal{M}} : \mathcal{M} \to \mathbb{S}^n$ is $C^\infty$-smooth.

    Next, we shall derive the explicit manifold differential formula \eqref{eq:xi_A}. For any $A \in \Mcal$ and tangent vector $H\in {\Tcal}_{A} \Mcal$, the directional derivative calculated in \cite{sun2002semismooth} dictates that
    \begin{equation*}
        \Pi_{\mathbb{S}^n_+}^{\prime}(A; H) = P \begin{bmatrix}
            \widetilde{H}_{\alpha \alpha} & \widetilde{H}_{\alpha \beta} & \Xi_{\alpha \gamma} \circ \widetilde{H}_{\alpha \gamma} \\
            \widetilde{H}_{\alpha \beta}^\top & \Pi_{\mathbb{S}^{|\beta|}_+}(0) & 0 \\
            \widetilde{H}_{\alpha \gamma}^\top \circ \Xi_{\alpha \gamma}^\top & 0 & 0
        \end{bmatrix} P^\top
        = P \begin{bmatrix}
            \widetilde{H}_{\alpha \alpha} & \widetilde{H}_{\alpha \beta} & \Xi_{\alpha \gamma} \circ \widetilde{H}_{\alpha \gamma} \\
            \widetilde{H}_{\alpha \beta}^\top & 0 & 0 \\
            \widetilde{H}_{\alpha \gamma}^\top \circ \Xi_{\alpha \gamma}^\top & 0 & 0
        \end{bmatrix} P^\top,
    \end{equation*}
    where $\widetilde{H} = P^\top H P$. The second equality strictly follows from the geometric restriction $\widetilde{H}_{\beta \beta} = 0$ established in Proposition~\ref{prop:TAM} for tangent vectors, which vanishes the nonsmooth block since $\Pi_{\mathbb{S}^{|\beta|}_+}(0) = 0$. Because the restriction of $\Pi_{\mathbb{S}^n_+}$ to the smooth manifold $\mathcal{M}$ is smooth, its directional derivative  coincides with its Fr\'echet differential on the manifold, yielding:
    \begin{equation*}
        d\bigl(\Pi_{\mathbb{S}^n_+}|_{\mathcal{M}}\bigr)_A(H) = \Pi_{\mathbb{S}^n_+}^{\prime}(A; H).
    \end{equation*}
    This completes the proof. \ep
\end{proof}

 {The next lemma gives the normal projection formula~\eqref{eqn:pi^2} and its local $C^2$ regularity on a fixed stratum.}

\begin{applemma}\label{applemma:pi-smooth}
    For $z \in \widetilde{\mathcal{M}}_{p,q}$, define the piecewise constant (and thus analytic) indicator function
    \begin{equation*}
        p_0(w) := \begin{cases}  0, & w \in D_1 \cup D_2, \\
        1, & w \in D_3, \end{cases}
    \end{equation*}
    where the mutually disjoint open disks $D_1, D_2$, and $D_3$ enclose the positive, negative, and zero eigenvalues of $G(z)$, respectively. Then, for any fixed matrix $H \in \mathbb{S}^n$, the normal projection operator $\pi_{z, \mathbb{S}^n}^2 : \mathbb{S}^n \to \mathcal{N}_{G(z)} \mathcal{M}_{p,q}$ defined according to Lemma~\ref{lemma:decomp-of-XxY} satisfies the algebraic identity:
    \begin{equation*}
        \pi_{z, \mathbb{S}^n}^2(H) = p_0\bigl(G(z)\bigr) H p_0\bigl(G(z)\bigr).
    \end{equation*}
     {Under the standing assumption $g\in C^2$, the map $z\mapsto\pi_{z,\mathbb{S}^n}^2(H)$ is locally $C^2$ on the fixed stratum $\widetilde{\mathcal{M}}_{p,q}$ for every fixed $H\in\mathbb{S}^n$.}
\end{applemma}
\begin{proof}
    It suffices to verify the matrix identity $\pi_{z, \mathbb{S}^n}^2(H) = p_0\bigl(G(z)\bigr) H p_0\bigl(G(z)\bigr)$. Let $(\alpha,\beta,\gamma,p,q, P, \lambda)$ be an IED of $G(z)$, and block-partition the orthogonal eigenvector matrix as $P = \begin{bmatrix} P_\alpha & P_\beta & P_\gamma \end{bmatrix}$. The columns of $P_{\alpha}$, $P_{\beta}$, and $P_{\gamma}$ form orthonormal bases for the eigenspaces of $G(z)$ corresponding to the positive, zero, and negative eigenvalues, respectively. By the properties of analytic functional calculus, $p_0(G(z))$, the so-called Riesz projector \cite[Theorem 2.3.3]{simon2015operator}, exactly reconstructs the orthogonal projector onto the zero eigenspace of $G(z)$. Thus, we have $p_0\bigl(G(z)\bigr) = P_\beta P_\beta^\top$. By~\eqref{eqn:pi^2}, the projection operator dictates:
    \begin{equation*}
        \pi_{z, \mathbb{S}^n}^2(H) = P \pi_{\beta\beta}(P^\top H P) P^\top,
    \end{equation*}
    where the inner block-extraction operator $\pi_{\beta\beta}$ is defined by
    \begin{equation*}
        \pi_{\beta\beta}(X):=\pi_{\beta\beta}\left( \left[\begin{array}{ccc}
    X_{\alpha \alpha} & X_{\alpha \beta} & X_{\alpha \gamma} \\
    X_{\beta \alpha} & X_{\beta \beta} & X_{\beta \gamma}\\
    X_{\gamma \alpha} & X_{\gamma \beta} & X_{\gamma \gamma}
    \end{array}\right] \right) = \left[\begin{array}{ccc}
    0 & 0 & 0\\
    0 & X_{\beta \beta} & 0\\
    0 & 0 & 0
    \end{array}\right] = \left[\begin{array}{ccc}
    0 & 0 & 0\\
    0 & I & 0\\
    0 & 0 & 0
    \end{array}\right] X \left[\begin{array}{ccc}
    0 & 0 & 0\\
    0 & I & 0\\
    0 & 0 & 0
    \end{array}\right], \quad X\in\mathbb{S}^n
    \end{equation*}
    Substituting this back into the expression for $\pi_{z, \mathbb{S}^n}^2(H)$ yields
    $$\pi_{z,{\Sbb^n}}^2(H) = P \pi_{\beta \beta}(P^\top H P) P^\top = P \left[\begin{array}{ccc}
    0 & 0 & 0\\
    0 & I_{|\beta|} & 0\\
    0 & 0 & 0
    \end{array}\right] (P^\top H P) \left[\begin{array}{ccc}
    0 & 0 & 0\\
    0 & I_{|\beta|} & 0\\
    0 & 0 & 0
    \end{array}\right] P^\top = (P_{\beta} P_{\beta}^\top) H (P_{\beta} P_{\beta}^\top) = p_0(G(z)) H p_0(G(z)),$$
     {which proves the identity. The regularity follows by composing the locally analytic spectral projector with $G\in C^2$, keeping the spectral groups fixed near the reference point.}
    \ep
\end{proof}

\section{Details on transversality}\label{appendix:transverse}

In this appendix, we prove the transversality and stability results stated in Section~\ref{subsubsec: trans}.

We begin by recalling the concept of a tubular neighborhood. Geometrically, a tubular neighborhood provides a localized vector-bundle structure around a submanifold, ensuring that normal vectors uniquely parameterize nearby points.

\begin{appdef}[{\cite[Chapter 6]{lee2012introduction}}]
Let $\mathbb{Y}$ be a Euclidean space and  $\mathcal{S} \subset \mathbb{Y}$ be an embedded submanifold. The \emph{normal bundle} $\mathcal{N} \mathcal{S}$ of $\mathcal{S}$ in $\mathbb{Y}$ is the vector bundle over $\mathcal{S}$ defined fiberwise by
    \begin{equation*}
        \mathcal{N}_x \mathcal{S} := (\mathcal{T}_x \mathcal{S})^\perp \quad \text{for each } x \in \mathcal{S},
    \end{equation*}
    where the orthogonal complement is taken with respect to the standard inner product on $\mathbb{Y}$.
\end{appdef}

In our specific setting, we identify the ambient space $\mathbb{Y}$ with $\mathbb{S}^n$ and the submanifold $\mathcal{S}$ with $\mathcal{M}_{p,q}$. Recall from Proposition~\ref{prop:TAM} that the tangent space is explicitly characterized by $\mathcal{T}_A \mathcal{M}_{p,q} = \bigl\{ H \in \mathbb{S}^n \mid (P^{\top} H P)_{\beta\beta} = 0 \bigr\}$, where $P$ is the orthogonal matrix associated with an IED of $A$. Consequently, under the trace inner product, the normal space consists of matrices supported exclusively on the $(\beta,\beta)$-block:
\begin{equation*}
    \mathcal{N}_A \mathcal{M}_{p,q} = \bigl\{ H \in \mathbb{S}^n \mid (P^{\top} H P)_{ij} = 0 \text{ for all } (i,j) \notin \beta \times \beta \bigr\}.
\end{equation*}
Accordingly, the normal bundle $\mathcal{N}\mathcal{M}_{p,q} := \bigcup_{A \in \mathcal{M}_{p,q}} \{ A \} \times \mathcal{N}_A\mathcal{M}_{p,q}$ admits a smooth vector bundle structure. Furthermore, the addition map $E: \mathbb{S}^n \times \mathbb{S}^n \to \mathbb{S}^n$, defined by $E(y_1, y_2) = y_1 + y_2$, acts as a diffeomorphism when restricted to a sufficiently small neighborhood of the zero section in $\mathcal{N} \mathcal{M}_{p,q}$. For a comprehensive treatment, we refer the reader to~\cite[Chapter~6]{lee2012introduction}.

\begin{appdef}[{\cite[Chapter 6]{lee2012introduction}}]\label{def:tubu-nbhd}
A \emph{tubular neighborhood} of a submanifold \( \Scal \) in \( \Ybb \) is an open neighborhood \( \Dcal\Scal \subseteq \Ybb \) of \( \Scal \), which is the image under the diffeomorphism \( E : \Wcal \to \Dcal\Scal \), where \( \Wcal \) is an open neighborhood of the $\Scal$ in \( \Ncal \Scal \).
\end{appdef}

Building upon the preceding results, we formally construct the geometry of $\Dcal \Ncal_z$.
The following lemma establishes that for any $z \in \tMpq$, there exists a sufficiently small open set $\Ucal \subset \Mpo$ such that the translation $\Ucal + \{ \Pi_{\Sbb^n_-}( G(z) ) - y \}$ forms a neighborhood of $g(x)$ within $\Ncal_{z}$, embedded in $\Mpq -y$.

\begin{lemmarep}{lemma:N_z-in-Mpq}
    Let $z \in \widetilde{\mathcal{M}}_{p,q}$ and set $A := G(z) \in \mathcal{M}_{p,q}$. Denote
    \begin{equation*}
        A_+ := \Pi_{\mathbb{S}^n_+}(A) \in \mathcal{M}_{p,0} \quad \text{and} \quad A_- := \Pi_{\mathbb{S}^n_-}(A) \in \mathcal{M}_{0,q},
    \end{equation*}
    so that $A = A_+ + A_-$. Then $\Ucal_z$ defined in \eqref{eqn:def-of-Uz} by
    \begin{equation}\label{eqn:def-of-Uz-2}
        \Ucal_z := \Mpo \cap \Bcal \left( {\Pi_{\mathbb{S}^n_+}(G(z))}, \frac{\delta(z)}{2} \right)
    \end{equation}
    is an open neighborhood of $A_+$ in $\mathcal{M}_{p,0}$ such that $\Ucal_z + A_-$ is smoothly embedded in $\mathcal{M}_{p,q}$,  where $\delta(z)$ is the the minimal non-vanishing eigenvalue modulus of $G(z)$ defined in \eqref{eqn: minieigen} by
    $$\delta(z):= \min \bigl\{ |\lambda_i(G(z))| \mid \lambda_i(G(z)) \neq 0 \bigr\}{\quad\text{if }G(z)\neq0}.$$
    {When $G(z)=0$, choose any $\delta(z)>0$. The} open ball $\Bcal(A,r)$ is defined by
    $$\Bcal(A,r) := \{ A' \in \Sbb^n \mid \| A-A' \|_2 < r \}.$$
\end{lemmarep}

\begin{proof}
    If $A=0$, then $p=q=0$ and the assertion is immediate. Suppose $A\ne0$.
    Let $A_+'\in\mathcal M_{p,0}\cap
    \Bcal(A_+,\delta(z)/2)$, so
    $\|A_+'-A_+\|_2<\delta(z)/2$. Then
    $\|(A_+'+A_-)-A\|_2=\|A_+'-A_+\|_2<\delta(z)/2$. By
    \cite[Corollary~III.2.6]{bhatia1997matrix},
    corresponding eigenvalues of $A_+'+A_-$ and $A$ differ by less than $\delta(z)/2$.  Since $A\in
    \mathcal{M}_{p,q}$ has $p$ eigenvalues $\geq\delta(z)$ and $q$ eigenvalues
    $\leq-\delta(z)$, it follows that $A_+'+A_-$ has at least $p$ positive
    eigenvalues and at least $q$ negative eigenvalues.

    For the upper bounds on the inertia indices: the null space of
    $A_+'\succeq 0$ has dimension $n-p$, and for every $v$ in it,
    $\langle v,(A_+'+A_-)v\rangle=\langle v,A_-v\rangle\leq 0$. By
    \cite[Theorem~III.1.2]{bhatia1997matrix},
    the matrix $A_+'+A_-$ has at most $p$ positive eigenvalues. Symmetrically, the null
    space of $A_-\preceq 0$ has dimension $n-q$, and
    $\langle v,(A_+'+A_-)v\rangle=\langle v,A_+'v\rangle\geq 0$ for all $v$
    in it, so $A_+'+A_-$ has at most $q$ negative eigenvalues.

    Thus $A_+'+A_-$ has exactly $p$ positive eigenvalues and exactly $q$ negative
    eigenvalues, so $A_+'+A_-\in\mathcal{M}_{p,q}$. The smooth embedding follows from the fact
    that translation by the fixed matrix $A_-$ is a diffeomorphism of
    $\mathbb{S}^n$.\ep
\end{proof}

Now we define a smooth family of submanifolds parametrized by $z = (x,y) \in \tMpq$ as follows:
\begin{equation}\nn
    \Ncal_{z}:= \Ucal_z + \Pi_{\Sbb^n_-}(G(z)) - y,
\end{equation}
where $\Ucal_z$ is defined in \eqref{eqn:def-of-Uz-2}. Consider the normal bundle of the shifted stratum $\mathcal{M}_{p,q} - y$ within the ambient space $\mathbb{S}^n$, defined as
\begin{equation*}
    \mathcal{N}(\mathcal{M}_{p,q} - y) := \bigcup_{A \in \mathcal{M}_{p,q}} \mathcal{N}_A (\mathcal{M}_{p,q} - y) \subseteq \mathbb{S}^n \times \mathbb{S}^n.
\end{equation*}
By Lemma \ref{lemma:N_z-in-Mpq},  $\Ncal_z$ is embedded in $\Mpq -y$. So we can restrict this bundle to $\Ncal_z$ and obtain a family of vector bundles, denoted by $\{ \Ecal_z := \Ncal (\Mpq-y) |_{\Ncal_z} \}_{z \in \tMpq}$. Furthermore, invoking the theory of tubular neighborhoods \cite[Chapter 6.4]{lee2012introduction}, the addition map $E: \mathbb{S}^n \times \mathbb{S}^n \to \mathbb{S}^n$, defined by $E(A,H) = A+H$, serves as a diffeomorphism from a neighborhood of the zero section in $\mathcal{N}\mathcal{M}_{p,q} $ onto a tubular neighborhood $\mathcal{D}\mathcal{M}$ of $\mathcal{M}_{p,q}$ in $\mathbb{S}^n$.
 We define the target family of extended submanifolds $\{ \mathcal{D}\mathcal{N}_{z} \}_{z \in \widetilde{\mathcal{M}}_{p,q}}$ via this restricted diffeomorphic image:
\begin{equation*}
    \mathcal{D}\mathcal{N}_{z} := {E\bigl(\{(A,H)\in\Ecal_z\mid\|H\|_2<\delta(z)/4\}\bigr)}.
\end{equation*}
 {Since $A+y\in\mathcal{M}_{p,q}$ and $\|A+y-G(z)\|_2<\delta(z)/2$, the nonzero eigenvalues of $A+y$ have modulus greater than $\delta(z)/2$, whereas $H$ is supported on its kernel and $\|H\|_2<\delta(z)/4$.}
{These separated spectral groups allow $A$ and $H$ to be recovered uniquely and smoothly from $A+H$, so $E$ restricts to a smooth embedding on this open subset of $\Ecal_z$.}
{This restriction} does not change the tangent space at $(g(x),0)$ and hence does not affect the transversality verification at the base point.
With the geometric structure of $\mathcal{D}\mathcal{N}_z$, we restate and prove Theorem~\ref{thm:trans=WSRCQ} below.

\begin{theoremrep}{thm:trans=WSRCQ}
    For the NLSDP~\eqref{prog:SDP}, let $z = (x, y) \in \mathbb{X} \times \mathbb{S}^n$ possess an IED $(\alpha, \beta, \gamma, p, q, P, \lambda)$ of $G(z)$ associated with $\mathcal{M}_{p,q}$. The following statements are equivalent:
    \begin{itemize}
        \item[(a)] $g \pitchfork_{x} \mathcal{D}\mathcal{N}_{z}$ in $\mathbb{S}^n$;
        \item[(b)] the W-SRCQ (Definition~\ref{def:W-SRCQ-nonKKT}) holds at $z = (x,y)$.
    \end{itemize}
\end{theoremrep}
\begin{proof}
    We begin with condition (a), which asserts the geometric transversality relation:
    \begin{equation}\label{eqn:transverse_of_gz}
        dg_{x}(\mathbb{X}) + \mathcal{T}_{g(x)} (\mathcal{D}\mathcal{N}_z) = \mathbb{S}^n.
    \end{equation}
    Given an IED of $G(z)$, we explicitly characterize the tangent space $\mathcal{T}_{g(x)} (\mathcal{D}\mathcal{N}_z)$.
    \begin{align}\label{eqn:cal-TDN}
        \mathcal{T}_{g(x)} (\mathcal{D}\mathcal{N}_z)
        &= dE_{(g(x),0)}(\Tcal_{(g(x),0)} \Ecal_z) \\ \nonumber
        &= dE_{(g(x),0)}(\mathcal{T}_{g(x)} \mathcal{N}_{z} \oplus \mathcal{N}_{g(x)} (\mathcal{M}_{p,q}-y)) \\ \nonumber
        &= \mathcal{T}_{g(x)} \mathcal{N}_{z} + \mathcal{N}_{g(x)} (\mathcal{M}_{p,q}-y) \\ \nonumber
        &= \bigl\{ P B P^\top \mid B \in \mathbb{S}^n,\ B_{\beta \beta} = 0,\ B_{\beta \gamma} = 0,\ B_{\gamma\gamma} = 0 \bigr\} \\ \nonumber
        &\quad + \bigl\{ P B P^\top \mid B \in \mathbb{S}^n,\ B_{ij} = 0 \text{ for all } (i,j) \notin \beta \times \beta \bigr\} \\ \nonumber
        &= \bigl\{ P B P^\top \mid B \in \mathbb{S}^n,\ B_{\beta \gamma} = 0,\ B_{\gamma\gamma} = 0 \bigr\},
    \end{align}
    where the first equality follows from that $E$ is a diffeomorphism near the zero section of $\Ecal_z$; the second and third follow from calculation of tangent space of a vector bundle at the zero section and linearization of addition map $E$ and the last two equalities are just definitions. Substituting this resultant tangent space back into~\eqref{eqn:transverse_of_gz}, we observe that condition (a) is equivalent to
    \begin{equation*}
        dg_{x}(\mathbb{X}) + \bigl\{ P B P^\top \mid B \in \mathbb{S}^n,\ B_{\beta \gamma} = 0,\ B_{\gamma\gamma} = 0 \bigr\} = \mathbb{S}^n,
    \end{equation*}
    which is the W-SRCQ condition specified in Definition~\ref{def:W-SRCQ-nonKKT}.
    \ep
\end{proof}

We proceed to prove Theorem~\ref{thm:stability-on-stratum}, the local stability of the W-SRCQ.

\begin{theoremrep}{thm:stability-on-stratum}
    If the W-SRCQ holds at $z = (x, y) \in \widetilde{\mathcal{M}}_{p,q}$, then there exists a neighborhood $\widetilde{\mathcal{U}} \subseteq \widetilde{\mathcal{M}}_{p,q}$ of $z$ such that the W-SRCQ holds at every point in $\widetilde{\mathcal{U}}$.
\end{theoremrep}

\begin{proof}

This is a direct consequence of Proposition~\ref{prop:open-W-SRCQ}, which establishes stability under perturbations that preserve the negative inertia index of $G(z)$. Let $\mathcal U$ be the neighborhood supplied by that proposition. Every point of $\widetilde{\mathcal M}_{p,q}$ has negative inertia index $q$, so the W-SRCQ holds throughout the relative neighborhood $\widetilde{\mathcal U}:=\mathcal U\cap\widetilde{\mathcal M}_{p,q}$ of $z$.

\ep
\end{proof}

\section{Proofs of convergence results}\label{appendix:alg-global}

In this appendix, we provide detailed proofs for the directional derivative formulas, descent properties, and the global and local convergence theorems presented in Section~\ref{section:global}.

We first present the proof of Proposition \ref{prop:dir-der-of-phi}.

\begin{propositionrep}{prop:dir-der-of-phi}
    For an arbitrary $z \in \mathbb{X} \times \mathbb{S}^n$ with $G(z)$ admitting an IED $(\alpha,\beta,\gamma,p,q,P,\lambda)$ and any direction $v = (v_x,v_y) \in \mathbb{X} \times \mathbb{S}^n$, the directional derivative of $\varphi$ at $z$ along $v$ evaluates to
    \begin{equation}\nonumber
        \varphi'(z; v) = \left\langle dF_z^* \bigl(F(z)\bigr), \begin{bmatrix} v_x \\ H_1 \end{bmatrix} \right\rangle + \left\langle \nabla g(x)^* \bigl(F_1(z)\bigr), \Pi_{\mathbb{S}^n_-}(H_2) \right\rangle + \left\langle \nabla g(x)^* \bigl(F_1(z)\bigr) + F_2(z), \Pi_{\mathbb{S}^n_+}(H_2) \right\rangle,
    \end{equation}
    where $F_1(z) := \nabla_x L(x,y) = \nabla f(x) + \nabla g(x)y$, $F_2(z) := -g(x)+\Pi_{\Sbb^n_+}(G(z)) $, and  $H_1$ and $H_2$ are defined in Lemma~\ref{lemma:decomp-of-XxY}.
\end{propositionrep}
\begin{proof}
    Applying the chain rule for B-differentiable functions, we obtain
    \begin{equation*}
        \varphi'(z; v) = \langle F(z), F'(z; v) \rangle = \left\langle F(z), \begin{bmatrix}
            \bigl(\nabla^2_{xx} L(z) - \nabla g(x) \nabla g(x)^* \bigr)v_x + \nabla g(x) H \\
            - \nabla g(x)^*v_x + \Pi_{\mathbb{S}^n_+}'\bigl(G(z); H\bigr)
        \end{bmatrix} \right\rangle.
    \end{equation*}

    It follows from \eqref{eq:Pi-dir} that
    \begin{align*}
        \Pi_{\mathbb{S}^n_+}'\bigl(G(z); H\bigr)
        &= P \begin{bmatrix}
            \widetilde{H}_{\alpha\alpha} & \widetilde{H}_{\alpha\beta} & \Xi_{\alpha\gamma} \circ \widetilde{H}_{\alpha\gamma} \\
            \widetilde{H}_{\alpha\beta}^\top & 0 & 0 \\
            \Xi_{\alpha\gamma}^\top \circ \widetilde{H}_{\alpha\gamma}^\top & 0 & 0
        \end{bmatrix} P^\top
        + P \begin{bmatrix}
            0 & 0 & 0 \\
            0 & \Pi_{\mathbb{S}^{|\beta|}_+}(\widetilde{H}_{\beta\beta}) & 0 \\
            0 & 0 & 0
        \end{bmatrix} P^\top \\
        &= \Pi_{\mathbb{S}^n_+}'\bigl(G(z); H_1\bigr) + \Pi_{\mathbb{S}^n_+}'\bigl(G(z); H_2\bigr),
    \end{align*}
    where
    $H_1 = P \begin{bmatrix}
        \widetilde{H}_{\alpha\alpha} & \widetilde{H}_{\alpha\beta} & \widetilde{H}_{\alpha\gamma} \\
        \widetilde{H}_{\beta\alpha} & 0 & \widetilde{H}_{\beta\gamma} \\
        \widetilde{H}_{\gamma\alpha} & \widetilde{H}_{\gamma\beta} & \widetilde{H}_{\gamma\gamma}
    \end{bmatrix} P^\top \in \mathcal{T}_{G(z)} \mathcal{M}_{p,q}$
    and
    $H_2 = P \begin{bmatrix}
        0 & 0 & 0 \\
        0 & \widetilde{H}_{\beta\beta} & 0 \\
        0 & 0 & 0
    \end{bmatrix} P^\top \in \bigl(\mathcal{T}_{G(z)} \mathcal{M}_{p,q}\bigr)^\perp$, ensuring $H = H_1 + H_2$.
    Substituting this back into the inner product yields
    \begin{align*}
        \varphi'(z; v)
        &= \left\langle F(z), \begin{bmatrix}
            \bigl(\nabla^2_{xx} L(z) - \nabla g(x) \nabla g(x)^* \bigr)v_x + \nabla g(x) H_1 \\
            - \nabla g(x)^*v_x + \Pi_{\mathbb{S}^n_+}'\bigl(G(z); H_1\bigr)
        \end{bmatrix} \right\rangle + \left\langle F(z), \begin{bmatrix}
            \nabla g(x) H_2 \\
            \Pi_{\mathbb{S}^n_+}'\bigl(G(z); H_2\bigr)
        \end{bmatrix} \right\rangle \\
        &= \left\langle F(z), \begin{bmatrix}
            \nabla^2_{xx} L(z) - \nabla g(x) \nabla g(x)^* & \nabla g(x) \\
            -\nabla g(x)^* & \xi_{G(z)}
        \end{bmatrix} \begin{bmatrix} v_x \\ H_1 \end{bmatrix} \right\rangle + \left\langle F(z), \begin{bmatrix}
            \nabla g(x) H_2 \\
            \Pi_{\mathbb{S}^n_+}(H_2)
        \end{bmatrix} \right\rangle \\
        &= \left\langle F(z), dF_z \begin{bmatrix} v_x \\ H_1 \end{bmatrix} \right\rangle + \left\langle \begin{bmatrix} F_1(z) \\ F_2(z) \end{bmatrix}, \begin{bmatrix} \nabla g(x) H_2 \\ \Pi_{\mathbb{S}^n_+}(H_2) \end{bmatrix} \right\rangle \\
        &= \left\langle dF_z^* \bigl(F(z)\bigr), \begin{bmatrix} v_x \\ H_1 \end{bmatrix} \right\rangle + \left\langle \nabla g(x)^* \bigl(F_1(z)\bigr), H_2 \right\rangle + \left\langle F_2(z), \Pi_{\mathbb{S}^n_+}(H_2) \right\rangle \\
        &= \left\langle dF_z^* \bigl(F(z)\bigr), \begin{bmatrix} v_x \\ H_1 \end{bmatrix} \right\rangle + \left\langle \nabla g(x)^* \bigl(F_1(z)\bigr), \Pi_{\mathbb{S}^n_-}(H_2) \right\rangle + \left\langle \nabla g(x)^* \bigl(F_1(z)\bigr) + F_2(z), \Pi_{\mathbb{S}^n_+}(H_2) \right\rangle,
    \end{align*}
    where we utilized the orthogonal decomposition $H_2 = \Pi_{\mathbb{S}^n_+}(H_2) + \Pi_{\mathbb{S}^n_-}(H_2)$.
    \ep
\end{proof}

The proof of Lemma~\ref{lemma:normal-control}, which explicitly computes the exact line search along normal directions, is provided below.

\begin{lemmarep}{lemma:normal-control}
    Let $z = (x,y) \in \mathbb{X} \times \mathbb{S}^n$. If $W_1(z) \neq 0$, then $\|\nabla g(x)W_1(z)\| \neq 0$, and the exact line search minimum satisfies
    \begin{equation*}
        \mathop{\mathrm{argmin}}_{t \ge 0} \ \varphi \left(z + \begin{bmatrix} 0 \\ tW_1(z) \end{bmatrix} \right) = \frac{\| W_1(z)\|^2}{\| \nabla g(x)W_1(z)\|^2},
    \end{equation*}
    achieving a guaranteed reduction:
    \begin{equation*}
        \varphi(z) - \min_{t \ge 0 } \varphi \left(z + \begin{bmatrix} 0 \\ tW_1(z) \end{bmatrix} \right) = \frac{1}{2} \frac{\| W_1(z)\|^4}{\| \nabla g(x)W_1(z)\|^2}.
    \end{equation*}
    Analogously, if $W_2(z) \neq 0$, the optimal stepsize and the corresponding reduction are
    \begin{equation*}
        \mathop{\mathrm{argmin}}_{t \ge 0} \ \varphi \left(z + \begin{bmatrix} 0 \\ tW_2(z) \end{bmatrix} \right) = \frac{\| W_2(z)\|^2}{\| W_2(z)\|^2 + \| \nabla g(x)W_2(z)\|^2}
    \end{equation*}
    and
    \begin{equation*}
        \varphi(z) - \min_{t \ge 0} \varphi \left(z + \begin{bmatrix} 0 \\ tW_2(z) \end{bmatrix} \right) = \frac{1}{2} \frac{\| W_2(z)\|^4}{\| W_2(z)\|^2 + \| \nabla g(x)W_2(z)\|^2}.
    \end{equation*}
\end{lemmarep}
\begin{proof}
    Note that $W_1(z)$ and $W_2(z)$ are both in ${\Ncal}_{G(z)}\Mpq$. So we can separate the directional derivative as what we did in the calculation of $\varphi'(z,v)$.

    Since $W_1(z)\in \mathcal{N}_{G(z)}\mathcal{M}_{p,q}$, we have
    \begin{align*}
        2 \varphi \left(z + t\begin{bmatrix} 0 \\ W_1(z) \end{bmatrix}\right)
        &= \left\| F \left( z + t\begin{bmatrix} 0 \\ W_1(z) \end{bmatrix} \right) \right\|^2 \\
        &= \left\| \begin{bmatrix}
            \nabla f(x) + \nabla g(x)\bigl(y + t W_1(z)\bigr) \\
            -g(x) + \Pi_{\mathbb{S}^n_+}\bigl(G(z) + t W_1(z)\bigr)
        \end{bmatrix} \right\|^2 \\
        &= \left\| F(z) + t \begin{bmatrix}
            \nabla g(x)W_1(z) \\
            \Pi_{\mathbb{S}^n_+}'\bigl(G(z); W_1(z)\bigr)
        \end{bmatrix} \right\|^2 \\
        &= \| F(z) \|^2 + 2 t \left\langle F(z) , \begin{bmatrix} \nabla g(x) W_1(z) \\ \Pi_{\mathbb{S}^n_+}'\bigl(G(z); W_1(z)\bigr) \end{bmatrix} \right\rangle + t^2 \left\| \begin{bmatrix} \nabla g(x) W_1(z) \\ \Pi_{\mathbb{S}^n_+}'\bigl(G(z); W_1(z)\bigr) \end{bmatrix} \right\|^2.
    \end{align*}
    Since $W_1(z) \in \mathbb{S}^n_-$, we obtain that $\Pi_{\mathbb{S}^n_+}'\bigl(G(z); W_1(z)\bigr) = 0$. Consequently, the coefficients of the quadratic polynomial evaluate to
    \begin{equation*}
        \left\langle F(z) , \begin{bmatrix} \nabla g(x) W_1(z) \\ 0 \end{bmatrix} \right\rangle = -\| W_1(z) \|^2 \quad \text{and} \quad \left\| \begin{bmatrix} \nabla g(x) W_1(z) \\ 0 \end{bmatrix} \right\|^2 = \| \nabla g(x)W_1(z)\|^2.
    \end{equation*}
    The desired explicit optimum formulas then follow directly, provided that the leading coefficient $\|\nabla g(x)W_1(z)\|^2 \neq 0$.

    Next, we shall establish this non-vanishing property by contradiction. Suppose, to the contrary, that
    \begin{equation*}
        \nabla g(x)W_1(z) = 0.
    \end{equation*}
    By the definition of $W_1(z)$ from~\eqref{eq:def-W_1} and taking the inner product with $F_1(z)$ yields:
    \begin{align*}
        0 &= \left\langle F_1(z), \nabla g(x) \Pi_{\mathbb{S}^n_-}\Bigl(-\pi_{z,\mathbb{S}^n}^2 \bigl(\nabla g(x)^* F_1(z)\bigr)\Bigr) \right\rangle \\
          &= \left\langle \nabla g(x)^* F_1(z), \Pi_{\mathbb{S}^n_-}\Bigl(-\pi_{z,\mathbb{S}^n}^2 \bigl(\nabla g(x)^* F_1(z)\bigr)\Bigr) \right\rangle \\
          &= \left\langle -(\pi^1_{z,\mathbb{S}^n} + \pi_{z,\mathbb{S}^n}^2)\bigl(\nabla g(x)^* F_1(z)\bigr), \Pi_{\mathbb{S}^n_-}\Bigl(-\pi_{z,\mathbb{S}^n}^2 \bigl(\nabla g(x)^* F_1(z)\bigr)\Bigr) \right\rangle \\
          &= \left\langle -\pi_{z,\mathbb{S}^n}^2 \bigl(\nabla g(x)^* F_1(z)\bigr), \Pi_{\mathbb{S}^n_-}\Bigl(-\pi_{z,\mathbb{S}^n}^2 \bigl(\nabla g(x)^* F_1(z)\bigr)\Bigr) \right\rangle \quad \text{(since } \operatorname{Im}(\pi^1_{z,\mathbb{S}^n}) \perp \operatorname{Im}(\pi^2_{z,\mathbb{S}^n})) \\
          &= \left\langle (\Pi_{\mathbb{S}^n_+} + \Pi_{\mathbb{S}^n_-})\Bigl(-\pi_{z,\mathbb{S}^n}^2 \bigl(\nabla g(x)^* F_1(z)\bigr)\Bigr), \Pi_{\mathbb{S}^n_-}\Bigl(-\pi_{z,\mathbb{S}^n}^2 \bigl(\nabla g(x)^* F_1(z)\bigr)\Bigr) \right\rangle \\
          &= \left\| \Pi_{\mathbb{S}^n_-}\Bigl(-\pi_{z,\mathbb{S}^n}^2 \bigl(\nabla g(x)^* F_1(z)\bigr)\Bigr) \right\|^2 \quad \text{(since } \mathbb{S}^n_+ \perp \mathbb{S}^n_-) \\
          &= \| W_1(z)\|^2.
    \end{align*}
    This implies $W_1(z) = 0$, contradicting the premise $W_1(z) \neq 0$. A similar and easier argument validates the stepsize and reduction along $W_2(z)$, by noting that the corresponding denominator $\|W_2(z)\|^2 + \|\nabla g(x)W_2(z)\|^2\neq 0$.
    \ep
\end{proof}

We proceed to prove Theorem~\ref{thm:global-convergence}, which establishes the global convergence of Algorithm~\ref{alg:SGN}. Recall the definition of the lower modulus of nonzero eigenvalue $\delta(z)$, as introduced in \eqref{eqn: minieigen}:
$$
\delta(z) := \min \{ |\lambda(G(z))| : \lambda(G(z)) \ne 0 \}.
$$

\begin{theoremrep}{thm:global-convergence}
    Let $\epsilon = 0$ and let $z^0$ be an arbitrary initial point. Then Algorithm~\ref{alg:SGN} either terminates finitely at a D-stationary point of $\varphi$ or generates an infinite sequence whose every accumulation point $z^\infty$ with $\delta(z^\infty)>\delta$ is a D-stationary point of $\varphi$.
\end{theoremrep}
We only need to prove this theorem under the assumption that $\mu(z) = \| F(z) \|^2$ is bounded away from 0, i.e. $\mu(z) \geq \mubar$ for some fixed $\mubar > 0$. Otherwise, if $\mu(z^{k_j}) \to 0$ for some subsequence $z^{k_j} \to z^\infty$, we will have $F(z^\infty) = 0$. And by Proposition \ref{prop:dir-der-of-phi}, $\varphi'(z^\infty,v) = 0$ for any $v \in \Xbb \times \Sbb^n$, which means $z^\infty$ is a D-stationary point.  So in the following lemmas and propositions, we will assume $\mu(z) \geq \mubar > 0$.

To establish this theorem, several auxiliary lemmas are required to secure a uniform lower bound on the Armijo line-search stepsize. Let $\widetilde{\Omega} \subseteq \widetilde{\mathcal{M}}$ be an arbitrary compact domain within a stratum. We define $\Omega$ as the closed disk bundle of radius $l$ over $\widetilde{\Omega}$:
\begin{equation*}
    \Omega := \bigl\{ (z,v) \mid z \in \widetilde{\Omega}, \, v \in \mathcal{T}_z \widetilde{\mathcal{M}}, \, \| v \| \le l \bigr\}.
\end{equation*}
By adjusting the radius $l > 0$, we ensure the exponential map is well-defined on $\Omega$ and satisfies $\exp_z(v) \in \widetilde{\mathcal{M}}$.

The following lemma establishes that the directional derivative of $\varphi$, evaluated along either the stratum-LM direction or the two normal directions, is of the order of the squared norm of the direction.

\begin{applemma}\label{lemma: phi~d2}
    Suppose there exists a compact domain $\widetilde{\Omega} \subseteq \widetilde{\mathcal{M}}$ and a fixed constant $\bar{\mu} > 0$ such that $\mu(z) \ge \bar{\mu}$ for all $z \in \widetilde{\Omega}$. Then, there exist $M > m > 0$ only depending on $F$ and $\widetilde{\Omega}$ such that for any $z \in \widetilde{\Omega}$ and any corresponding direction $d \in \{ W_1(z), W_2(z), v^{\mathrm{LM}}(z) \}$, the directional derivative satisfies
    \begin{equation*}
        m \| d\|^2 \le -\varphi'(z; d) \le M \| d\|^2.
    \end{equation*}
\end{applemma}
\begin{proof}

    If $d = W_2(z) \in (\mathcal{T}_z\widetilde{\mathcal{M}})^\perp$, it is clear that
    \begin{align*}
        \varphi'(z; d)
        &= \left\langle \nabla g(x)^* F_1(z), \Pi_{\mathbb{S}^n_-}(W_2(z)) \right\rangle + \left\langle \nabla g(x)^* F_1(z) + F_2(z), \Pi_{\mathbb{S}^n_+}(W_2(z)) \right\rangle \\
        &= \left\langle \nabla g(x)^* F_1(z) + F_2(z), W_2(z) \right\rangle \\
        &= \left\langle \nabla g(x)^* F_1(z) + F_2(z), \Pi_{\mathbb{S}^n_+}\bigl(-(\nabla g(x)^* F_1(z) + F_2(z))\bigr) \right\rangle \\
        &= -\| W_2(z) \|^2 = -\| d \|^2.
    \end{align*}
    Similarly, for $d = W_1(z)$, we also have $\varphi'(z; d) = -\| d \|^2$.

    If $d = v^{\mathrm{LM}}(z) \in \mathcal{T}_z\widetilde{\mathcal{M}}$, the stratum-LM construction implies
    \begin{align*}
        \varphi'(z; d) &= -F(z)^\top dF_z \bigl(\mu(z) I + dF_z^* dF_z\bigr)^{-1} dF_z^* F(z) \\
        &= -\bigl(v^{\mathrm{LM}}(z)\bigr)^\top \bigl(\mu(z) I + dF_z^* dF_z\bigr) v^{\mathrm{LM}}(z).
    \end{align*}
    Since $F$ is smooth on the compact domain $\tOmega \subseteq \tMcal$,  the spectral norm of $dF_z$ is uniformly bounded  by some constant $L_F$, we have
    \begin{equation*}
        \bar{\mu} \|d\|^2 \le \mu(z) \|d\|^2 \le -\varphi'(z; d) \le \bigl(\mu_{\max} + L_F^2\bigr) \|d\|^2.
    \end{equation*}
    Selecting $m = \min\{1, \bar{\mu}\}$ and $M = \max\{1, \mu_{\max} + L_F^2\}$ uniformly satisfies the conditions for all three direction types.
    \ep
\end{proof}

 {The following lemma gives uniform approximation estimates. Here $LC^2$ means $C^2$ with locally Lipschitz second derivatives, and $\exp$ uses the pullback product metric induced by $\Psi(x,y)=(x,g(x)+y)$.}

\begin{applemma}\label{lemma:retr~exp}
    For an arbitrary compact domain $\widetilde{\Omega} \subseteq \widetilde{\mathcal{M}}$, there exists $l_{\widetilde{\Omega}} > 0$ such that for $\Omega = \{ (z,v) \mid z \in \widetilde{\Omega}, v\in \mathcal{T}_z \widetilde{\mathcal{M}}, \| v \| \le l_{\widetilde{\Omega}} \}$, we have
    \begin{itemize}
        \item[(a)] $R_{\widetilde{\mathcal{M}}}(\Omega) \subseteq \widetilde{\mathcal{M}}$ where $R_{\tMcal}$ is the retraction defined in (\ref{retraction_tM}) and $\exp(\Omega) \subseteq \widetilde{\mathcal{M}}$;
        \item[(b)] there exists a locally continuous coefficient $\kappa_z$ satisfying $\|R_z(v) - (z+v)\| \le \kappa_z \| v \|^2$ for all $(z,v) \in \Omega$.
    \end{itemize}
     {Moreover, there exist constants $C'_{\Omega},\kappa_{\widetilde{\Omega}},L_{\varphi,\Omega}>0$, independent of $z$, such that (c), (d), and (f) below hold for all $(z,v)\in\Omega$. The Taylor estimates in (e) hold under the assumptions specified there.}
    \begin{itemize}
        \item[(c)] $\| \exp_z(v) - (z + v) \| \le C'_{\Omega} \| v \|^2$.
        \item[(d)] $\|R_z(v) - (z+v) \| \le \kappa_{\widetilde{\Omega}} \| v \|^2$.
        \item[(e)]  {For $T=R,\exp$, the standing $C^2$ assumptions give $\varphi(T_z(v))=\varphi(z)+\varphi'(z;v)+o(\|v\|)$ uniformly in $z\in\widetilde\Omega$. If $f,g\in LC^2$ near the projection of $\widetilde\Omega$ onto $\mathbb X$, then, after reducing $l_{\widetilde\Omega}$ if necessary, there exist $C_{F,\Omega},C_{\varphi,\Omega}>0$ such that
        \[
        \begin{aligned}
            \|F(T_z(v))-F(z)-dF_z(v)\|&\le C_{F,\Omega}\|v\|^2,\\
            |\varphi(T_z(v))-\varphi(z)-\varphi'(z;v)|&\le C_{\varphi,\Omega}\|v\|^2
        \end{aligned}
        \]
        for all $(z,v)\in\Omega$.}
        \item[(f)] $|\varphi(R_z(v)) - \varphi(\exp_z(v))| \le L_{\varphi,\Omega} \| R_z(v) - \exp_z(v) \|$.
    \end{itemize}
\end{applemma}
\begin{proof}
    \begingroup
     Through $\Psi$, the maps $R$ and $\exp$ are lifts of the smooth retraction and exponential maps on $\mathbb X\times\mathcal M$. For fixed $z$, their second derivatives in $v$ are continuous in $(z,v)$. Compactness therefore gives a common domain and the uniform quadratic geometric bounds in (a)--(d), with $\kappa_z$ chosen constant on $\widetilde\Omega$.
    Since $F$ is $C^1$ on the stratum, the first order Taylor formula gives the uniform little-$o$ estimate in (e). Under $LC^2$, $dF$ is locally Lipschitz in stratum coordinates, giving the quadratic estimates for $F$ and $\varphi=\|F\|^2/2$. Finally, local Lipschitz continuity of $\varphi$ in the ambient space gives (f).
    \par\endgroup\ep
\end{proof}

 {The first order estimate in Lemma~\ref{lemma:retr~exp}(e) ensures that the tangential Armijo line search in Algorithm~\ref{alg:SLMN} is well defined under the standing $C^2$ assumptions.}

\begingroup
\begin{appprop}\label{prop:armijo}
     Let $\widetilde{\Omega}\subseteq\widetilde{\mathcal M}$ be a compact subset of a fixed stratum such that $\mu(z)\ge\bar\mu>0$ for every $z\in\widetilde{\Omega}$. Let $\Omega$ and $l_{\widetilde{\Omega}}$ be as in Lemma~\ref{lemma:retr~exp}, and let $\eta\in(0,1/2)$ be the parameter in Algorithm~\ref{alg:SLMN}. For each $z\in\widetilde{\Omega}$ and $v=v^{\mathrm{LM}}(z)\ne0$, there exists $\varepsilon_z\in(0,1]$ such that $(z,tv)\in\Omega$ and
    \begin{equation}\label{def:epsilon_z}
        \varphi(z)-\varphi\bigl(R_z(tv)\bigr)\ge-\eta t\varphi'(z;v),
        \qquad \forall\,t\in[0,\varepsilon_z).
    \end{equation}
    If $v=0$, one may take $\varepsilon_z=1$, and the inequality holds with equality.

    Moreover, suppose $z^k\to z^\infty$ in $\widetilde{\Omega}$ and put $v^k=v^{\mathrm{LM}}(z^k)$. If $\inf_k\|v^k\|>0$, then there exists $\overline\varepsilon\in(0,1]$ such that $(z^k,tv^k)\in\Omega$ and the Armijo inequality~\eqref{def:epsilon_z} holds at every $z^k$ for every $t\in[0,\overline\varepsilon)$.
\end{appprop}
\begin{proof}
     Lemma~\ref{lemma:retr~exp}(e) gives, uniformly for $z\in\widetilde\Omega$,
    \[
        \varphi(R_z(tv))=\varphi(z)+t\varphi'(z;v)+o(t\|v\|).
    \]
    Since $-\varphi'(z;v)\ge m\|v\|^2$ by the LM part of Lemma~\ref{lemma: phi~d2}, each fixed $v\ne0$ admits $\varepsilon_z\in(0,1]$ with $\varepsilon_z\|v\|\le l_{\widetilde\Omega}$ for which the remainder is bounded by $(1-\eta)(-t\varphi'(z;v))$. This proves~\eqref{def:epsilon_z}; the zero-direction case follows from $R_z(0)=z$.
    On $\widetilde\Omega$, boundedness of $F,dF$ and $\mu\ge\bar\mu>0$ imply that the LM directions are bounded. If their norms also have a positive lower bound, the same uniform remainder estimate gives a common $\overline\varepsilon\in(0,1]$, reduced if necessary so that all trial vectors remain in $\Omega$.
    \ep
\end{proof}
\endgroup
\begin{appprop}\label{prop:armijo-bound}
    Under the assumption of Proposition~\ref{prop:armijo}, the final computed stepsize $\rho^j$ within the Armijo backtracking loop in Algorithm \ref{alg:SLMN} satisfies $\rho^j \ge \rho \varepsilon_z$. We obtain that
    \begin{equation*}
        \varphi(z) - \varphi\bigl(R_z(\rho^j v)\bigr) \ge \eta \rho \varepsilon_z \bar{\mu} \| v \|^2.
    \end{equation*}
\end{appprop}
\begin{proof}
    If $j=0$, the stepsize bound follows from $\varepsilon_z\le1$ and $\rho<1$. If $j\ge1$, the preceding trial step was rejected, so $\rho^{j-1}\ge\varepsilon_z$, i.e. $\rho^j\ge\rho\varepsilon_z$. Then, by Proposition \ref{prop:armijo}, Lemma \ref{lemma: phi~d2}, and the above estimate of the stepsize, we obtain the desired result.
    \ep
\end{proof}

 {The following lemma gives a common admissible Armijo step range near a point with a nonzero LM direction.}
\begin{applemma}\label{lemma: epsilon-cont-on-M}
     {Suppose $z^k\to z^\infty$ in a fixed stratum $\widetilde{\mathcal M}$ and $v^{\mathrm{LM}}(z^\infty)\ne0$. Under the standing bound $\mu\ge\bar\mu>0$, there exists $\overline\varepsilon\in(0,1]$ such that, for all sufficiently large $k$, the retraction $R_{z^k}(t v^{\mathrm{LM}}(z^k))$ is defined and satisfies the Armijo inequality~\eqref{def:epsilon_z} for every $t\in[0,\overline\varepsilon)$.}
\end{applemma}
\begin{proof}
     {Choose a sufficiently small compact neighborhood of $z^\infty$ in $\widetilde{\mathcal M}$. The LM direction is continuous on the fixed stratum under the standing bound on $\mu$, so its norm is bounded away from zero near $z^\infty$. Proposition~\ref{prop:armijo} applied to the tail of the sequence proves the claim.}
    \ep
\end{proof}
\begin{applemma}\label{lemma: W-smooth-on-M}
    Restricted to a fixed stratum $\widetilde{\mathcal{M}}$, the geometric normal mappings
    \begin{equation*}
        W_1(z) = \Pi_{\mathbb{S}^n_-}\Bigl(-\pi_{z,\mathbb{S}^n}^2 \bigl(\nabla g(x)^* F_1(z)\bigr)\Bigr) \quad \text{and} \quad W_2(z) = \Pi_{\mathbb{S}^n_+}\Bigl(-\pi_{z,\mathbb{S}^n}^2 \bigl(\nabla g(x)^* F_1(z) + F_2(z)\bigr)\Bigr)
    \end{equation*}
    {are locally Lipschitz functions from $\widetilde{\mathcal{M}}$ into $\mathbb{S}^n$.}
\end{applemma}
\begin{proof}
     {On a fixed stratum, the normal projection~\eqref{eqn:pi^2} is locally $C^2$. The conclusion follows from $f,g\in C^2$ and nonexpansiveness of the positive and negative semidefinite cone projections.}
    \ep
\end{proof}

With the necessary preliminary results established, we proceed to the proof of the global convergence theorem.

\begin{proof}[of Theorem~\ref{thm:global-convergence}]
If the algorithm terminates in finitely many iterations at some $z^k$, then the stopping test enforces $s(z^k)=0$. By the definition of the stationarity measure $s(\cdot)$, this implies that $z^k$ is a $D$-stationary point of $\varphi$.

Suppose that the algorithm generates an infinite sequence $\{z^k\}$. By the acceptance rule, $\{\varphi(z^k)\}$ is nonincreasing and bounded below by $0$, and hence convergent. In particular,
\begin{equation}\label{finitesum}
\sum_{k=0}^{\infty}\big(\varphi(z^k)-\varphi(z^{k+1})\big)\le \varphi(z^0)<\infty.
\end{equation}

Let $\{z^k\}_{k\in\Theta}$ be a convergent subsequence with $z^k\to z^\infty$ as $k\in\Theta\to\infty$, where $\delta(z^\infty)>\delta$. Let $\widetilde{\Mcal}$ be the stratum containing $z^\infty$, and let $\widetilde{\Omega}$ be a compact set such that $z^k\in\widetilde{\Omega}$ for all $k\in\Theta$ sufficiently large. By Lemma~\ref{lemma:corr-almost-proj}, the corrected points $\widehat z^k$ also lie in $\widetilde{\Omega}$ for all sufficiently large $k\in\Theta$ and satisfy $\widehat z^k\to z^\infty$.

If $\mu_k\to 0$ along $\Theta$, then $F(z^k)\to 0$ along $\Theta$, and by continuity of $F$ we obtain $F(z^\infty)=0$, i.e., $z^\infty$ is a KKT pair. We therefore assume in the remainder of the proof that there exists $\bar\mu>0$ such that $\mu_k\ge \bar\mu$ for all $k\in\Theta$ sufficiently large.

Fix such a large $k\in\Theta$. Since $z^{k+1}$ is accepted, we have $\varphi(z^{k+1})\le \varphi\big(\mathrm{SLMN}(\widehat z^k)\big)$,
and hence
\begin{align*}
\varphi(z^k)-\varphi(z^{k+1})
&\ge \varphi(z^k)-\varphi\big(\mathrm{SLMN}(\widehat z^k)\big)\\
&\ge \varphi(\widehat z^k)-\varphi\big(\mathrm{SLMN}(\widehat z^k)\big)
      -\big|\varphi(z^k)-\varphi(\widehat z^k)\big|.
\end{align*}
 By the definition of $\mathrm{SLMN}$, we obtain that
\begin{align*}
\varphi(\widehat z^k)-\varphi\big(\mathrm{SLMN}(\widehat z^k)\big)
\ge \max\Big\{&
\varphi(\widehat z^k)-\varphi\big(R_{\widehat z^k}({\rho^{j_k}}\widehat v^k)\big),\
\varphi(\widehat z^k)-\min_{t\ge 0}\varphi(\widehat z^k+t\widehat W_1^k),\\
&\varphi(\widehat z^k)-\min_{t\ge 0}\varphi(\widehat z^k+t\widehat W_2^k)\Big\}.
\end{align*}
Combining Proposition~\ref{prop:armijo-bound} with Lemma~\ref{lemma:normal-control} yields
\[
\varphi(\widehat z^k)-\varphi\big(\mathrm{SLMN}(\widehat z^k)\big)
\ge \max\left\{
\eta\rho\,\varepsilon_{\widehat z^k}\,\bar\mu\,\|\widehat v^k\|^2,\
\frac{1}{2L_{g,\Omega}^2}\|\widehat W_1^k\|^2,\
\frac{1}{2(L_{g,\Omega}^2+1)}\|\widehat W_2^k\|^2
\right\}.
\]
 {Suppose, for a contradiction, that $s(z^\infty)>0$. By Proposition~\ref{prop:armijo-bound} and Lemma~\ref{lemma: epsilon-cont-on-M} when $v^{\mathrm{LM}}(z^\infty)\ne0$, or by Lemmas~\ref{lemma:normal-control} and~\ref{lemma: W-smooth-on-M} otherwise, together with continuity of the directions on the fixed stratum, there exists $\bar\varepsilon>0$ such that, for all $k\in\Theta$ sufficiently large,}
\[
\varphi(\widehat z^k)-\varphi\big(\mathrm{SLMN}(\widehat z^k)\big)
\ge \bar\varepsilon\max\big\{\|\widehat v^k\|^2,\ \|\widehat W_1^k\|^2,\ \|\widehat W_2^k\|^2\big\}
=\bar\varepsilon\, s(\widehat z^k)^2.
\]
Moreover, Lemma~\ref{lemma:corr-almost-proj} implies $\big|\varphi(z^k)-\varphi(\widehat z^k)\big|\to 0$ along $\Theta$, and thus, for all large $k\in\Theta$,
\[
\varphi(z^k)-\varphi(z^{k+1}) \ \ge\ \tfrac12\bar\varepsilon\, s(\widehat z^k)^2.
\]
Passing to the limit along $\Theta$ and using $\widehat z^k\to z^\infty$ gives $s(\widehat z^k)\to s(z^\infty)$, so if $s(z^\infty)>0$ then the right-hand side is bounded away from $0$ on an infinite subset of $\Theta$, contradicting the summability in~\eqref{finitesum}. Hence $s(z^\infty)=0$, and therefore $z^\infty$ is a $D$-stationary point of $\varphi$. \ep
\end{proof}

\section{Proof of local convergence result}\label{appendix:alg-local}

 We now establish the local {superlinear convergence, and quadratic convergence under the additional $LC^2$ assumption,} of Algorithm~\ref{alg:SGN}. Our analysis decomposes the iterate displacement into components along the tangent and normal spaces of the stratum $\mathcal{M}{p,q}$. Under the W-SOC and W-SRCQ, the tangential component exhibits a local {superlinear rate, and a quadratic rate under $LC^2$}; see Remark~\ref{remark:local-rate-LMGN}. To control the normal component, we {strengthen these assumptions to the SOSC and the SRCQ} and exploit the specific design of Algorithm~\ref{alg:SGN} to ensure that motion in the normal directions cannot induce an excessive decrease in the merit function. This mechanism guarantees finite identification of the stratum $\mathcal{M}{p,q}$. To apply these local estimates at a D-stationary accumulation point, we invoke Theorem~\ref{thm:dstat-implies-kkt} to obtain the KKT property under the SOSC and SRCQ. Theorem~\ref{thm:global-convergence} supplies D-stationarity when the correction threshold condition holds.

Next, we begin with some elementary computations. Given the specific stratum $\mathcal{M} = \mathcal{M}_{p,q}$ and its preimage $\widetilde{\mathcal{M}} = G^{-1}(\mathcal{M})$, we analyze a point $z \in \widetilde{\mathcal{M}}$ perturbed by a normal variation $\Delta \in \mathcal{N}_{G(z)}\mathcal{M} = (\mathcal{T}_{G(z)}\mathcal{M})^\perp$. Proposition~\ref{prop:TAM} characterizes this space as
\begin{equation}\label{eq:NAM}
    \mathcal{N}_{G(z)} \mathcal{M} = \bigl\{ P \widetilde{H} P^\top \mid \widetilde{H} \in \mathbb{S}^n,\ \widetilde{H}_{\alpha \alpha} = 0,\ \widetilde{H}_{\alpha \beta} = 0,\ \widetilde{H}_{\alpha \gamma} = 0,\ \widetilde{H}_{\beta \gamma} = 0,\ \widetilde{H}_{\gamma \gamma} = 0 \bigr\}.
\end{equation}
Applying this transverse perturbation $z \mapsto z + (0, \Delta)$ evaluates the shifted KKT mapping as
\begin{equation}\label{eq:F-normal}
    F\left(z + \begin{bmatrix}0 \\ \Delta\end{bmatrix}\right)
    = \begin{bmatrix}
        \nabla f(x) + \nabla g(x)(y + \Delta) \\
        -g(x) + \Pi_{\mathbb{S}^n_+} (g(x) + y + \Delta)
    \end{bmatrix}
    = F(z) + \begin{bmatrix} \nabla g(x)\Delta \\ \Pi_{\mathbb{S}^n_+}(\Delta) \end{bmatrix}.
\end{equation}
We define the associated incremental operator:
\begin{equation*}
    J_z(\Delta) := \begin{bmatrix} \nabla g(x)\Delta \\ \Pi_{\mathbb{S}^n_+}(\Delta) \end{bmatrix}.
\end{equation*}
Expanding the squared merit function along the normal ray $t \ge 0$ yields
\begin{equation}\label{eq:varphi-normal}
    \varphi\left(z + t\begin{bmatrix}0 \\ \Delta\end{bmatrix}\right) = \varphi(z) + t \left\langle F(z), J_z(\Delta) \right\rangle + \frac{t^2}{2}\|J_z(\Delta)\|^2 \ge \left[1 - \left\langle \frac{F(z)}{\|F(z)\|}, \frac{J_z(\Delta)}{\|J_z(\Delta)\|} \right\rangle^2 \right] \varphi(z),
\end{equation}
where the normalized inner product is strictly defined as $0$ if either $F(z) = 0$ (implying $\varphi(z) = 0$) or $J_z(\Delta) = 0$. With this definition, it is easy to verify that~\eqref{eq:varphi-normal} holds for any $z$.

The next two lemmas characterize the zero and the Lipschitz behavior of the incremental mapping $J_z(\Delta)$.

\begin{applemma}\label{lemma:assu-to-isolated}
     Let $\overline{z}$ be a KKT pair of~\eqref{prog:SDP} equipped with an IED $(\alpha,\beta,\gamma,p,q,\overline{P},\overline{\lambda})$ of $G(\overline{z})$. If both the W-SOC and the SRCQ hold at $\overline{z}$, then $0=\Delta\in\mathbb{S}^n$ is the unique root of the restricted mapping $\Delta \mapsto F\bigl(\overline{z} + (0, \Delta)\bigr)$ on $\mathcal{N}_{G(\overline{z})}\mathcal{M}_{p,q}$.
\end{applemma}

\begin{proof}
Recall from~\eqref{eq:F-normal} that $J_{\overline z}(\Delta)=0$ implies
\[
\nabla g(\overline x)\Delta=0
\qquad\text{and}\qquad
\Pi_{\mathbb{S}^n_+}(\Delta)=0,
\]
and hence $\Delta\in\mathbb{S}^n_-$. Since $\Delta\in \mathcal{N}_{G(\overline z)}\mathcal{M}_{p,q}$, we also have that, in the coordinates of $\overline P$, only the block $\Delta_{\beta\beta}$ may be nonzero.

Assume to the contrary that there exists $0\neq \Delta\in \mathcal{N}_{G(\overline z)}\mathcal{M}_{p,q}$ with $J_{\overline z}(\Delta)=0$. Define
\[
Y:=\overline P
\begin{bmatrix}
0 & 0 & 0\\
0 & -I_{|\beta|} & 0\\
0 & 0 & 0
\end{bmatrix}
\overline P^\top.
\]
By the SRCQ at $\overline z$, there exist $u\in\mathbb{X}$ and $Y_0\in \Bigl\{\overline P B\overline P^\top\in\mathbb{S}^n \,\Big|\,
B_{\beta\beta}\succeq 0,\ B_{\beta\gamma}=0,\ B_{\gamma\gamma}=0\Bigr\}$ such that
\[
\nabla g(\overline x)^*u+Y_0=Y.
\]
Taking inner products with $\Delta$ gives
\[
\langle Y,\Delta\rangle=\langle u,\nabla g(\overline x)\Delta\rangle+\langle Y_0,\Delta\rangle=\langle Y_0,\Delta\rangle,
\]
where we used $\nabla g(\overline x)\Delta=0$. Combining $((\overline{P})^\top Y_0 \overline{P})_{\beta\beta} \succeq 0$, $\Delta \in \mathcal{N}_{G(\overline{z})}\mathcal{M}_{p,q}$, and $\Delta \preceq 0$, we have $\langle Y_0, \Delta \rangle \le 0$, and therefore $\langle Y, \Delta \rangle \le 0$.

On the other hand, by the definitions of $Y$ and $\Delta$,
\[
\langle Y,\Delta\rangle
=\left\langle -I_{|\beta|},((\overline{P})^\top\Delta\overline{P})_{\beta\beta}\right\rangle
=-\operatorname{tr}(((\overline{P})^\top\Delta\overline{P})_{\beta\beta})>0,
\]
since $((\overline{P})^\top\Delta\overline{P})_{\beta\beta}\preceq 0$ and $((\overline{P})^\top\Delta\overline{P})_{\beta\beta}\neq 0$. This contradiction shows that such a nonzero $\Delta$ cannot exist. \ep
\end{proof}

\begin{applemma}\label{lemma:bound-J}
The incremental mapping $J_z$ is positively homogeneous. If the W-SOC and the SRCQ hold simultaneously at a KKT pair $\overline{z}$, there exist constants $0 < c_1 < c_2$ such that
    \begin{equation*}
        c_1\|\Delta\| \le \|J_{\overline{z}}(\Delta)\| \le c_2\|\Delta\| \quad \forall\, \Delta \in \mathcal{N}_{G(\overline{z})}\mathcal{M}_{p,q},
    \end{equation*}
    where $\mathcal{M}_{p,q}$ is given by $\overline z$.
\end{applemma}

\begin{proof}
Positive homogeneity follows directly from the definition of $J_z$ and the positive homogeneity of the projection. The existence of $c_2$ is a consequence of the Lipschitz continuity of $\Pi_{\mathbb{S}^n_+}$ together with the boundedness of $\nabla g(\overline x)$.

For the lower bound, suppose by contradiction that no such $c_1$ exists. Then there exists a sequence $\Delta^\nu\in \mathcal{N}_{G(\overline z)}\mathcal{M}_{p,q}$ with $\|\Delta^\nu\|=1$ and $\|J_{\overline z}(\Delta^\nu)\|\to 0$. By compactness of the unit sphere, along a subsequence $\Delta^\nu\to \Delta^\infty$ with $\|\Delta^\infty\|=1$. By continuity of $J_{\overline z}$, we obtain $J_{\overline z}(\Delta^\infty)=0$, contradicting Lemma~\ref{lemma:assu-to-isolated}. \ep
\end{proof}

The following lemma provides upper and lower bounds for the directional derivative $\varphi'(z, v)$ in terms of the LM direction, under the assumption that the W-SOC and the W-SRCQ hold simultaneously.

\begin{applemma}\label{Lemma:ww-dir-of-phi-order}
    Let $\overline{z}$ be a KKT pair of~\eqref{prog:SDP} with an IED $(\alpha,\beta,\gamma,p,q,\Pbar,\overline \lambda)$ of $G(\zbar)$. If the W-SOC and the W-SRCQ hold simultaneously at $\zbar$, then there exist constants $0 < m < M$ such that for any point $z \in \widetilde{\mathcal{M}}_{p,q}$ sufficiently close to $\overline{z}$ and directional derivative associated with $v = v^{\mathrm{LM}}(z)$ the following estimation holds
    \begin{equation*}
        m \| v \|^2 \le \bigl\| \varphi'(z; v) \bigr\| \le M \| v \|^2.
    \end{equation*}
\end{applemma}
\begin{proof}
By the definition of $v^{\mathrm{LM}}(z)$, we have
\[
\varphi'(z;v)=\langle F(z),dF_z(v)\rangle
=-\big\langle v,\bigl(\mu(z)I+dF_z^*dF_z\bigr)v\big\rangle,
\]
where $\mu(z)\ge 0$ is the regularization parameter. On a fixed stratum, $dF_z$ depends smoothly on $z$, hence $\|dF_z\|$ is locally bounded. Moreover, by Corollary~\ref{coro:WR-uniform-bounded}, the W-SOC and the W-SRCQ yield a uniform lower bound on $dF_z^*dF_z$ over $z$ in a neighborhood of $\overline z$ within the stratum. These two bounds imply that the eigenvalues of $\mu(z)I+dF_z^*dF_z$ are uniformly bounded above and below on that neighborhood. This completes the proof. \ep
\end{proof}

 Next, we present an estimation result establishing that, under the {SOSC} and SRCQ assumptions, the decrease of $\varphi$ along the normal direction is well-controlled.

\begin{appprop}\label{prop:esti-cor}
     \par\nobreak Let $\overline{z} = (\overline{x}, \overline{y})$ be a KKT pair of~\eqref{prog:SDP} with an IED $(\alpha, \beta, \gamma, p, q, \overline{P}, \overline{\lambda})$ of $G(\overline{z})$. If the {SOSC} and the SRCQ hold at $\overline{z}$, then there exists a constant $c_3 > 0$ such that for any $z \in \tMpq$ sufficiently close to $\overline{z}$, $\Delta \in \mN_{G(z)}\Mpq$ and $t\in\Rbb$,
    \begin{equation}\label{eq:esti-cor}
        \varphi(z) \leq c_3 \, \varphi\left(z + t\begin{bmatrix}0 \\ \Delta\end{bmatrix}\right).
    \end{equation}
\end{appprop}
\begin{proof}
     If we have the following estimate for the coefficient multiplying $\varphi(z)$ on the right-hand side of~\eqref{eq:varphi-normal},
    \begin{equation}\label{eq:esti-coef}
        \limsup_{ z\to \overline z,\  z\in \tMpq}\left[\sup_{\Delta\in \mN_{G( z)}\Mpq}\la \frac{F( z)}{\|F( z)\|},\frac{J_{ z}(\Delta)}{\|J_{ z}(\Delta)\|}\ra^{{2}}\right]< 1,
    \end{equation}
    then the desired result~\eqref{eq:esti-cor} follows directly.
    We will prove \eqref{eq:esti-coef} by contradiction. Recall that when $F(z) = 0$ or $J_z(\Delta) = 0$, $\la \frac{F( z)}{\|F( z)\|},\frac{J_{ z}(\Delta)}{\|J_{ z}(\Delta)\|}\ra$ is defined to be zero. Then, noting the positive homogeneity of $J_z$ by Lemma \ref{lemma:bound-J}, the negation of the claim is equivalent to saying there exist $z^\nu\in \tMpq$, $z^\nu\to \overline z$, $\Delta^\nu\in \mN_{G(z^\nu)}\Mpq$  and $\|\Delta^\nu\|=1$ that
    \begin{equation*}
        \lim_{\nu\to\infty}\la \frac{F( z^\nu)}{\|F( z^\nu)\|},\frac{J_{ z^\nu}(\Delta^\nu)}{\|J_{ z^\nu}(\Delta^\nu)\|}\ra^{{2}}=1,
    \end{equation*}
     {so, after passing to a subsequence, there exists $\iota\in\{-1,1\}$ such that}
    \begin{equation}\label{eq:F-J=0}
        \lim_{\nu\to\infty} \frac{F( z^\nu)}{\|F( z^\nu)\|}-{\iota}\frac{J_{ z^\nu}(\Delta^\nu)}{\|J_{ z^\nu}(\Delta^\nu)\|}=0.
    \end{equation}
    For the first part, by the smoothness of $F$ on $\tMpq$, we have that as $\nu\to\infty$,
    \begin{align}\label{eq:F/F-norm}
         \frac{F( z^\nu)}{\|F( z^\nu)\|}=\frac{F(\overline z)+ dF_{\overline z}(w^\nu) + o( \| z^\nu-\overline z\|)}{\|F( z^\nu)\|} = dF_{\overline z} (\frac{w^\nu}{\| w^\nu \|}) \frac{\| w^\nu \|}{\| F( z^\nu) \|}+\frac{o(\|z^\nu-\overline z\|)}{\|F( z^\nu)\|},
    \end{align}
     where $w^{\nu} = {\exp_{\zbar}^{-1}(z^\nu)}$ and we used the fact that $\| w^\nu \|$ has the same order with $\| z^\nu -\overline z\|$ implied by the smoothness of the exponential map.
    Since $z^\nu \in \tMpq$, the stratum-restricted local error bound \eqref{eq:res-local-EB} and the Lipschitz property of $F$ on $\Mpq$ imply that $\|z^\nu - \overline z\|/\|F(z^\nu)\|$ is bounded below away from zero and bounded above. Taking a subsequence if necessary, we can assume that
    \begin{align}\label{eq:vc}
        \frac{w^\nu}{\|w^\nu\|}\to v \in \Tcal_{\overline z}\tMpq \quad \text{ and } \quad \frac{\|w^\nu\|}{\|F( z^\nu)\|}\to c_4,
    \end{align}
    where $\|v \|=1$ and $c_4>0$.
    Putting \eqref{eq:vc} into \eqref{eq:F/F-norm}, we have
    \begin{align}\label{eq:lim-F/norm}
          \lim_{\nu\to\infty}\frac{F( z^\nu)}{\|F( z^\nu)\|}=c_4\dd F_{\overline z} (v).
    \end{align}
    For the second part, without loss of generality, we assume that $\Delta^{\nu}\to\Delta^\infty$. Since $\|\Delta^\nu \| = 1$, we know $0\neq\Delta^\infty\in\mN_{G(\overline z)}\Mpq$. Then
    \begin{equation}\label{eq:lim-J/norm}
        \lim_{\nu\to\infty} J_{ z^\nu}(\Delta^\nu)=\lim_{\nu\to\infty} \begin{bmatrix} \nabla g(x^\nu)\Delta^\nu \\\Pi_{\Sbb^n_+} (\Delta^\nu) \end{bmatrix}=\begin{bmatrix} \nabla g(\overline x)\Delta^\infty \\\Pi_{\Sbb^n_+} (\Delta^\infty) \end{bmatrix}=J_{\overline z}(\Delta^{\infty}).
    \end{equation}
     \par\nobreak Combining \eqref{eq:F-J=0}, \eqref{eq:lim-F/norm} and \eqref{eq:lim-J/norm}, {and replacing $v$ by $\iota v$ if necessary,} set $s:= 1/(c_4\|J_{\overline z}(\Delta^\infty)\|)>0$ so that
    \begin{align} \nn
        \dd F_{\overline z}(v)=sJ_{\overline z}(\Delta^{\infty}),
    \end{align}
    i.e.,
    \begin{align}\label{eq:dFv-JDelta-eq}
        \left\{
        \begin{aligned}
        &(\nabla^2_{xx} L^*-\nabla g(\overline x) \nabla g(\overline x)^*)v_x+ \nabla g(\overline x)H -s\nabla g(\overline x)\Delta^\infty = 0,  \\
        &-\nabla g(\overline x)^* v_x +  \xi_{G(\overline z)}  H-s\Pi_{\Sbb^n_+} (\Delta^\infty)=0 ,\\
        \end{aligned}
        \right.
    \end{align}
    where $v = (v_x, H) \in \Tcal_{\overline x} \Xbb \times \Tcal_{G(\overline z)} \Mpq \cong \Tcal_{\overline z}\tMpq$ and we have made use of the characterization of $\dd F_{\overline z}(v)$ given in~\eqref{eq:dF_z(v)}.  The second equation in~\eqref{eq:dFv-JDelta-eq} implies that
    \begin{align}\label{eq:gvx-xiH-2}
        & (-\nabla g(\overline x)^* v_x +  \xi_{G(\overline z)}  H)_{\alpha\gamma} = 0 \quad \mbox{and}\quad {-v_x \in \mathcal C(\overline z)},
    \end{align}
    {where the critical cone $\mathcal C(\overline z)$ is defined in~\eqref{eq:def-C-nonKKT}.}
    Then, we have
    \begin{align}\nonumber
        0=&\la v_x,(\nabla^2_{xx} L^*-\nabla g(\overline x) \nabla g(\overline x)^*)v_x+ \nabla g(\overline x)H -s\nabla g(\overline x)\Delta^\infty\ra \\\nonumber
        =&\la v_x,\nabla^2_{xx} L^* v_x\ra +\la \nabla g(\overline x)^*v_x,H-s\Delta^\infty-\nabla g(\overline x)^*v_x\ra \\\nonumber
        =&\la v_x,\nabla^2_{xx} L^* v_x\ra +\la  \xi_{G(\overline z)}  H-s\Pi_{\Sbb^n_+} (\Delta^\infty),H-s\Delta^\infty- \xi_{G(\overline z)}  H+s\Pi_{\Sbb^n_+} (\Delta^\infty)\ra \\\nonumber
        =&\la v_x,\nabla^2_{xx} L^* v_x\ra +\la  \xi_{G(\overline z)}  H,H- \xi_{G(\overline z)}  H\ra -\la  s\Pi_{\Sbb^n_+} (\Delta^\infty),-s\Delta^\infty+s\Pi_{\Sbb^n_+} (\Delta^\infty)\ra \\ \nonumber
        =&\la v_x,\nabla^2_{xx} L^* v_x\ra +\la  \xi_{G(\overline z)}  H,H- \xi_{G(\overline z)}  H\ra , \\ \label{eq:vx-first-2}
        =&\mQ_{\overline z}(v_x),
    \end{align}
    where the fourth equation follows from $H \in \mathcal{T}_{G(\overline z)} \mathcal{M}_{p,q}$, $\Delta^\infty \in \mathcal{N}_{G(\overline z)}\mathcal{M}_{p,q}$ and definition of $\xi_{G(z)}$ in \eqref{eq:xi_A}, the fifth equation follows from the Moreau-Yosida decomposition of $K$, and the last equation follows from Lemma \ref{lemma:identity-Q(v)} as we already have $(-\nabla g(\overline x)^* v_x +  \xi_{G(\overline z)}  H)_{\alpha\gamma} = 0$ from \eqref{eq:gvx-xiH-2}.

     {By~\eqref{eq:vx-first-2} and~\eqref{eq:gvx-xiH-2}, $-v_x\in\mathcal C(\overline z)$ and $\mQ_{\overline z}(-v_x)=\mQ_{\overline z}(v_x)=0$, so the SOSC yields $v_x=0$.} Therefore, \eqref{eq:dFv-JDelta-eq} reduced to
    \begin{align}\label{eq:dFv-JDelta-eq-reduced}
        \left\{
        \begin{aligned}
        &\nabla g(\overline x)H -s\nabla g(\overline x)\Delta^\infty = 0,  \\
        &\xi_{G(\overline z)}  H-s\Pi_{\Sbb^n_+} (\Delta^\infty)=0 .\\
        \end{aligned}
        \right.
    \end{align}
    As $H \in \Tcal_{G(\overline z)} \Mpq$ and $\Delta^\infty \in \mN_{G(\overline z)}\Mpq$, we must have $\Delta^\infty\in \Sbb^n_-$, $\widetilde{H}_{\alpha\alpha}=0$ and $\widetilde{H}_{\alpha\beta}=0$. Then, for the following special choice of element in ${\Sbb^n}$
    $$
        Y = P\begin{bmatrix}
            0 & 0 & 0\\
            0 & -I_{|\beta|} & 0\\
            0 & 0 & 0
        \end{bmatrix}P^{\top},
    $$
    by the SRCQ, there exist $x \in \Xbb$ and $Y_0\in \Bigl\{\overline P B\overline P^\top\in\mathbb{S}^n \,\Big|\,
    B_{\beta\beta}\succeq 0,\ B_{\beta\gamma}=0,\ B_{\gamma\gamma}=0\Bigr\}$ such that
    $$
        \nabla g(\overline x)^*x + Y_0 = Y.
    $$
    Noticing that $0 \ne \Delta^\infty\in\mN_{G(\overline z)}\Mpq\cap \Sbb^n_-$, we have
    \begin{align} \nn
        0 > &\langle Y, H-s\Delta^\infty \rangle \\ \nn
        = &\langle \nabla g(\overline x)^*x, H-s\Delta^\infty \rangle + \langle Y_0,H-s\Delta^\infty \rangle \\ \nn
        = &\langle x, \nabla g(\overline x)(H-s\Delta^\infty) \rangle + \langle Y_0,H \rangle-s\langle Y_0,\Delta^\infty \rangle \ge 0,
    \end{align}
    where the first inequality follows from $H \in \mathcal{T}_{G(\overline z)} \mathcal{M}_{p,q}$, $\Delta^\infty \in \Sbb^n_-$, and the definition of $Y$, while the last inequality further follows from \eqref{eq:dFv-JDelta-eq-reduced} and $Y_0\in \Bigl\{\overline P B\overline P^\top\in\mathbb{S}^n \,\Big|\,
    B_{\beta\beta}\succeq 0,\ B_{\beta\gamma}=0,\ B_{\gamma\gamma}=0\Bigr\}$. All in all, the contradiction shows that the statement of this proposition holds, and the proof is then completed.
    \ep
\end{proof}

The following proposition captures the second order behavior of $\varphi$ along the LM step under the retraction curve, where the linear term is scaled due to the implicit cancellation inherent in the Gauss--Newton system. It is essential for our local analysis, as it guarantees that the line search will eventually take a unit step.

\begin{appprop}\label{proposition:str-Newton-Taylor}
    Let $\zbar = (\xbar,\ybar)$ be a KKT pair of~\eqref{prog:SDP} with an IED $(\alpha,\beta,\gamma,p,q,\Pbar,\overline \lambda)$ of $G(\zbar)$. If the W-SOC and the SRCQ holds at $\zbar$, then, for $z \in \tMpq$ sufficiently close to $\zbar$ and the associated LM direction $v = v^{\mathrm{LM}}(z)$ given by~\eqref{eq:SLM-direction}, we have
    \begin{equation}\label{eq:str-Newton-Taylor}
        \varphi\bigl(R_z(v)\bigr) = \varphi(z) + \frac{1}{2} \varphi'(z; v) + o\bigl(\| v \|^2\bigr).
    \end{equation}
\end{appprop}
\begin{proof}
    By Corollary~\ref{coro:WR-uniform-bounded}, since the SRCQ implies the W-SRCQ, we know that when $z \in \tMpq$ is sufficiently close to $\overline{z}$, the norm $\left\| \left((dF_{z})^* dF_{z} \right)^{-1} \right\|$ is uniformly bounded, and thus $\left\| \left( (dF_{z})^* dF_{z} + \|F(z)\|^2 I \right)^{-1} \right\|$, is also uniformly bounded. As a result, $v=\vLM(z) \rightarrow 0$ when $z \rightarrow \overline{z}$.

At $z=\bar z$, we have $F(z)=0$ and $v=0$, so the equality holds.
For $z\ne\bar z$, let $w=\exp_z^{-1}(\bar z)$. Lemma~\ref{lemma:retr~exp} and the $C^1$ regularity of $F|_{\tMpq}$ give
\[
    w=\bar z-z+O(\|z-\bar z\|^2),
    \qquad F(z)+dF_z w=o(\|z-\bar z\|).
\]
The LM equation yields
\[
    v-w=-(dF_z^*dF_z+\mu(z)I)^{-1}
    \bigl[dF_z^*(F(z)+dF_z w)+\mu(z)w\bigr].
\]
Since $dF_z$ is locally bounded and $\mu(z)=\|F(z)\|^2=O(\|z-\bar z\|^2)$, it follows that $v-w=o(\|z-\bar z\|)$. Thus $\|v\|$ and $\|z-\bar z\|$ are comparable, and $\|F(z)\|=O(\|v\|)$.

The $C^1$ regularity of $F$ on the stratum and of the local retraction gives
$F(R_z(v))=F(z)+dF_zv+o(\|v\|)$ as $z\to\bar z$. The remainder is uniform for base points in a sufficiently small compact neighborhood in the stratum, by continuity of the differential in local coordinates. Squaring this expansion and using $\varphi=\|F\|^2/2$ gives
\[
    \varphi(R_z(v))
    =\varphi(z)+\varphi'(z;v)+\tfrac12\|dF_zv\|^2
    +o(\|v\|^2).
\]
Taking the inner product of the LM equation with $v$ yields
$\|dF_zv\|^2=-\varphi'(z;v)-\mu(z)\|v\|^2$.
Since $\mu(z)\to0$, substitution proves the assertion.

\ep
\end{proof}

 {The next proposition gives the one-step convergence rates.}

\begin{appprop}\label{prop:quad-rate}
     \par\nobreak Let $\overline{z}$ be a KKT pair of \eqref{prog:SDP} with an IED $(\alpha,\beta,\gamma,p,q,\Pbar,\overline \lambda)$ of $G(\zbar)$. If the W-SOC and the W-SRCQ holds at $\zbar$, then for $z\in\tMpq$ sufficiently close to $\zbar$, we have
    \[
        \bigl\| R_{z}(v) - \overline{z} \bigr\| = {o\bigl(\|z - \overline{z}\|\bigr)},
    \]
    where $v = \vLM(z)$ is the associated LM direction given by~\eqref{eq:SLM-direction} and $R_z$ is the retraction of $\tMpq$ at $z$ given in~\eqref{retraction_tM}. {If, in addition, $f$ and $g$ are $LC^2$ on a neighborhood of $\overline x$, then $\|R_z(v)-\overline z\|=O(\|z-\overline z\|^2)$.}
\end{appprop}
\begin{proof}
By Lemma~\ref{lemma:retr~exp}, the retraction $R_{z}$ admits a second order expansion. Thus, for all $z$ sufficiently close to $\overline{z}$,
\[
\bigl\| R_{z}(v) - \overline{z} \bigr\|
= \left\| z - \overline{z} - \bigl((dF_{z})^* dF_{z} + {\mu(z)} I\bigr)^{-1}(dF_{z})^* F(z) \right\| + O(\|v\|^2),
\]
where $\mu(z) = \|F(z)\|^2$ is the regularization parameter used in the definition of the Levenberg--Marquardt step $v = v_{\mathrm{LM}}(z)$.
Let $w(z) := \exp^{-1}_{z}(\overline{z}) \in T_{z}\widetilde{\mathcal{M}}_{p,q}$. By the smoothness of the exponential map $\exp: \Tcal\tMpq \to \tMpq$, we have
\[
w(z) = - (z - \overline{z}) + O(\|z - \overline{z}\|^2),
\]
and hence $z - \overline{z} = -w(z) + O(\|z - \overline{z}\|^2)$. Furthermore, since $\| v \| = O(\| z -\overline z\|)$ (implied by the boundedness of the operator inverse and smoothness of $F$), we have
\begin{align*}
&\bigl\| R_{z}(v) - \overline{z} \bigr\| \\
=& \left\| \bigl((dF_{z})^* dF_{z} + \mu(z) I\bigr)^{-1} \Bigl[ (dF_{z})^*\bigl(dF_{z} w(z) + F(z)\bigr) + \mu(z) w(z) \Bigr] \right\| + O(\|z - \overline{z}\|^2) \\
\leq& \left\| \bigl((dF_{z})^* dF_{z} + \mu(z) I\bigr)^{-1} \right\|   \Bigl( \|(dF_{z})^*\bigl(dF_{z} w(z) + F(z)\bigr)\| + \|\mu(z) w(z)\| \Bigr) + O(\|z - \overline{z}\|^2).
\end{align*}
We now bound each component:
\begin{enumerate}[label=(\roman*)]
    \item By Corollary~\ref{coro:WR-uniform-bounded}, the W-SOC and the W-SRCQ at $\overline{z}$ imply the uniform boundedness
    \[
        \left\| \bigl((dF_{z})^* dF_{z} + \mu(z) I\bigr)^{-1} \right\| = O(1).
    \]
    \item  Because $w(z) = \exp^{-1}_{z}(\overline{z})$, the first order Taylor expansion of $F$ along the geodesic from $z$ to $\overline{z}$ gives
    \[
        F(z) + dF_{z} w(z) = {o(\|z - \overline{z}\|)},
    \]
    {Here we use that $F|_{\tMpq}$ is $C^1$. Since $dF_z$ is locally bounded, $\|(dF_z)^*(dF_z w(z)+F(z))\|=o(\|z-\overline z\|)$. Under the additional $LC^2$ assumption, $dF$ is locally Lipschitz and both remainders are $O(\|z-\overline z\|^2)$.}
    \item By definition, $\mu(z) = \|F(z)\|^2 = O(\|z - \overline{z}\|^2)$, and $\|w(z)\| = O(\|z - \overline{z}\|)$. Thus,
    \[
        \|\mu(z) w(z)\| = O(\|z - \overline{z}\|^3).
    \]
\end{enumerate}
 Combining estimates (i)–(iii), we finally conclude that
\[
    \bigl\| R_{z}(v) - \overline{z} \bigr\| = {o(\|z - \overline{z}\|)},
\]
{Under the additional $LC^2$ assumption, the same estimates give $\|R_z(v)-\overline z\|=O(\|z-\overline z\|^2)$.} This completes the proof.
\ep
\end{proof}

\begin{theoremrep}{thm:local-quad-rate}
     Suppose $z^\infty$ is an accumulation point of the sequence generated by Algorithm~\ref{alg:SGN}. If the correction threshold bound satisfies $0 < \delta < \delta(z^\infty)$ and $z^\infty$ is a D-stationary point of $\varphi$ at which the {SOSC} and the SRCQ hold simultaneously, then the whole sequence converges {superlinearly} to $z^\infty$ in the sense that
    \begin{equation*}
        \|z^{k+1} - z^\infty\| \leq {o(\|z^{k} - z^\infty\|)}.
    \end{equation*}
    {If $f$ and $g$ are additionally of class $LC^2$ near $x^\infty$, then the convergence is quadratic.} Moreover, for all sufficiently large $k$, the sequence $\{z^k\}$ lies in the active stratum containing $z^\infty$.
\end{theoremrep}
\begin{proof}
    By the assumed D-stationarity of $\varphi$ at $z^\infty$ and the SOSC and SRCQ, Theorem~\ref{thm:dstat-implies-kkt} implies that $z^\infty$ is a KKT pair. Let $\{z^k\}_{k\in\Theta}$ denote the subsequence converging to $z^\infty$. Then, by the continuity of eigenvalues and $0 < \delta < \delta(z^\infty)$, the correction point $\zhatk$ generated by $z^k$ (see \eqref{eq:hat-z}) lies on the stratum containing $z^\infty$ for large enough $k\in\Theta$. By Lemma \ref{lemma:corr-almost-proj}, as $z^k \to_{k \in \Theta} z^\infty$, we know $\zhatk \to_{k \in \Theta} z^\infty$.
    Now, consider any $z$ that lies on the stratum containing $z^\infty$. By Proposition \ref{proposition:str-Newton-Taylor}, once $z$ sufficiently close to $z^\infty$
    \begin{align} \nn
        \varphi(R_{z}(v))-\varphi(z)=\frac{1}{2}\varphi'(z,v)+o(\|v\|^2),
    \end{align}
    where $v=\vLM(z)$ is given by \eqref{eq:SLM-direction}. Using the fact that $\varphi'(z,v)$ is of the same order as $\|v\|^2$, which is implied by Lemma \ref{Lemma:ww-dir-of-phi-order}, and that $\eta\in(0,1/2)$, we conclude that the line search of the stratum LM step starting from $z$ takes the unit step.  Combining this with Proposition \ref{prop:quad-rate}, which shows
    \begin{equation}\label{eq:quad-rate-R-sequence}
        \|R_{z}(v)-z^\infty\|\le {o(\|z-z^\infty\|)},
    \end{equation}
     we obtain
    \begin{equation}\label{eq:quad-rate-R-varphi}
        \varphi(R_{z}(v))\le O(\|R_{z}(v)-z^\infty\|^2)\le {o(\|z-z^\infty\|^2)} \le {o(\varphi(z))},
    \end{equation}
    where the first inequality follows from the Lipschitz continuity of $F$, and the last inequality follows from the stratum-restricted local error bound~\eqref{eq:res-local-EB}.
    Now, for $k\in \Theta$ large enough, denote by $z^{k+1}_1$ and $z^{k+1}_2$ the normal steps from $\zhatk$ (see Algorithm~\ref{alg:SLMN}). Then, by Proposition~\ref{prop:esti-cor}, we have
    \begin{align} \nn
        \varphi(\zhatk) \leq c_3 \min\{\varphi(z^k), \varphi(z^{k+1}_1), \varphi(z^{k+1}_2)\},
    \end{align}
    which, together with~\eqref{eq:quad-rate-R-varphi} and the fact that $\varphi(\zhatk) \to 0$, implies
    \begin{align} \nn
        \varphi(R_{\zhatk}(\widehat{v}^k)) < \min\{\varphi(z^k), \varphi(z^{k+1}_1), \varphi(z^{k+1}_2)\},
    \end{align}
    where $\vhatk=\vLM(\zhatk)$. Finally, by the logic of Algorithm~\ref{alg:SGN}, we know that the next iterate must be $z^{k+1} = R_{\zhatk}(\widehat{v}^k)$. Moreover, we have $\widehat{z}^{k+1} = z^{k+1}$ and, by~\eqref{eq:quad-rate-R-sequence}, $\|z^{k+1} - z^\infty\| < \|z^{k} - z^\infty\|$.
     By a recursive argument, we can conclude that, for $k$ large enough, $z^k$ lies on the stratum containing $z^\infty$, $z^{k+1} = R_{{z}^k}(\vLM(z^k))$, and $z^k\to z^\infty$ {at a superlinear rate, and at a quadratic rate if $f$ and $g$ are additionally of class $LC^2$ near $x^\infty$, by Proposition~\ref{prop:quad-rate}}.
    \ep
\end{proof}

\end{document}